\documentclass[12pt]{article}
\usepackage[french,english]{babel}
\usepackage[utf8]{inputenc}
\usepackage{color}
\usepackage{amsmath}
\usepackage{amscd}
\usepackage{amstext}
\usepackage{amsbsy}
\usepackage{amsopn}
\usepackage{amsthm}
\usepackage{amsxtra}
\usepackage{amssymb}
\usepackage{hyperref}
\usepackage[all]{xy}
\usepackage[active]{srcltx}
\let\svthefootnote\thefootnote
\newcommand\blankfootnote[1]{%
\let\thefootnote\relax\footnotetext{#1}%
\let\thefootnote\svthefootnote}
\newcommand{\be}{\begin{enumerate}}
\newcommand{\ee}{\end{enumerate}}
\newtheorem{cor}[subsubsection]{Corollaire}

\newtheorem{sousprop}[paragraph]{Proposition}
\newtheorem{sousth}[paragraph]{Th\'eor\`eme}
\newtheorem{prop}[subsubsection]{Proposition}
\newtheorem{defi}[subsubsection]{D\'efinition}
\newtheorem{sousdef}[paragraph]{D\'efinition}

\newtheorem{lem}[subsubsection]{Lemme}
\newtheorem{souslem}[paragraph]{Lemme}
\newtheorem{thm}[subsubsection]{Th\'eor\`eme}
\newcommand{\inft}{{\scriptstyle\infty}}

\newcommand{\what}{\widehat}
\newcommand{\ot}{\otimes}

\newcommand{\rig}{\rightarrow}
\newcommand{\trig}{\twoheadrightarrow}
\newcommand{\thrig}{\twoheadrightarrow}

\newcommand{\lorig}{\longrightarrow}
\newcommand{\hrig}{\hookrightarrow}
\newcommand{\sta}{\stackrel}

\newcommand{\lrig}{\leftarrow}
\newcommand{\gr}{{\rm gr}}
\newcommand{\pg}{_{\bullet}}

\newcommand{\crofrac}[2]{\genfrac{\langle}{\rangle}{0pt}{}{#1}{#2}}

\newcommand{\parfrac}[2]{\genfrac{(}{)}{0pt}{}{#1}{#2}}

\newcommand{\acfrac}[2]{\genfrac{\{}{\}}{0pt}{}{#1}{#2}}

\newcommand{\der}{\partial}
\newcommand{\la}{\langle}
\newcommand{\ra}{\rangle}

\newcommand{\Qr}{{\bf Q}}
\newcommand{\Ne}{{\bf N}}
\newcommand{\Ze}{{\bf Z}}

\newcommand{\Rr}{{\bf R}}
\newcommand{\vp}{{\rm v}_{p}}

\newcommand{\Zp}{{\bf Z}_p}

\newcommand{\Sg}{{\bf S}}

\newcommand{\varep}{\varepsilon}
\newcommand{\lam}{\lambda}

\newcommand{\Ga}{\Gamma}
\newcommand{\ga}{\gamma}

\renewcommand{\AA}{{\cal A}}

\newcommand{\DD}{{\cal D}}
\newcommand{\EE}{{\cal E}}
\newcommand{\FF}{{\cal F}}
\newcommand{\GG}{{\cal G}}
\newcommand{\II}{{\cal I}}
\newcommand{\HH}{{\cal H}}

\newcommand{\LL}{{\cal L}}

\newcommand{\OO}{{\cal O}}
\newcommand{\PP}{{\cal P}}
\newcommand{\QQ}{{\cal Q}}

\renewcommand{\SS}{{\cal S}}
\newcommand{\TT}{{\cal T}}
\newcommand{\UU}{{\cal U}}

\newcommand{\XX}{{\cal X}}

\newcommand{\ui}{\underline{i}}

\newcommand{\uk}{\underline{k}}
\newcommand{\ul}{\underline{l}}

\newcommand{\ut}{\underline{t}}
\newcommand{\uu}{\underline{u}}

\newcommand{\ur}{\underline{r}}

\newcommand{\uder}{\underline{\der}}
\newcommand{\ual}{\underline{\alpha}}

\newcommand{\uxi}{\underline{\xi}}
\newcommand{\ueta}{\underline{\eta}}

\newcommand{\utau}{\underline{\tau}}
\newcommand{\ovu}{\overline{u}}
\newcommand{\ovA}{\overline{A}}
\newcommand{\ovB}{\overline{B}}

\newcommand{\ovot}{\overline{\otimes}}

\newcommand{\Dm}{\DD^{(m)}}
\newcommand{\Dcm}{\what{\DD}^{(m)}}
\newcommand{\Ddag}{\DD^{\dagger}}
\begin{document}
%
\title{Algèbres de distributions et $\DD$-modules arithmétiques}
\author{Christine~Huyghe et Tobias~Schmidt}
\date{20 décembre 2016}
\maketitle
\selectlanguage{french}
\selectlanguage{english}
\begin{abstract}
Let $p$ be a prime number, $V$ a complete discrete valuation ring of unequal caracteristics $(0,p)$, $G$ a smooth affine
algebraic group over $Spec \,V$. Using partial divided powers techniques of Berthelot, we construct arithmetic
distribution algebras, with level $m$, generalizing the classical construction of the distribution algebra.
We also construct the weak completion
of the classical distribution algebra over a finite extension $K$ of $\Qr_p$ . 
We then show that these distribution algebras can be identified with
invariant arithmetic differential operators over $G$, and prove a coherence result when the ramification index of $K$ is $<p-1$.
\end{abstract}
\selectlanguage{french}
\begin{abstract} Soient $p$ un nombre premier, $V$ un anneau de valuation discrète complet d'inégales caractéristiques
$(0,p)$, $G$ un groupe algébrique affine lisse sur $Spec \,V$. En utilisant les techniques de puissances divisées de
niveau $m$ de Berthelot, nous construisons dans ce cadre des algèbres de
distributions arithmétiques, avec des niveaux $m$, généralisant la construction classique. Nous construisons
aussi la complétion faible de l'algèbre des distributions classique sur une extension finie $K$ de $\Qr_p$.
Nous montrons alors que ces algèbres de distribution s'identifient aux opérateurs différentiels arithmétiques invariants sur $G$
et donnons un théorème de cohérence si l'indice de ramification de $K$ est  $<p-1$.
\end{abstract}
\blankfootnote{MSC classification 2010 : 20G05, 20G25, 13N10, 16S32}
\tableofcontents
\section{Introduction}
Soient $p$ un nombre premier, $V$ une $\Ze_{(p)}$-algèbre noetherienne, $S=Spec \,V$, et
 $G$ un schéma en groupes affine et lisse sur $S$. On suppose dans la suite que $S$ est connexe.
 C'est un résultat classique de
Demazure-Gabriel \cite{Demazure_gr_alg}
que l'algèbre de distributions de $G$, $Dist(G)$, s'identifie
aux opérateurs différentiels globaux sur $G$ qui sont $G$-invariants.
Rappelons que l'algèbre $Dist(G)$ joue un rôle central
dans la théorie des représentations de $G$.
 Si, en outre, $S$ est un corps de caractéristique
$0$, l'algèbre enveloppante de l'algèbre de Lie de $G$ s'identifie à l'algèbre de distributions de
$G$. Pour tout niveau $m$,
 Berthelot a construit des faisceaux d'opérateurs différentiels
 arithmétiques $\Dm_G$. Dans cet article, nous construisons des algèbres de distributions sur
 $G$ canoniquement isomorphes aux opérateurs différentiels arithmétiques
 de niveau $m$ invariants sous l'action de $G$. Par analogie à la caractéristique $0$, on peut donc
 voir ces algèbres de distributions comme des analogues arithmétiques de niveau $m$ de l'algèbre
 enveloppante de $Lie(G)$ sur une base arithmétique $S$.
 Nous construisons aussi des algèbres de distributions $p$-adiquement complètes de niveau
 $m$ pour un schéma formel en groupes, $\mathcal G$,
 sur le spectre formel d'un anneau de valuation discrète complet d'inégales caractéristiques
 $(0,p)$. Dans ce cas, on peut aussi passer à la limite sur $m$ et construire une algèbre de
 distributions faiblement complète correspondant aux opérateurs différentiels arithmétiques
 sur $\GG$ qui sont $\GG$-invariants.

\vskip5pt
Une motivation importante de ces constructions réside dans le cas
du complété formel d'un schéma
en groupes réductif connexe déployé sur $V$.
Dans ce contexte, nous utilisons
ces algèbres (\cite{HS-loc}) pour montrer une version arithmétique, i.e.
une version pour les $\DD$-modules arithmétiques de Berthelot, du théorème de localisation classique de
Beilinson-Bernstein \cite{BeBe}.

\vskip5pt
Classiquement, l'algèbre $Dist(G)$ est définie en dualisant les fonctions sur les voisinages infinitésimaux
 de la section unité de $G$. L'idée de notre construction de l'algèbre des distributions de niveau $m$
est de remplacer ces voisinages infinitésimaux
par les $m$-PD voisinages infinitésimaux de la section unité de $G$, au sens des puissances divisées partielles
de niveau $m$ de Berthelot.

\vskip5pt
Indiquons maintenant brièvement la structure de notre article.

Dans la section~\ref{sec_rapdiff}, nous rappelons la théorie des puissances divisées partielles et la construction des opérateurs différentiels arithmétiques. Nous expliquons le formalisme des opérateurs différentiels twistés par un faisceau inversible dans le cadre des
$\DD$-modules arithmétiques de Berthelot et donnons quelques propriétés de base. Ces faisceaux twistés jouent
un rôle important en théorie
des représentations et c'est ce qui motive la construction cristalline faite ici, valable sur une base arithmétique.

Dans la section~\ref{desc_actions_G}, nous rappelons quelques généralités sur les faisceaux équivariants sur les schémas munis
d'une action du schéma en groupes $G$, qui sont éparpillées dans la littérature. Nous appliquons ces considérations
aux versions arithmétiques des faisceaux de parties principales, aux algèbres symétriques et aux opérateurs différentiels dans le cas d'un schéma en groupes $G$
lisse sur $S$. Nous montrons en particulier que ces faisceaux sont $G$-équivariants.

\vskip5pt

Dans la section~\ref{sect-distribution-algebras}, nous étudions systématiquement les algèbres de distributions
$D^{(m)}(G)$, ainsi que les modules
sur ces algèbres. 
Nous montrons ainsi que la construction des algèbres $D^{(m)}(G)$ est fonctorielle
en $G$. D'autre part, ces algèbres sont filtrées par la filtration par l'ordre et le gradué associé à cette filtration
est commutatif (\ref{prop-commutativite}) et
s'identifie à l'algèbre symétrique de niveau $m$ de $Lie(G)$, l'algèbre de Lie du groupe $G$. On dispose en particulier
d'un théorème du type Poincaré-Birkhoff-Witt pour $D^{(m)}(G)$ en termes d'une $V$-base de $Lie(G)$.
 De plus, les anneaux $D^{(m)}(G)$ sont noetheriens à gauche et à droite. Nous donnons une description explicite de
ces algèbres dans le cas du groupe additif et du groupe multiplicatif, et introduisons la notion de
PD stratification de niveau $m$ en ~\ref{strat_m}, permettant de décrire plus facilement les
$D^{(m)}(G)$-modules.

 En outre, si $X$ est un $S$-schéma muni d'une
action de $G$ (à droite), nous donnons en~\ref{liens-faisdiffG} et ~\ref{op_diff_inv}
une interprétation géométrique des algèbres $D^{(m)}(G)$
en montrant
qu'il existe un homomorphisme d'anneaux de $D^{(m)}(G)$ vers l'algèbre
des sections globales sur $X$ du faisceau des opérateurs différentiels de niveau $m$ sur $X$ : $\Dm_X$. 
De plus, et c'est un résultat très important pour les applications,
nous montrons que cet homomorphisme d'anneaux est surjectif si $X$ est un espace homogène.
Enfin, dans le cas où $X=G$, ces constructions permettent d'identifier $D^{(m)}(G)$ avec
les opérateurs différentiels sur $G$, $G$-invariants. Tous ces résultats sont aussi
valables en caractéristique finie, ou si $V$ possède de la torsion. Pour résumer on dispose du \\
\noindent{\bf Théorème} ~\ref{opp_diff_inv}. {\it Soit $\DD^{(m)}_{G}$ le faisceau des opérateurs différentiels
arithmétiques de Berthelot sur $G$, $\Ga(G,\DD^{(m)}_{G})^G$ les sections globales $G$-invariantes, 
alors il existe un isomorphisme d'algèbres canonique $\Ga(G,\DD^{(m)}_{G})^G \simeq D^{(m)}(G).$
}

Au début de la section ~\ref{alg_dist_comp}, nous établissons des analogues des résultats précédents dans le cas où
$V$ est un anneau de valuation discrète complet d'inégales caractéristiques $(0,p)$ et
pour un schéma formel en groupes (cette notion étant définie en ~\ref{distrib_cris}). Soit  $\GG$ un $\SS$-schéma formel
en groupes, alors on peut construire des algèbres de distributions complétées de niveau $m$ : $\what{D}^{(m)}(\GG)$. En
passant à la limite inductive sur $m$, on en déduit une algèbre $D^{\dagger}(\cal G)_{\Qr}$.
Dans l'énoncé qui suit $\Ga(\GG,\Dcm_{\GG})^{\GG}$ désigne les sections globales $\GG$-invariantes de $\Dcm_{\GG}$. \\
\noindent {\bf Théorème.} \ref{opp_diffc_inv} {\it 
Il existe des isomorphismes canoniques d'algèbres filtrées,
 entre les algèbres $\what{D}^{(m)}(\cal G)$ et $\Ga(\GG,\Dcm_{\GG})^{\GG}$
(resp. entre les algèbres $D^{\dagger}(\cal G)_{\Qr}$ et $\Ga(\GG,\Ddag_{\GG,\Qr})^{\GG}$).}

En ~\ref{dist_an_rig} et ~\ref{rep_an_rig}, nous relions ces constructions à la théorie des représentations $p$-adiques dans le cas o\`u
 $V$ est l'anneau des entiers d'un corps local $p$-adique. Nous introduisons un certain groupe
analytique rigide $G^\circ$ dont le groupe des points rationnels $G^\circ(K)$ est égal au 'premier groupe de congruence'
de $G(V)$. Son algèbre des distributions analytiques rigides $D^{an}(G^\circ)$ est le dual continu de
l'espace des sections globales de $G^\circ$, muni du produit de convolution. Nous montrons que
l'inclusion de $Lie(G^\circ)$ dans $D^{an}(G^\circ)$ s'étend naturellement en un isomorphisme entre
$D^{\dagger}(\cal G)_{\Qr}$ et $D^{an}(G^\circ)$, ce qui généralise
\cite[2.3.3.]{patel_schmidt_strauch_dist_algebra_GL2} dans le cas $V=\Ze_p$ et ${\rm GL}(2)$. 
Un résultat d'Emerton \cite{Emerton} implique alors que l'anneau $D^{\dagger}(\cal G)_{\Qr}$ est coh\'erent, si l'indice
de ramification $e$ de $K/\Qr_p$ satisfait $e<p-1$. Précisément, nous obtenons les résultats suivants\\
\noindent {\bf Proposition.} ~\ref{prop-isomorph}  {\it L'application
 $U(Lie(G^\circ))\rightarrow D^{an}(G^\circ)$ s'étend en un isomorphisme d'anneaux topologiques
$ D^{\dagger}(\GG)_\Qr\stackrel{\simeq}{\longrightarrow}D^{an}(G^\circ).$} \\
\noindent {\bf Corollaire.} ~\ref{thm-coherence}  {\it Soit $e$ l'indice de ramification de $K$ sur $\Qr_p$. Si $e<p-1$,
 alors l'anneau $D^{\dagger}(\GG)_\Qr$ est cohérent.}\\
Nous introduisons ensuite les catégories des 
représentations analytiques rigides des groupes de Lie $G^\circ(K)$ et $G(V)$ ainsi que leurs sous-catégories de
représentations admissibles. Par construction, ces dernières sont anti-équivalentes
aux catégories de modules de présentation finie sur $D^{an}(G^\circ)$.

Les algèbres de distributions arithm\'etiques avaient été définies ad hoc par Kaneda et Ye
(\cite{kaneda-ye-equivariant_loc_symp_4}), sans utiliser un formalisme conceptuel et fonctoriel comme ici, et qui ne les ont pas étudiées systématiquement.
Remarquons aussi que, dans le cas ${\rm GL}(2)$ et $V=\Ze_p$, ces algèbres sont définies ad hoc par Patel, Schmidt et Strauch \cite{patel_schmidt_strauch_dist_algebra_GL2} pour calculer des sections globales des opérateurs log-différentiels
arithmétiques sur certains modèles formels semistables de la droite projective.

Notons aussi que
Le Stum et Quiros ont introduit les $m$-PD-enveloppes de la section unité d'un groupe $G$ dans
\cite{stum_quir_poinc_high}.

\vskip5pt

Le premier auteur remercie Michel Gros, et Adriano Marmora pour leurs encouragements à rédiger ce travail, ainsi
que King Fai Lai pour l'avoir initiée à ces problèmes. Nous remercions d'autre part
 Pierre Berthelot pour ses réponses à nos questions.

Le second auteur remercie Deepam Patel et Matthias Strauch pour quelques conversations intéressantes sur certains points de ce travail.

Le second auteur a accompli une partie de ce travail alors qu'il était lauréat d'une bourse Heisenberg attribuée
par la Deutsche Forschungsgemeinschaft. Il tient à remercier cette institution pour son soutien.
\vskip5pt
{\it Notations:} \; $V$ désigne toujours une $\Ze_{(p)}$-algèbre noetherienne et $S=Spec\, V$ est connexe. Si, plus
particulièrement, $V$ est un anneau de valuation discrète complet d'inégales caractéristiques $(0,p)$ (noté AVDC dans la
suite), nous noterons $\pi$ une uniformisante de $V$, $K=Frac(V)$ et $k$ le corps résiduel de $V$. Si $V$ est un AVDC,
on écrit $\SS=Spf\, V$. Un $\SS$-schéma formel est un schéma formel localement noethérien $\XX$ sur $\SS$, tel que
$\pi\OO_{\XX}$ soit un idéal de definition. Si $\XX$ est un $\SS$-schéma formel, alors $X_i$ sera le $S$-schéma
$\XX\times_{\SS}Spec(V/\pi^{i+1}V)$. Si $X$ est un $S$-schéma quelconque, le $\SS$-schéma formel obtenu par compléter $X$ le long de l'idéal $\pi\OO_X$, sera toujours noté $\XX$.

Tous les schémas considérés dans cet article sont localement noetheriens.

\section{Rappels sur les opérateurs différentiels arithmétiques}
\label{sec_rapdiff}

\subsection[Enveloppes à puissances divisées partielles]{Enveloppes à puissances divisées partielles}
%
Nous utiliserons dans la suite le formulaire (1.1.3 de~\cite{Be1}).
\subsubsection{Définitions}
\label{subsubsection-formules}
Fixons un entier $m$. Pour un entier $k\in\Ne$, $q_k$ est le quotient de la division
euclidienne de $k$ par $p^m$. Berthelot introduit les coefficients suivants pour deux
entiers $k$, $k'$ avec $k\geq k'$

$$\acfrac{k}{k'}=\frac{q_k!}{q_{k'}!q_{k''}!}\,\in\mathbf{N},\quad
                \crofrac{k}{k'}=\parfrac{k}{k'}{\acfrac{k}{k'}}^{-1} \,\in\mathbf{Z}_p  $$
                où $k''=k-k'$.
On peut généraliser ces coefficients pour des multi-indices en posant
pour $\uk=(k_1,\ldots,k_N)\in\Ne^N$, $$q_{\uk}!=q_{k_1}!\cdots q_{k_N}!,$$
et
$$ \acfrac{\uk}{\uk'}=\frac{q_{\uk}!}{q_{_{\uk'}}!q_{{\uk''}}!}\,\in\mathbf{N},\quad
                \crofrac{\uk}{\uk'}=\parfrac{\uk}{\uk'}{\acfrac{\uk}{\uk'}}^{-1} \,\in\mathbf{Z}_p.  $$
S'il y a lieu, on précisera dans ces notations le niveau $m$ en indice.

On se réfère ici à 1.3.5 de \cite{Be1}. Soit $A$ une $\mathbf{Z}_{(p)}$-algèbre commutative,
$(I,J,\gamma)$ un $m$-PD-idéal. Par définition, cela signifie que
$(J,\gamma)$ est un idéal à puissances divisées (un PD-idéal) et que $I$ est muni
de puissances divisées partielles, c'est-à-dire que, pour tout
entier $k$ qui se décompose $k=p^m q_k+r$ (avec $ r<p^m $), il existe une opération définie
pour tout $x$ de $I$ par $$x^{\{k\}}=x^r \gamma_k(x^{p^m}).$$
Quand nous voudrons préciser le niveau $m$, nous noterons $q_k^{(m)}$ l'entier $q_k$ que
nous venons de définir. On a ainsi la relation
$$q_k!x^{\{k\}}=x^k,$$
et le formulaire
\begin{gather*}\label{formulesmPD}
\forall x\in I,\, x^{\{0\}}=1, x^{\{1\}}=x,\, \forall k\geq1, \quad x^{\{k\}}\in I,\,
\forall k\geq p^m, x^{\{k\}}\in J \\
\forall x\in I, \forall a\in A, \forall k\in\Ne,\quad (ax)^{\{k\}}=a^kx^{\{k\}},\\
\forall x,y\in I, \forall k\in\Ne,\quad (x+y)^{\{k\}}=\sum_{k'+k''=k}\crofrac{k}{k'}
x^{\{k'\}}y^{\{k''\}},\\
\forall x\in I,\,\forall k',k''\in\Ne,\quad x^{\{k'\}}x^{\{k''\}}=\acfrac{k'+k''}{k'}
x^{\{k'+k''\}}.
\end{gather*}
Un homomorphisme d'algèbres $\varphi$ entre deux $m$-PD-algèbres $(A,I,J,\gamma)$ et
$(A',I',J',\gamma')$ est un $m$-PD-morphisme si $\varphi(I)\subset I'$ et si
$\varphi$: $(A,J,\gamma)\rig (A',J',\gamma')$ est un $PD$-morphisme, autrement
dit $\gamma'(\varphi(x))=\varphi(\gamma(x))$ pour tout $x$ de $J$. Dans cette situation,
on a donc, pour tout $x\in I$,
$$\varphi\left(x^{\{k\}}\right)=\left(\varphi(x)\right)^{\{k\}}.$$
Ces notations seront aussi utilisées avec des multi-indices $\uk=(k_1,\ldots,k_N)\in\Ne^N$
et les conventions habituelles, pour des multi-indices $\uk$ et $\uk'$ de longueur $N$ avec  $\uk\geq \uk'$
$$ \crofrac{\uk}{\uk'}=\prod_{i=1}^N \crofrac{k_i}{k'_i},\quad
\acfrac{\uk}{\uk'}=\prod_{i=1}^N \acfrac{k_i}{k'_i}.$$

Enfin, on notera $|\uk|=k_1+\cdots +k_N$.
\subsubsection{Enveloppes à puissances divisées d'un faisceau d'idéaux}
\label{descript_loc_pn}

Avec ce formalisme, les $m$-PD-enveloppes à puissances divisées sont construites
de façon analogue aux PD-enveloppes à puissances divisées (1.4.2 de \cite{Be1}). Ces
constructions se faisceautisent. 

On dit qu'un idéal $I$ d'une $V$-algèbre commutative $A$ est régulier si 
$I/I^2$ est un $V=A/I$-module libre de rang fini. En ce cas, si $t_1,\ldots,t_N\in I$ sont tels que 
modulo $I^2$, ces éléments forment une base de $I/I^2$, on dit que 
$t_1,\ldots,t_N$ forment une suité régulière d'éléments de $I$.
Suivant 1.5.3 de \cite{Be1} (que nous appliquons avec
$R=V$ et $A/I=V$), si $I$ est un idéal régulier 
d'une $V$-algèbre $A$ et $(t_1,\ldots,t_N)$ est une suite régulière de paramètres de $I$,
alors la $m$-PD-enveloppe de $I$, notée $P_{(m)}(I)$, est un $V$-module
libre de base les éléments $$\ut^{\{\uk\}}=t_1^{\{k_1\}}\cdots t_N^{\{k_N\}},$$
où $q_i!t_i^{\{k_i\}}=t_i^{k_i}$.
Sous ces hypothèses, ces algèbres sont indépendantes du choix de la $m$-PD-structure
compatible sur la base $V$. Ces algèbres sont munies d'une filtration décroissante par
les idéaux $I^{\{n\}}$,
 tels que
\begin{gather} \label{PDfilt}
 I^{\{n\}}=\bigoplus_{|\uk|\geq n} V \cdot \ut^{\{\uk\}}.
\end{gather}

 Les quotients $P^n_{(m)}(I):= P_{(m)}(I)/I^{\{n+1\}}$ sont donc engendrés comme $V$-module par les éléments $\ut^{\{k\}}$ pour
$|\uk|\leq n$.

Une base duale des $\ut^{\{\uk\}}$ sera notée $\ueta^{\la \uk\ra}$. Ces
éléments vérifient formellement
\begin{gather}
\ueta^{\la \uk\ra}=\prod_i \eta_i^{\la k_i\ra}, \text{ où } \frac{k_i!}{q_{k_i}!}\eta_i^{\la k_i\ra}=\eta_i^{k_i}.
\end{gather}

\subsection{Faisceaux d'opérateurs différentiels}

\subsubsection{Faisceaux gradués de parties principales}
\label{faisceauxpp_grad}
Nous reprenons ici des constructions de \cite{Hu1}. Soient $Y$ un $S$-schéma,
$\EE$ un faisceau localement libre sur $Y$, $\Sg(\EE)$ l'algèbre
symétrique associée au faisceau $\EE$, $\II$ l'idéal des éléments homogènes de
degré $1$ de cette algèbre, on définit ci-dessous 
\begin{gather} \Ga_{(m)}(\EE)=P_{\Sg(\EE),(m)}(\II), 
\end{gather}
où $P_{\Sg(\EE),(m)}(\II)$ est le faisceau de $m$-PD-enveloppes de l'algèbre 
symétrique $\Sg(\EE)$ pour l'idéal $\II$ et,
\begin{gather}
                \Ga_{(m)}^{n}(\EE)=\Ga_{(m)}(\EE)/\II^{\{n+1\}}.
\end{gather}
Ces algèbres sont graduées
\begin{gather}\Ga_{(m)}^{n}(\EE) =\bigoplus_{n' \leq n}  \Ga_{(m),n'}(\EE).
\end{gather}
 Si $\EE$ admet localement pour base $\xi_1,\ldots,\xi_N$, on a la description
locale suivante
 \begin{gather}\label{baseGam}
\Ga_{(m),n'}(\EE)\simeq \bigoplus_{|\uk|=n'}\OO_Y \uxi ^{\{\uk\}}, \text{ où }
q_i!\xi_i^{\{k_i\}}=\xi_i^{k_i}.
\end{gather}

Définissons maintenant par dualité
\begin{gather} \label{defSm}\Sg^{(m)}(\EE)=\bigcup_n \HH om_{\OO_Y}(\Ga_{(m)}^{n}(\EE^{\vee}),\OO_Y).
\end{gather}
On obtient ainsi une algèbre commutative graduée
$$\Sg^{(m)}(\EE)=\bigoplus_n \Sg_{n}^{(m)}(\EE).$$
Soient $\xi_1,\ldots, \xi_N$ une base locale de $\EE$, $\uxi ^{\la\uk\ra}$ la
base duale des éléments $\uxi ^{\{\uk\}}$ de~\ref{baseGam}. On a alors
la description suivante
\begin{gather}\Sg_{n}^{(m)}(\EE)\simeq \bigoplus_{|\uk|=n}\OO_Y \uxi ^{\la\uk\ra},
 \text{ où } \frac{k_i!}{q_{k_i}!}\xi^{\la k_i\ra}=\xi^{k_i}.
\end{gather}
Quand le contexte sera clair on omettra parfois de noter le schéma $Y$ en indice.

\vspace{+2mm}
Appliquons maitenant la définition d'enveloppe à puissances divisées partielles dans le cas des faisceaux de parties principales (1.5.3 et 2.1 \cite {Be1})
de niveau $m$ sur un $S$-schéma lisse $Y$.

\subsubsection{Faisceaux de parties principales}

\label{parties_principales}
 On note $\PP_{Y,(m)}^{n}$
le faisceau des parties principales de niveau $m$ et d'ordre $n$ de $Y$ relativement à
$S$: ce sont les enveloppes à puissance divisées partielles de niveau $m$ de l'idéal $\II$
de l'immersion diagonale $Y\hrig Y\times Y$. Par construction, il existe un morphisme canonique de $\OO_S$-algèbres
\begin{gather}\label{def_rm}
 r_m \,\colon\, \OO_{Y\times Y} \rig \PP_{Y,(m)}^{n}.
\end{gather}
On dispose de $2$ morphismes canoniques induits
par $a\mapsto a\ot 1$, noté $\OO_{Y,g}$ et par $a\mapsto 1\ot a$ noté $\OO_{Y,d}$ sur
$\PP_{Y,(m)}^{n}$, qui munissent ce faisceau d'algèbres de $2$ structures de
$\OO_Y$-module. Lorsque le contexte sera clair les mentions $d$ et $g$ ne
figureront pas dans les notations. L'application $r_m$ est $\OO_{Y}$-linéaire pour la
structure de $\OO_{Y,g}$-module de $\PP_{Y,(m)}^{n}$.

 Si $(y_1,\ldots,y_N)$ est un système de coordonnées locales
sur $Y$, et si $\tau_i=1\ot y_i -y_i\ot 1\in\OO_{Y\times Y}$,
on dispose localement d'un isomorphisme de $\OO_Y$-modules
\begin{gather}\label{basetau}\PP_{Y,(m)}^{n}\simeq \bigoplus_{|\uk|\leq n}\OO_Y \tau_1^{\{u_1\}}\cdots\tau_N^{\{u_N\}},
\text{ où } q_i!\tau_i^{\{k_i\}}=\tau_i^{k_i}.
\end{gather}
L'idéal $\II^{\{k\}}$ est l'idéal engendré par les éléments $\utau^{\{\uu\}}$ avec
$|\uu|\geq k$.

Ces définitions ont un sens pour $m=\infty$ où l'on trouve
que $\PP_{Y,(\infty)}^{n}=\OO_Y/\II^{n+1}$. Le cas $m=0$ correspond à l'algèbre
à puissances divisées classique de l'idéal d'une $V$-algèbre $A$.

En appliquant les résultats de 1.1.5 de ~\cite{Hu1}, on voit que l'algèbre graduée associée
à la filtration $m$-PD-adique de $\PP_{Y,(m)}$, s'identifie à la $m$-PD-algèbre
graduée $\Ga_{(m)}(\II/\II^2)=\Ga_{(m)}(\Omega^1_Y)$ (cf~\ref{faisceauxpp_grad}). Plus précisément, on dispose d'un isomorphisme canonique de $m$-PD-algèbres
\begin{gather} \label{isomgrP}d_m^*\,\colon\,\Ga_{(m)}(\Omega^1_Y) \sta{\sim}{\rig} \gr\pg \PP_{Y,(m)}.
\end{gather}

Dans la suite, on note plus simplement
\begin{gather} \label{notGa}
 \Ga_{Y,(m)}^n=\Ga_{(m)}^n(\Omega^1_Y) \textrm{ (resp. sans $n$)}.
\end{gather}

 Donnons maintenant une description en coordonnées locales de ces algèbres. Avec les
notations précédentes, notons $\xi_i$ la classe de $\tau_i$ dans $\II/\II^2$, alors
localement les faisceaux $\Ga_{(m)}^n$ sont des $\OO_Y$-modules libres de base les
éléments $\uxi^{\{\uu\}}= \xi_1^{\{u_1\}}\cdots\xi_N^{\{u_N\}}$ avec $\sum u_i\leq n$.
\vspace{+2mm}
Soit $f$ un morphisme de $S$-schémas lisses: $Y\rig X$, alors il existe un homomorphisme
canonique $\OO_Y$-linéaire noté $df$: $f^*\PP_{X,(m)}^{n}\rig \PP_{Y,(m)}^{n}$ (2.1.4 de
\cite{Be1}). De plus, on a $df(1\ot h -h \ot 1)=1\ot f^{-1}(h) -f^{-1}(h)\ot 1$ si $h$ est
une section locale de $\OO_X$. On peut ainsi vérifier que si $g$ est un
morphisme de $S$-schémas lisses: $Z\rig Y$, alors le morphisme $d(f\circ g)$:
$(f\circ g)^*\PP_{X,(m)}^{n}\rig \PP_{Z,(m)}^{n}$ coïncide avec
$dg\circ g^*df$.
Considérons le cas où $Y=Z\times X$ et où $p_1$ et $p_2$ sont les deux projections.
On dispose aussi d'une projection $r_2$ : $Y\times Y \rig X \times X$. L'application
canonique $r_2^{-1}\OO_{X\times X}\rig \OO_{Y\times Y}$, envoie l'idéal diagonal
de $X$ vers l'idéal diagonal de $Y$. Par la propriété universelle des $m$-PD enveloppes,
on a donc un morphisme canonique, pour tout $n$
$$ dp_2 \,\colon \, p_2^*\PP_{X,(m)}^{n} \rig \PP_{Y,(m)}^{n}.$$
De même on dispose de
$ dp_1 \,\colon \, p_1^*\PP^{n}_{Z,(m)} \rig \PP^{n}_{Y,(m)},$ et si $J_Z$ est le $m$-PD-idéal
de la $m$-PD-algèbre $\PP^{n}_{Z,(m)}$, on dispose d'un $m$-PD-morphisme
\begin{gather}\label{projP1} s_2 \, \colon \, \PP_{Y,(m)}^{n} \trig \PP_{Y,(m)}^{n}/p_1^*(J_Z) .\end{gather}

Décrivons localement ces morphismes.
 Supposons que $t'_1,\ldots,t'_M$, $t_1,\ldots,t_N$ soient des
coordonnées locales sur $Z$ et $X$ respectivement, notons $\tau'_i=1\ot t'_i -t'_i\ot 1$,
resp. $\tau_i=1\ot t_i -t_i\ot 1$, de sorte que les élements $p_1^*(t'_1),\ldots, p_1^*(t'_M),
p_2^*(t_1),\ldots,p_2^*(t_N)$ forment un système de coordonnées locales de $Y$.
Localement, on peut donc identifier
$$ \PP_{Y,(m)}^{n} \simeq \oplus_{|\ul_1|+|\ul_2|\leq n} \OO_Y
p_1^*({\utau'}_1^{\{\ul_1\}})p_2^*(\utau_2^{\{\ul_2\}}),$$
et $p_1^*(J_Z)$ avec le sous $\OO_X$-module engendré par les
éléments $p_1^*({\utau'}_1^{\{\ul_1\}})p_2^*(\utau_2^{\{\ul_2\}})$ pour lesquels
$|\ul_1|\geq 1$, ce qui donne
$$ \PP_{Y,(m)}^{n}/p_1^*(J_Z) \simeq \oplus_{|\ul_2|\leq n} \OO_Y
p_2^*(\utau_2^{\{\ul_2\}}).$$ Or, avec ce choix de coordonnées locales, on a aussi
$$ \PP_{X,(m)}^{n}\simeq \oplus_{|\ul_2|\leq n} \OO_X
\utau_2^{\{\ul_2\}},$$ en d'autres termes, $s_2 \circ dp_2$ est un $m$-PD-isomorphisme
d'algèbres, noté $\lam_2$. On pose donc
\begin{gather}\label{projP}
  q_2=\lam_2^{-1}\circ s_2 \,\colon \,  \PP_{Y,(m)}^{n} \rig p_2^*\PP_{X,(m)}^{n}.
\end{gather}
Avec les identifications précédentes, on a
$$q_2\left(\sum_{\ul_1,\ul_2}a_{\ul_1,\ul_2}p_1^*({\utau'}_1^{\{\ul_1\}})p_2^*(\utau_2^{\{\ul_2\}})\right)=
\sum_{\ul_2}a_{0,\ul_2}p_2^*(\utau_2^{\{\ul_2\}}).$$ On vérifie ainsi localement que
$q_2$ est une section canonique de $dp_2$.

Le raisonnement qui précède peut aussi s'appliquer aux faisceaux d'algèbres $\Ga_{X,(m)}^n$ et
 $\Ga_{Y,(m)}^n$ si $Y=Z \times X$. Comme on a un isomorphisme
$\Omega^1_{Y}\simeq p_1^*\Omega^1_{Z} \bigoplus p_2^*\Omega^1_{X} $,
 la proposition 1.2.2 de \cite{Hu1}, implique que l'on a un
isomorphisme canonique
$\Ga_{Z\times X,(m)}\simeq p_1^*\Ga_{Z,(m)}\otimes_{\OO_Y} p_2^*\Ga_{X,(m)}.$
En procédant exactement comme ci-dessus, on voit qu'il existe un $m$-PD-morphisme canonique
$$  p_2^*\Ga_{X,(m)}^{n} \rig \Ga_{Z\times X,(m)}^{n},$$
et une section à ce morphisme, qui est aussi un $m$-PD-morphisme
$$q'_2\,\colon\, \Ga_{Z\times X,(m)}^{n} \rig p_2^*\Ga_{X,(m)}^{n}.$$
En coordonnées locales, et avec les notations précédentes, notons
$\xi'_i$ la classe de $1\ot t'_i- t'_i\ot 1$ dans
$  \Ga_{Z,(m)}$ (resp. $\xi_i$ la classe de $1\ot t_i- t_i\ot 1$ dans
$ \Ga_{X,(m)}$). On peut alors identifier
$$\Ga_{Z\times X,(m)}^{n} \simeq \oplus_{|\ul_1|+|\ul_2|\leq n} \OO_Y
p_1^*({\uxi'}_1^{\{\ul_1\}})p_2^*(\uxi_2^{\{\ul_2\}}),\;
\Ga_{X,(m)}^{n} \simeq \oplus_{|\ul|\leq n} \OO_X \uxi_2^{\{\ul\}}),$$
et avec ces identifications
$$q'_2\left(\sum_{\ul_1,\ul_2}a_{\ul_1,\ul_2}p_1^*({\uxi'}_1^{\{\ul_1\}})p_2^*(\uxi_2^{\{\ul_2\}})\right)=
\sum_{\ul_2}a_{0,\ul_2}p_2^*(\uxi_2^{\{\ul_2\}}).$$
\subsubsection{Faisceaux d'opérateurs différentiels}
\label{operateur-diff}
Soit $X$ un $S$-schéma lisse, on définit le faisceau $\Dm_{X,n}$ des opérateurs
différentiels d'ordre $n$ et de niveau $m$, comme
$$\Dm_{X,n}=\HH om_{\OO_X}(\PP_{X,(m)}^{n},\OO_X),$$
le dual étant pris pour la structure gauche de $\OO_X$-module sur le
faisceau $\PP_{X,(m)}^{n}$. Le faisceau des opérateurs différentiels de niveau
$m$ est $\Dm_X=\varinjlim_n \Dm_{X,n}$.

Le faisceau $\Dm_X$ est muni d'une structure d'anneau de la façon suivante (2.2.1 de \cite{Be1}).
L'application $\OO_{X\times X}\rig \OO_{X\times X}\ot_{\OO_X}\OO_{X\times X}$,
qui envoie $a\ot b $ sur $a\ot 1 \ot 1 \ot b$,
 induit un unique $m$-PD morphisme de $m$-PD-algèbres
\begin{gather}\label{def-deltaD}
\delta^{n,n'} \,\colon \,
\PP_{X,(m)}^{n+n'}\sta{\delta^{n,n'}}{\rig}\PP_{X,(m)}^{n}\ot_{\OO_X} \PP_{X,(m)}^{n'},
\end{gather}
où le produit tensoriel est donné par la structure de $\OO_X$-module à gauche de $\PP^{n'}_{X,(m)}$ et
par la structure de $\OO_X$-module à droite de $\PP^{n}_{X,(m)}$.

Soient $P\in \Dm_{X,n}$, $P'\in \Dm_{X,n'}$, l'opérateur $PP'$ est obtenu comme composé
\begin{gather}\label{produitD}
\PP_{X,(m)}^{n+n'}\sta{\delta^{n,n'}}{\rig}\PP_{X,(m)}^{n}\ot \PP_{X,(m)}^{n'}
\sta{Id \ot P'}{\rig}\PP_{X,(m)}^{n} \sta{P}{\rig} \OO_X.
\end{gather}

Nous renvoyons à 2.3 de \cite{Be1} pour le fait que se donner sur un $\OO_X$-module
$\EE$ une structure de $\Dm_X$-module revient à se donner une $m$-PD stratification. En dualisant,
on déduit facilement de l'isomorphisme $d_m^*$ de \ref{isomgrP}, comme en 1.3.7.3 de \cite{Hu1}, l'isomorphisme canonique
d'algèbres commutatives
\begin{gather} \label{grdiffDm}
d_m  \,\colon\,\gr\pg\Dm_X \sta{\sim}{\rig} \Sg^{(m)}(\TT_X)
\end{gather}
où $\TT_X$ désigne le faisceau tangent de $X$. En particulier, l'algèbre $\gr\pg\Dm_X(U)$ est noetherienne sur les ouverts affines $U$ de $X$,
et donc aussi $\Dm_X(U)$ par un argument classique. 

Sur $\XX$ le complété formel de $X$, nous introduirons
$$\Dcm_{\XX}=\varprojlim_i  \Dm_X/\pi^{i+1} \Dm_X,\textrm{ et }\, \Ddag_{\XX,\Qr}=\varinjlim_m\Dcm_{\XX,\Qr}.$$

Si $x_1,\ldots,x_N$ est un système de coordonnées locales sur $X$ sur un ouvert $U$, $\tau_i=1\ot x_i -x_i \ot 1$, on note
$\uder^{\la \uk \ra}$ la base duale des $\utau^{\{\uk\}}$ de ~\ref{basetau}, et $\der_{i}^{[k_i]}=\der_i^{k_i}/k_i!$.
Alors, $\der_{i}^{\la k_i \ra}=q_{k_i}!\der_{i}^{[k_i]}$.
On a la description locale suivante sur l'ouvert $U$
\begin{gather}\label{desclocD} {\Dm_X}_{|U} \simeq \oplus_{\uk}\OO_U \uder^{\la \uk \ra} \end{gather}
La structure de $\Dm_X$-module de $\OO_X$ est donné par le composé suivant, pour $P$ une section locale de $\Dm_{X,n}$,
\begin{gather}\label{action-D}
\xymatrix @R=0mm { \OO_X \ar@{->}[r]^{u_d} & \PP_{X,(m)}^{n} \ar@{->}[r]^{P} & \OO_X \\
                  f\ar@{{|}->}[r] & 1\ot f .&  }
\end{gather}

Pour les opérations cohomologiques pour les $\Dm_X$ (resp. $\Dcm_{\XX}$,
$\Ddag_{\XX,\Qr}$)-modules cohérents, nous nous référons aux chapitres 2 et 4 de
\cite{Be-survey}. Nous utiliserons ces opérations pour des complexes de la catégorie
dérivée des $\Dm_X$-modules, à cohomologie bornée et cohérente, que nous noterons
$D^b_{c}(\Dm_X)$. Notons que si $f$ : $Y \rig X$ est un morphisme de $S$-schémas lisses,
le foncteur image inverse $f^!$ va de la catégorie dérivée $D^b_{c}(\Dm_X)$ vers
$D^b_{c}(\Dm_Y)$.

Soit $\EE$ un $\Dm_X$-module cohérent (vu comme complexe en degré
$0$), alors, comme $\OO_Y$-module,
$f^!\EE$ s'identifie à $\OO_Y\ot_{f^{-1}\OO_X}^{{\bf L}}f^{-1}\EE$ (4.2 de chap.VI de
\cite{Bo}). Mais comme $f$ est
plat, la cohomologie de ce complexe est concentrée en degré $0$, où il est isomorphe à
$f^*\EE$. Dans la suite de ce texte la notation $\HH^0f^!\EE$ désignera donc, si $f$ est lisse
et $\EE$ est un $\Dm_X$-module cohérent, un $\Dm_Y$-module cohérent isomorphe, comme
$\OO_Y$-module, à $f^*\EE$.

Nous aurons aussi besoin du fait suivant. Soit $\II$ l'idéal de l'immersion diagonale
$X\hrig X\times X$, alors, on a une suite exacte
$$ 0 \rig \II^{\{1\}}/\II^{\{2\}} \rig \PP_{X,(m)}^{1}  \rig \OO_X,$$
où on identifie $\OO_X $ à l'algèbre $\PP_{X,(m)}^{0}$. Or, comme
$\II^2\subset \II^{\{2\}}$, on a une flèche canonique $\II/\II^2 \rig
\II^{\{1\}}/\II^{\{2\}}$. Il résulte de ~\ref{PDfilt} que $\II^{\{1\}}=\II+\II^{\{2\}}$, et on voit aussi que
 cette flèche est un isomorphisme en coordonnées locales donc un isomorphisme.
 D'autre part,
l'application canonique $\OO_X \rig \PP_{X,(m)}^{1}$ donne un scindage de la suite exacte
ci-dessus,
qui fournit un isomorphisme
\begin{gather} \label{desc-P1diff}   \PP_{X,(m)}^{1}\sta{\sim}{\lrig} \OO_X \bigoplus \Omega^1_X.
\end{gather}
En dualisant, on trouve donc des isomorphismes
\begin{gather} \label{desc-D1diff} B_m\,\colon\, \Dm_{X,1} \sta{\sim}{\rig} \OO_X \bigoplus \TT_X,\,
 \textrm{ et } \overline{B}_m \,\colon\, \Dm_{X,1}/\Dm_{X,0} \sta{\sim}{\rig} \TT_X \textrm{ et }
 \overline{B}_m^* \,\colon\,\II/\II^2 \sta{\sim}{\rig} \gr_1 \PP_{X,(m)},
\end{gather}
où $\overline{B}_m^*$ est obtenu en dualisant $\overline{B}_m$.

Si $V$ est un AVDC, $\SS=Spf\,V$ et si $\XX$ est un schéma formel lisse sur $\SS$,
 nous introduirons
$$\Dcm_{\XX}=\varprojlim_i  \Dm_{X_i},\textrm{ et }\, \Ddag_{\XX,\Qr}=\varinjlim_m\Dcm_{\XX,\Qr}.$$
Si $x_1,\ldots,x_N$ est un système de coordonnées locales sur $\XX$ sur un ouvert formel $\UU$, $\tau_i=1\ot x_i -x_i \ot 1$, on note
$\uder^{\la \uk \ra}$ la base duale des $\utau^{\{\uk\}}$ de ~\ref{basetau}, et $\der_{i}^{[k_i]}=\der_i^{k_i}/k_i!$.
Alors, $\der_{i}^{\la k_i \ra}=q_{k_i}!\der_{i}^{[k_i]}$.
On a la description locale suivante sur l'ouvert affine $\UU$
\begin{gather}\label{desclocDdag}
                     \Ga(\UU,\Dcm_{\XX})=\left\{\sum_{\uk\in \Ne^{N}}a_{\uk}\uder^{\la \uk \ra}, \, a_{\uk}\in
\OO_{\XX}(\UU) \; | \;  \vp(a_{\uk})\rig +\infty \; \text{si} \;|\uk|\rig +\infty \right\} \\
            \Ga(\UU,\Ddag_{\XX,\Qr})=\left\{\sum_{\uk\in \Ne^{N}}a_{\uk}\uder^{[ \uk ]} , \, a_{\uk}\in\OO_{\XX,\Qr}(\UU)
\; | \; \exists c\in\Rr, \eta >0,  \text{tels que } \vp(a_{\uk})>\eta |\uk| + c \right\}.
\end{gather}
Ici, $\vp$ désigne la valuation $p$-adique sur $V$-algèbre plate $\OO_{\XX}(\UU)$ et aussi la valuation induite sur $\Qr$-algèbre $\OO_{\XX,\Qr}(\UU)=\OO_{\XX}(\UU)\ot\Qr$.
\subsubsection{Faisceaux d'opérateurs différentiels twistés}
\label{operateur-difft}
Nous aurons besoin d'une version twistée par un faisceau inversible des
constructions précédentes. Nous commencerons par
définir les faisceaux de parties principales twistés.
Reprenons ici les hypothèses de la sous-section
précédente~\ref{operateur-diff}. Soit $\LL$ un $\OO_X$-module inversible sur
$X$. Introduisons le faisceau twisté des parties principales
$$ {}^t\PP_{X,(m)}^{n} = \LL^{-1}\ot_{\OO_{X,g}}\PP_{X,(m)}^{n}\ot_{\OO_{X,d}}\LL.$$
La notation $\OO_{X,d}$ (resp. $\OO_{X,g}$) signifie que l'on prend la structure de $\OO_X$-module
à droite (resp. à gauche) sur $\PP_{X,(m)}^{n}$ décrite en ~\ref{parties_principales}.

\vskip4pt
Remarque: les faisceaux ${}^t\PP_{X,(m)}^{n}$ ne sont pas des faisceaux d'algèbres, sauf si $\LL$ est trivial.

Définissons
sur $X$ les faisceaux d'opérateurs différentiels twistés d'ordre inférieur à $n$
${}^t\DD^{(m)}_{X,n}$ par
\begin{align*}{}^t{}\DD^{(m)}_{X,n}&=\HH om_{\OO_{X,g}}({}^t\PP_{X,(m)}^{n},\OO_X) \\
                                  &\simeq \HH om_{\OO_{X,g}}(\PP_{X,(m)}^{n}\ot_{\OO_{X,d}}\LL,\LL).
\end{align*}
Identifions $\LL^{-1}$ à $\HH om_{\OO_X}(\LL,\OO_X)$. Alors on dispose d'un isomorphisme canonique
$$
\xymatrix @R=0mm {  \HH om_{\OO_{X,g}}(\PP_{X,(m)}^{n},\OO_X)\ot_{\OO_{X,d}} \ar@{->}[r]^{\sim}\LL^{-1} &
        \HH om_{\OO_{X,g}}(\PP_{X,(m)}^{n}\ot_{\OO_{X,d}}\LL,\OO_X)\\
                  Q\ot \varphi \ar@{{|}->}[r] & (T\ot l \mapsto Q(T\cdot(1\ot \varphi(l))) }, $$
qui nous donne finalement, après tensorisation sur $\OO_{X,g}$ par $\LL$ un isomorphisme canonique
$$ \HH om_{\OO_{X,g}}(\PP_{X,(m)}^{n}\ot_{\OO_{X,d}}\LL,\LL) \simeq \LL\ot_{\OO_{X,g}}\DD^{(m)}_{X,n}\ot_{\OO_{X,d}}\LL^{-1},$$
et donc
\begin{gather}
 ^t{}\DD^{(m)}_{X,n}\simeq \LL\ot_{\OO_{X,g}}\DD^{(m)}_{X,n}\ot_{\OO_{X,d}}\LL^{-1}.
\end{gather}
On définit aussi
   \begin{align*} {}^t\DD^{(m)}_X&=\varinjlim_n {}^t\DD^{(m)}_{X,n} \\
   &\simeq \LL\ot_{\OO_{X,g}}\DD^{(m)}_X\ot_{\OO_{X,d}}\LL^{-1}.
    \end{align*}
Observons que le $m$-PD-morphisme $\delta^{n,n'}$ de ~\ref{def-deltaD} est
$\OO_{X,g}\times \OO_{X,d}$-linéaire pour
la structure de $\OO_{X,d}$-module sur $\PP_{X,(m)}^{n'}$ (resp. $\PP_{X,(m)}^{n+n'}$),
et la structure de $\OO_{X,g}$ -module sur $\PP_{X,(m)}^{n}$ (resp. $\PP_{X,(m)}^{n+n'}$),
de sorte que l'on peut considérer l'application
\begin{gather*}
{}^t\delta^{n,n'} \,\colon \,\xymatrix{
\LL^{-1}\ot_{\OO_{X,g}}\PP_{X,(m)}^{n+n'}\ot_{\OO_{X,d}}\LL
\ar@{->}[r]& \LL^{-1}\ot_{\OO_{X,g}} \PP_{X,(m)}^{n}
\ot \PP_{X,(m)}^{n'}\ot_{\OO_{X,d}}\LL}
\end{gather*}
defini par $Id\ot \delta^{n,n'}\ot Id$.
Autrement dit, on dispose d'un $m$-PD-morphisme
\begin{gather}\label{def-deltaDt}
{}^t\delta^{n,n'} \,\colon \,\xymatrix{
{}^t\PP_{X,(m)}^{n+n'}
\ar@{->}[r]^(0.32){{}^t\delta^{n,n'}}&{}^t\PP_{X,(m)}^{n} \ot {}^t\PP_{X,(m)}^{n'}.}
\end{gather}
Le faisceau ${}^t\DD^{(m)}_X$ est alors un faisceau d'anneaux comme dans le cas classique.
Soient $P\in {}^t\Dm_{X,n}$, $P'\in {}^t\Dm_{X,n'}$, l'opérateur $PP'$ est obtenu comme composé
\begin{gather}\label{produitDt}\xymatrix{
{}^t\PP_{X,(m)}^{n+n'}\ar@{->}[r]^(0.4){{}^t\delta^{n,n'}}&{}^t\PP_{X,(m)}^{n}\ot {}^t\PP_{X,(m)}^{n'}
\ar@{->}[r]^<<<<<{Id \ot P'}&{}^t\PP_{X,(m)}^{n} \ar@{->}[r]^{P}& \OO_X.}
\end{gather}
Soient $a$ une section locale de $\OO_X$, $\tau$ une section locale de l'idéal diagonal $\II$,
alors $(a\ot 1)\tau=\tau  (1\ot a)$ et donc les structures de $\OO_{X,g}$ et $\OO_{X,d}$-modules
co\"\i ncident sur $\II/\II^{2}$. En particulier, on dispose d'un isomorphisme
$$ \LL^{-1}\ot_{\OO_{X,g}} \Ga_{(m)}(\Omega^1_X)\ot_{\OO_{X,d}}\LL \simeq \Ga_{(m)}(\Omega^1_X).$$

D'après ~\ref{isomgrP}, l'algèbre graduée pour la filtration $\II$-adique de
$^t\PP_{X,(m)}^{n}$ s'identifie à
$$ \LL^{-1}\ot_{\OO_{X,g}} \Ga_{(m)}(\Omega^1_X)\ot_{\OO_{X,d}}\LL \simeq \gr\pg^t\PP_{X,(m)}^{n},$$
ce qui, compte tenu de la remarque précédente, donne un isomorphisme canonique
 \begin{gather} \label{isomgrPt}\Ga_{(m)}(\Omega^1_X) \sta{\sim}{\rig} \gr\pg {}^t\PP_{X,(m)}.
\end{gather}
En dualisant, cet isomorphisme donne un isomorphisme canonique pour l'algèbre graduée des
opérateurs différentiels twistés
 \begin{gather} \label{isomgrDt}  \gr\pg {}^t\DD_{X,(m)}\sta{\sim}{\rig} \Sg^{(m)}(\TT_X).
\end{gather}
De ce fait le gradué associé à la filtration par l'ordre des opérateurs différentiels est
commutatif et noetherien sur les ouverts affines. Comme dans le cas classique, on en déduit que le
faisceau ${}^t\Dm_X$ est cohérent et à sections noethériennes sur les ouverts affines.

Soient $x_1,\ldots,x_N$ un système de coordonnées locales sur un ouvert $U$ de $X$, $\tau_i=1\ot x_i -x_i \ot 1$.
On suppose de plus que $\LL_{|U}$ est engendré par un élément $u$. En reprenant les notations
de ~\ref{basetau} et ~\ref{desclocD}, on dispose des isomorphismes
suivants sur cet ouvert $U$
\begin{gather}\label{basetw}^t\PP_{X,(m)}^{n}\simeq \bigoplus_{|\uk|\leq n}\OO_X \, u^{-1}\ot
\utau^{\{\uk\}}\ot u,\\
{}^t\DD_{X,(m)}\simeq \bigoplus_{|\uk|}\OO_X \, u \ot \uder^{\la \uk \ra} \ot u^{-1}.
\end{gather}
De plus, les éléments $u \ot \uder^{\la \uk \ra} \ot u^{-1}$ forment la base duale de la base
constituée des éléments $u^{-1}\ot\utau^{\{\uk\}} \ot u $.

Si $V$ est un AVDC, $\SS=Spf\,V$, $\XX$ un schéma formel lisse sur $\SS$,
nous introduisons
$${}^t\Dcm_{\XX}=\varprojlim_i  {}^t\Dm_{X_i},\textrm{ et }\, {}^t\Ddag_{\XX,\Qr}=\varinjlim_m{}^t\Dcm_{\XX,\Qr}.$$
On dispose alors de la
\begin{sousprop} \be \item[(i)] Les faisceaux complétés ${}^t\Dcm_{\XX}$ sont cohérents, à sections noethériennes
sur les ouverts affines.
\item[(ii)] Le faisceau ${}^t\Ddag_{\XX,\Qr}$ est cohérent.
\ee
\end{sousprop}
\begin{proof} Le (i) se démontre comme dans le cas classique, à partir du fait
que les faisceaux $ {}^t\DD_{X,(m)}$ sont à sections
noetheriennes sur les ouverts affines, et des propriétés du passage au
complété $p$-adique.

Pour le (ii), on peut en outre remarquer que si $u$ est un générateur local de
$\LL$ sur un ouvert affine $U$, on dispose d'un isomorphisme d'algèbres défini
localement sur $\UU$ par
\begin{gather}\xymatrix @R=0mm { {\Ddag_{\XX,\Qr}}_{|\UU} \ar@{->}[r] &  \LL\ot
\Ddag_{\XX,\Qr}\ot \LL^{-1}_{|\UU} \\
          P \ar@{{|}->}[r] & u \ot P \ot u^{-1}.}
\end{gather}
De cette façon, on voit que ${}^t\Ddag_{\XX,\Qr}$ est cohérent puisque c'est
le cas de $\Ddag_{\XX,\Qr}$. Nous aurions aussi pu utiliser un argument du
même type pour le (i).
\end{proof}
%
\subsubsection{Compléments sur les faisceaux d'opérateurs différentiels twistés}
\label{comp_opdifft}
Reprenons les notations et hypothèses de la sous-section précédente~\ref{operateur-difft}.
Définissons en suivant les notations de ~\ref{faisceauxpp_grad} le fibré vectoriel associé à $\LL$
$$Y=Spec~ \Sg(\LL),$$ et $q$ le morphisme canonique : $Y\rig X$. Nous identifierons dans la suite
$\PP_{X,(m)}^{n}\ot_{\OO_{X,d}} \LL$ au sous-$\PP_{X,(m)}^{n}$-module de $q_*\PP_{Y,(m)}^{n}$ engendré
par $1 \ot \LL$ via $T\ot f \mapsto T\cdot (1\ot f)$, pour $f\in\LL$ et $T\in \PP_{X,(m)}^{n}$.

Introduisons aussi  $\Dm_Y(\LL)$, le sous-faisceau de $\OO_X$-modules constitué des
opérateurs différentiels sur $Y$ qui se restreignent en des opérateurs
différentiels agissant sur $X$ et $\LL$, c'est-à-dire
\begin{gather} \label{def_DYL}
\Dm_{Y,n}(\LL)=\{P\in q_*\Dm_Y\, | \, P(\PP_{X,(m)}^{n})\subset \OO_X \text{ et }
P(\PP_{X,(m)}^{n}\ot_{\OO_{X,d}} \LL)\subset \LL \},
\end{gather}
et $$\Dm_{Y}(\LL)=\varinjlim_n \Dm_{Y,n}(\LL).$$
Si $P\in \Dm_{Y,n}(\LL)$, on définit $$r_{\LL}(P)=id_{\LL^{-1}}\ot P_{|\PP_{X,(m)}^{n}\ot \LL} \, \in
{}^t\DD^{(m)}_{X,n},$$ qu'on appelle morphisme de restriction à $\LL$.

Soient $x_1,\ldots,x_N$ des coordonnées locales sur un ouvert $U$ de $X$,
tel que $\LL_{|U}$ soit un $\OO_U$-module libre engendré par $u$. Alors
$x_1,\ldots,x_N,u$ forment un système de coordonnées sur l'ouvert $q^{-1}(U)\subset Y$.
Notons $\uder^{\la \uk \ra}_{x}$ les opérateurs sur $U$ correspondant aux coordonnées
$x_1,\ldots,x_N$ sur $X$, et $\der_u^{\la l_u \ra}$ les opérateurs de $\Dm_Y$ correspondant à la coordonnée
$u$ sur $q^{-1}(U)$. Reprenons les notations de ~\ref{basetw}. Le faisceau $\Dm_Y$ est muni de la filtration
par l'ordre des opérateurs différentiels. On dispose alors d'isomorphismes canoniques~\ref{grdiffDm}
$\gr\pg\Dm_Y\simeq \Sg^{(m)}(\TT_Y)$ (resp. sur $X$). Comme $Y$ est un fibré vectoriel sur $X$, on a un projecteur
$q_*\TT_Y \rig \TT_X$, qui est une section du morphisme canonique $\TT_X \rig q_*\TT_Y$. On en déduit un projecteur
$q_*\Sg^{(m)}(\TT_Y)\rig \Sg^{(m)}(\TT_X)$, et donc via les isomorphismes canoniques précédents un projecteur
$\lambda_q$ : $\gr\pg\Dm_Y\rig \gr\pg\Dm_X$.
On a la description locale suivante
\begin{sousprop} \label{rL}\begin{enumerate} \item[(i)] $P$ est dans $ \Dm_Y(\LL)(U)$ si et seulement si il existe
$ a_{\uk},b_{\uk}\in \OO_X(U)$, $c_{\uk,l_u}\in q_*\OO_Y(U)$ tels que
           $$P=\sum_{\uk\in\Ne^N} a_{\uk}\uder^{\la \uk \ra}_{x}+\sum_{\uk\in\Ne^N} b_{\uk}\uder^{\la \uk \ra}_{x}u\der_u
          + \sum_{\uk\in\Ne^N,l_u\geq 2}c_{\uk,l_u}\uder^{\la \uk \ra}_{x}\der_u^{\la l_u \ra}.$$
           \item[(ii)] Si $P$ est dans $\Dm_Y(\LL)(U)$ et s'écrit comme précédemment, alors
             $$r_{\LL}(P)=\sum_{\uk\in\Ne^N} a_{\uk}\,u\ot \uder^{\la \uk \ra}_{x} \ot u^{-1} +
                          \sum_{\uk\in\Ne^N} b_{\uk}\, u\ot \uder^{\la \uk\ra}_{x} \ot u^{-1}.$$
         \item[(iii)] L'application $r_{\LL}$ est filtrée, surjective, et n'est pas injective.
         \item[(iv)] L'application graduée $\gr\pg r_{\LL}$ induite par $r_{\LL}$ sur $ \gr\pg\Dm_Y(\LL)$ est égale au
projecteur $\lambda_q$ (restreint à $
\gr\pg\Dm_Y(\LL)$).
          \end{enumerate}
\end{sousprop}
\begin{proof} Soient $\tau_i=1\ot x_i - x_i \ot 1$, $\tau_u=1\ot u - u \ot 1$. Nous utilisons les notations de
~\ref{parties_principales}. Soit $P\in q_*\Dm_Y(\LL)$. Nous avons alors le
\begin{souslem} $P$ est dans $\Dm_Y(\LL)(U)$ si et seulement si
                 $$P(\PP_{X,(m)}^{n})\subset \OO_X \text{ et }
    \forall \ul\in\Ne^N,\,P(\utau^{\{\ul\}}\tau_u)\in \LL.$$
\end{souslem}
Démontrons ce lemme. Soient $a \in \OO_X(U)$, $\ul\in\Ne^N$ et $P$ vérifiant les conditions du lemme.
Comme $P$ est $q_*\OO_Y$-linéaire à gauche,
la condition du lemme implique
		$$P(\PP_{X,(m)}^{n}\tau_u)\in \LL.$$ Notons $T=\utau^{\{\ul\}}(1\ot a)\in \PP_{X,(m)}^{n}$,
alors
\begin{align*} P(\utau^{\{\ul\}}(1\ot a)) &=P(T(1\ot u)) \\
                                           &=P((u\ot 1)T+T\tau_u) \\
                                           &=P(T)u+P(T\tau_u)\,\in \LL,
\end{align*}
de sorte que $P\in\Dm_Y(\LL)$. Réciproquement, si $P\in\Dm_Y(\LL)$, alors
$$P(\utau^{\{\ul\}}\tau_u)=P(\utau^{\{\ul\}}(1\ot u))- P(\utau^{\{\ul\}})u\,\in \LL,$$
de sorte que $P$ vérifie les conditions du lemme.

Revenons à la démonstration de la proposition. Soit $P\in q_*\Dm_{Y}$. Ecrivons
$$P=\sum_{\uk,l_u} d_{\uk,l_u}\uder^{\la \uk \ra}_{x}\der_u^{\la lu \ra},$$ avec $d_{\uk,l_u}\in\OO_Y$.
Alors $d_{\uk,0}=P(\utau^{\{\uk\}})\in\OO_X$ et $d_{\uk,1}=P(\utau^{\{\uk\}}\tau_u)\in\LL$, ce qui donne le (i).

Pour le (ii), il suffit d'utiliser la base locale de $^t\PP_{X,(m)}^{n}$ constituée des éléments
$u^{-1}\ot \utau^{\{\uk\}}\ot u$ donnée en~\ref{basetw}. Or, on a
\begin{align*} \uder^{\la \uk \ra}_{x}(\utau^{\{\ul\}}\ot u)
             &=\uder^{\la \uk \ra}_{x}(\utau^{\{\ul\}}\tau_u+(u\ot 1)\utau^{\{\ul\}}) \\
             &=u \uder^{\la \uk \ra}_{x}(\utau^{\{\ul\}})\\
             &=u \delta_{\uk,\ul},
\end{align*}
où $\delta_{\uk,\ul}$ désigne le symbôle de Kronecker. De façon analogue,
\begin{align*} u\uder^{\la \uk \ra}_{x}\der_u(\utau^{\{\ul\}}\ot u)
             &=u\uder^{\la \uk \ra}_{x}\der_u(\utau^{\{\ul\}}\tau_u+(u\ot 1)\utau^{\{\ul\}}) \\
             &=u \delta_{\uk,\ul},
\end{align*}
d'où le (ii) du lemme. Passons à (iii). L'application $r_{\LL}$ est filtrée par définition.
 Le reste se vérifie localement. Or on a une section locale de $r_{\LL}$ en posant,
pour $P\in\Dm_X(U)$, $r_{\LL}(u\ot P \ot u^{-1})=P\in \Dm_Y(\LL)(U)$. De plus $r_{\LL}(\der_u^{\la 2 \ra})=0$,
de sorte que $r_{\LL}$ n'est pas injectif. L'affirmation (iv) est une conséquence immédiate de (ii) .\end{proof}
Sur la description locale de $\Dm_Y(\LL)$, on voit que c'est un sous-faisceau
de $\OO_S$-algèbres (et de $\OO_X$-modules) de $q_*\Dm_Y$. On a plus précisément la
\begin{sousprop} \be \item[(i)] Le faisceau $\Dm_Y(\LL)$ est un sous-faisceau
d'algèbres du faisceau $q_*\Dm_Y$.
\item[(ii)] La flèche de restriction $r_{\LL}$ : $\Dm_Y(\LL)\rig {}^t\DD^{(m)}_{X}$ est un homomorphisme surjectif de
faisceaux d'algèbres.
\ee
\end{sousprop}
\begin{proof} Commençons par (i). Soient $P,P'\in \Dm_Y(\LL)$. Rappelons que
le produit $PP'$ est donné par le composé suivant \ref{produitD}
$$ q_*\PP_{Y,(m)}^{n+n'}\sta{q_*\delta^{n,n'}}{\rig}q_*\PP_{Y,(m)}^{n}\ot q_*\PP_{Y,(m)}^{n'}
\sta{Id \ot P'}{\rig}q_*\PP_{Y,(m)}^{n} \sta{P}{\rig} q_*\OO_Y,  $$
où $\delta^{n,n'}$ est induit par $a\ot b \mapsto a\ot 1 \ot 1 \ot b$. Ainsi,
on a
$q_*\delta^{n,n'}(1\ot \LL)\subset 1 \ot 1 \ot 1 \ot \LL,$
et $(id\ot P')(1 \ot 1 \ot 1 \ot \LL)\subset 1 \ot \LL$, si bien que
$PP'(1\ot \LL)\subset \LL$. De plus, les faisceaux
$\PP_{X,(m)}^{n}\ot \PP_{X,(m)}^{n'}\ot_{\OO_{X,d}}\LL$ sont des sous-faisceaux de
$q_*(\PP_{Y,(m)}^{n}\ot \PP_{Y,(m)}^{n'})$, et les morphismes
$\delta^{n,n'}$ sont des $m$-PD-morphismes. Le fait que
$PP'(1\ot \LL)\subset \LL$ entraîne donc formellement que $PP'(\PP_{X,(m)}^{n}\ot_{\OO_{X,d}} \LL)\subset \LL.$
Il est d'autre part clair que $PP'(\PP_{X,(m)}^{n})\subset \OO_X$ et ceci montre (i).

D'après le (i), le composé $PP'$ est donné en restriction à $(1\ot \LL)$ par
$$ \PP_{X,(m)}^{n+n'}\ot_{\OO_{X,d}}\LL\sta{\delta^{n,n'}}{\rig}\PP_{X,(m)}^{n}\ot \PP_{X,(m)}^{n'}\ot_{\OO_{X,d}}\LL
\sta{Id \ot P'}{\rig}\PP_{X,(m)}^{n}\ot_{\OO_{X,d}}\LL \sta{P}{\rig} \LL.$$
En tensorisant ce diagramme à gauche (sur $\OO_{X,g}$) par $\LL^{-1}$, on
retrouve exactement le diagramme \ref{produitDt} définissant le produit dans
${}^t\DD^{(m)}_{X}$, ce qui montre que $r_{\LL}$ est un homomorphisme
d'anneaux. La surjectivité a été établie dans la proposition précédente ~\ref{rL}.
\end{proof}

\section{Faisceaux équivariants}
\label{desc_actions_G}

\subsection{Notations-Rappels}
\label{section-rappels}

Soient $D$ une $V$-algèbre commutative, $A$, $B$ des $D$-modules,
$C$ une $D$-algèbre, $u$ (resp. $v$) un homomorphisme $D$-linéaire: $A\rig C$,
(resp. $v$:$B\rig C$),
$m$ l'application produit: $C\ot_D C\rig C$ telle que $m(a\ot b)=ab$. On note
alors $u\overline{\ot }v=m\circ (u\ot v)$: $A\ot_D B\rig C$.

Soient $G$ un schéma en groupes affine et lisse sur $S$, $\mu$: $G\times G \rig G$ l'application produit,
$e$: $S \hrig G$ l'élément neutre. Les applications déduites de $\mu$ et $e$ au niveau des
faisceaux structuraux seront notées $\mu^{\sharp}$ et $\varep_G$. 

Un $G$-comodule $M$ est un $V$-module $M$ muni d'une action à droite de $G$, c'est-à-dire
que pour toute $V$-algèbre $R$, $R\ot_V M$ est un $G(R)$-module à droite. Comme $G$ est affine, cela revient à se donner une
application de comodule dual $\Delta_M$: $M\rig V[G]\ot_V M$, vérifiant les deux égalités
suivantes:
\begin{gather}\label{comodule}
(id_{V[G]}\ot \Delta_M)\circ \Delta_M=(\mu^{\sharp}\ot id_M)\circ \Delta_M \\
\label{comodule2} (\varep_G\overline{\ot }id_M)\circ \Delta_M=id_M .
\end{gather}
Notons que cette définition diffère de celle de $G$-module de Jantzen (I 2.8 de \cite{Jantzen})
qui décrit la relation de comodule sur un $V$-module $M$, pour laquelle $ M\ot_V R $ est un
$G(R)$-module à gauche pour toute $R$-algèbre $V$.

\vspace{+2mm}
Nous aurons besoin de petits lemmes techniques qui font l'objet de la sous-section
suivante.
\subsection{Lemmes techniques}
\label{subsection-lem_techniques}
On rappelle quelques faits classiques.
\begin{lem} Soient $f,g$ deux morphismes de $\OO_X$-modules cohérents
$\FF \rig \GG$
sur un $S$-schéma $X$ 
localement noethérien. Pour $x$ un point fermé de $X$,
on note $i_x$ l'immersion : $\{x\}\hrig X$. On suppose que pour tout point fermé de $X$,
les morphismes induits $f(x),g(x)$: $i_x^*\FF \rig i_x^*\GG$ coïncident. Alors $f=g$.
\end{lem}
\begin{proof} On peut supposer que $X$ est affine et noethérien. Soit $\HH$ un $\OO_X$-module
cohérent sur $X$ tel que $i_x^*\HH=0$ pour tout point fermé $x$ de $X$. Alors le support
de $\HH$ est fermé et s'il est non vide, il correspond à un sous-schéma fermé réduit de $X$
contenant un point fermé $x$. Par hypothèse,
$i_x^*\HH$ est nul, donc $\HH_x$ est nul par le lemme de Nakayama, ce qui est absurde:
cela montre que $\HH$ est nul.  Soit $\II$ le faisceau cohérent
image de $f-g$. La remarque précédente appliquée à $\II$ donne le résultat.
\end{proof}

Si on applique cette même remarque aux faisceaux cohérents $Ker(f)$ et $\GG/Im(f)$
on obtient aussi le résultat suivant.
\begin{lem} Sous les hypothèses et notations précédentes, soit $f$ un morphisme de $\OO_X$-modules
cohérents : $\FF \rig \GG$. Alors $f$ est surjectif (resp. un
isomorphisme, resp. injectif) si et seulement si en tout point fermé $x$ de $X$, l'homorphisme induit
$f(x)$ est surjectif (resp. un isomorphisme, resp. injectif). \end{lem}
%

\subsection{Faisceaux $G$-équivariants}
\label{FaisceauxGequiv}
Il s'agit ici essentiellement de faire quelques rappels.
\subsubsection{Définitions}
\label{O-mod-equiv}
Dans cette section, $X$ est un $S$-schéma, muni d'une action à gauche $\sigma_X$:
$G\times_S X \rig X$. On note $\sigma^{\sharp}:\sigma^{-1}\OO_X\rig \OO_{G\times X}$
l'application qu'on en déduit au niveau des faisceaux structuraux. L'application
$\sigma_X$ sera tout simplement notée $\sigma$ lorsqu'il n'y aura pas ambiguité. De plus, le produit fibré $\times_S$ sera notée $\times$.
 On note
$p_2$ la deuxième projection: $G\times X \rig X$, $t_1,t_2,t_3$:
$G\times G\times X\rig G\times X$ définies respectivement par $t_1(g_1,g_2,x)=(g_1,g_2x)$,
$t_2(g_1,g_2,x)=(g_1g_2,x)$, $t_3(g_1,g_2,x)=(g_2,x)$.
Suivant \cite{GIT}, un faisceau de $\OO_X$-modules
$\EE$ est $G$-équivariant, s'il existe un isomorphisme $\Phi$ : $\sigma^*\EE\simeq
p_{2}^*\EE$, qui vérifie les conditions de cocycles (cf \cite{Kashi-D_flag}), c'est-à-dire
tel que le diagramme suivant commute:
\begin{gather}\label{cond_cocycles}
 \xymatrix{ t_2^*\sigma^*\EE \ar@{->}[rr]^{t_2^*(\Phi)}\ar@{->}[d]^{\wr} &&
               t_2^*p_2^*\EE\ar@{}[d]\ar@{->}[d]^{\wr} \\
                 t_1^*\sigma^*\EE  \ar@{->}[r]^(.4){t_1^*(\Phi)}  &
t_1^*p_2^*\EE=t_3^*\sigma^*\EE\ar@{->}[r]^(.6){t_3^*(\Phi)} &t_3^*p_2^*\EE .      }
\end{gather}

Par exemple, le faisceau $\OO_X$ est $G$-équivariant, ainsi que tous les faisceaux
``différentiels'' sur $X$, comme nous le vérifierons en ~\ref{desc_actions_G}.
Si $\LL$ est un faisceau de $\OO_X$-modules localement libres, $\LL$ est
$G$-équivariant si et seulement si le fibré vectoriel associé à $\LL$ est
muni d'une action de $G$ compatible à l'action de $G$ sur $X$.

\subsubsection{Propriétés}
On a la description suivante des faisceaux de $\OO_X$-modules $G$-équivariants.
Soient $R$ une $V$-algèbre, $g,g'\in G(R)$, $X_R=Spec(R)\times X$,
$\sigma_R$ l'application déduite de $\sigma$ par ce changement de base,
$\LL_R$ le tiré en arrière de $\LL$ sur $X_R$. L'opérateur de translation  $T_g$ :
$X_R\rig X$ donné par $T_g=\sigma\circ (g\times id_{X})$ s'étend canoniquement
en un opérateur toujours noté $T_g$: $X_R\rig X_R$. On dispose d'une
famille d'isomorphismes $\OO_X$-linéaires $\Phi_g$: $T_g^*\LL_R\simeq \LL_R$, vérifiant la
condition
$\Phi_{gg'}=\Phi_{g'}\circ T_{g'}^*(\Phi_g),$
provenant de la condition de cocycle. Ces applications
induisent pour tout ouvert $U$ des applications
$\Phi_{g,U}\,\colon \, \LL_R(g U)\rig \LL_R(U),$
semi-linéaires par rapport aux applications:
$T_g^{-1}:\OO_{X_R}(g U)\simeq \OO_{X_R}(U).$
Par définition, on a l'égalité $T_{g'}^*(\Phi_{g,U})=\Phi_{g,g'U},$
de sorte que
la relation de cocyle se traduit
par $\Phi_{gg',U}=\Phi_{g',U}\circ \Phi_{g,g' U}.$
Si $U$ est le schéma $X_R$, notons $\Phi_g=\Phi_{g,X_R}$, pour
$e\in\Ga(X_R,\LL_R)$, on définit une action à droite de $G(R)$ sur
$\Ga(X_R,\LL_R)$ en posant $g\cdot e=\Phi_g(e)$, grâce à la relation
$\Phi_{gg'}=\Phi_{g}\circ \Phi_{g'}$. Nous verrons en ~\ref{op_diff_inv} que cette action
à droite correspond à une structure de comodule dual sur $\Ga(X_R,\LL_R)$.

Enfin, nous aurons besoin des propositions classiques suivantes.
\begin{sousprop}\label{equiv-dualite} Soit $\LL$ un faisceau de $\OO_X$-modules localement
libres de rang fini $G$-équivariant,
alors $\HH om_{\OO_X}(\LL,\OO_X)$ est un faisceau de $\OO_X$-modules localement libres
de rang fini $G$-équivariant.
\end{sousprop}
\begin{proof} Il suffit d'utiliser le fait que pour un faisceau localement libre de
rang fini, on
a un isomorphisme canonique et fonctoriel:
$$\sigma^*\HH om_{\OO_X}(\LL,\OO_X)
\sta{\sim}{\rig}\HH om_{\OO_{G\times X}}(\sigma^*\LL,\OO_{G\times X}),$$
de sorte que $^t\Phi^{-1}$, c'est-à-dire l'application transposée de l'isomorphisme $\Phi^{-1}$, définit la structure $G$-équivariante sur
$\HH om_{\OO_X}(\LL,\OO_X)$. L'action à droite de $G(R)$ sur les sections globales
de $\HH om_{\OO_{X_R}}(\LL_R,\OO_{X_R})$ est donc aussi donnée par $^t\Phi_{g}^{-1}$ et donc,
si $e$ est une section globale de $\LL_R$, $\varphi $ une section locale de
$\HH om_{\OO_{X_R}}(\LL_R,\OO_{X_R})$, alors pour tout $g $ de $G(R)$ on a
$ (\varphi\cdot g)(e)=\varphi(e\cdot g^{-1}).$
\end{proof}
\begin{sousprop} Soient $\LL$ et $\LL'$ deux faisceaux de $\OO_X$-modules quasi-cohérents,
$G$-équivariant, alors $\LL\ot_{\OO_X}\LL'$ est un faisceau de $\OO_X$-modules
quasi-cohérent, $G$-équivariant.
\end{sousprop}
\begin{proof} On a un isomorphisme fonctoriel:
$\sigma^*(\LL \ot_{\OO_X}\LL') \sta{~}{\rig}\sigma^*\LL \ot_{\OO_{G\times X}}
\sigma^*\LL' .$
\end{proof}
Considérons maintenant un schéma formel $\XX$ tel que pour tout $i\in\Ne$, le
schéma $X_i$ soit muni d'une action du schéma en groupes $G_i$. Nous poserons alors
$\alpha_i: X_i\hrig \XX$, $\gamma_i: G_i\times X_i\hrig G_{i+1}\times X_{i+1}$.
\begin{sousdef} Un faisceau $\EE$ de $\OO_{\XX}$-modules complets pour la topologie
$p$-adique sera dit $G$-équivariant si,
pour tout $i\in\Ne$, le faisceau $\EE_i=\alpha_i^*\EE$ est un faisceau de
$\OO_{X_i}$-modules $G$-équivariants tel que les diagrammes suivants soient commutatifs
$$\xymatrix {p_{2,i}^*\EE_i \ar@{->}[r]^{\sim}_{\Phi_i}&\sigma_{2,i}^*\EE_i \\
 \ga_i^*p_{2,i+1}^*\EE_{i+1}\ar@{->}[u]_{\wr}\ar@{->}[r]^{\sim}_{\ga_i^*\left(\Phi_{i+1}\right)}&
\ga_i^*\sigma_{2,i+1}^*\EE_{i+1}\ar@{->}[u]_{\wr},}.$$
\end{sousdef}
Dans le cas où $\EE$ est un faisceau de $\OO_{\XX}$-modules cohérents,
cette définition équivaut au fait que l'on se donne un isomorphisme $\Phi$:
$\sigma^*\EE\simeq p_2^*\EE$ vérifiant les conditions de cocycles énoncées en
~\ref{cond_cocycles}, ce qui nous amène à
\begin{sousdef} Un faisceau $\EE$ de $\OO_{\XX,\Qr}$-modules cohérents sera dit
$G$-équivariant s'il existe un isomorphisme $\Phi$: $\sigma^*\EE\simeq p_2^*\EE$ vérifiant les conditions de cocycles énoncées en
~\ref{cond_cocycles}.
\end{sousdef}
Si $\EE$ est un faisceau de $\OO_{\XX}$-modules cohérents $G$-équivariant, alors $\EE_{\Qr}:=\EE\otimes\Qr$
est un faisceau de $\OO_{\XX,\Qr}$-modules cohérents $G$-équivariant.

Pour les faisceaux équivariants de $\DD$-modules nous demandons en outre que les
applications soient compatibles avec la structure de $\DD$-module, ce qui nous amène
à la définition suivante,
pour laquelle nous remplaçons les morphismes images inverses au
sens des faisceaux de $\OO_{G\times X}$-modules par les images inverses au sens des
$\Dm_{G\times X}$, suivant les explications qui seront données en ~\ref{operateur-diff}.
\begin{sousdef} Un faisceau $\EE$ de $\Dcm_{\XX,\Qr}$-modules cohérents (resp.
$\Ddag_{\XX,\Qr}$-modules) sera dit
$G$-équivariant s'il existe un isomorphisme $\Phi$: $\sigma^!\EE\simeq p_2^!\EE$ vérifiant
les conditions de cocycles énoncées en
~\ref{cond_cocycles} (où l'on remplace tous les morphismes $s^*$ par $\HH^0 s^!$,
$s^!$ étant l'image inverse au sens des $\DD$-modules).
\end{sousdef}
\subsubsection{Structures $G$-équivariantes à droite}
On considère ici un $S$-schéma $X$ muni d'une action à droite d'un schéma en groupes $G$
donnée par un morphisme de schémas $\tau$: $X\times G \rig X$. Une structure
$G$-équivariante à droite sur un faisceau de $\OO_X$-modules $\EE$ consiste en la donnée d'un
isomorphisme de $\OO_{X\times G}$-modules $\Phi$:$\tau^*\EE \sta{\sim}{\rig} p_1^*\EE$,
vérifiant la condition de cocyle suivante. Soient $t_1,t_2,t_3$: $X\times G \times G \rig
X\times G$ définis par $t_1(x,g_1,g_2)=(xg_1,g_2)$, $t_2(x,g_1,g_2)=(x,g_1g_2)$,
$t_3(x,g_1,g_2)=(x,g_1)$. On demande que $\Psi$ vérifie $t_2^*(\Psi)=t_3^*(\Psi)\circ
t_1^*(\Psi)$. Ainsi, pour toute $V$-algèbre $R$, pour tout $g\in G(R)$ on a un opérateur
de translation $T_g$:$X_R\rig X_R$. Pour tout ouvert
$U$ de $X_R$, on dispose d'applications $\Psi_{U,g}$: $\EE_R(Ug)\rig \EE_R(U)$ semi-linéaires
par rapport aux applications $T_g^{-1}$: $\OO_{X_R}(Ug)\rig \OO_{X_R}(U)$ telles que
$\Psi_{U,g_1g_2}=\Psi_{g_1}\circ \Psi_{g_2,Ug_1}$. Ainsi, l'application
$g\in G(S)\mapsto \Psi_g$ définit une action à gauche de $G(S)$ sur $\Ga(X,\EE)$ pour
tout faisceau $G$-équivariant. Soit enfin l'involution $inv\times id_X$: $G\times X\rig G\times X$,
où $inv$ est l'application de passage à l'inverse $G\rig G$. L'application
$\sigma=inv \times id_X \circ \tau$ définit une action à gauche de $G$ sur $X$. Un faisceau
$\EE$ est $G$-équivariant (pour $\tau$) si et seulement si $\EE $ est $G$-équivariant pour
$\sigma$. L'isomorphisme $\Phi$ définissant cette équivariance est alors donné par
$\Phi=(inv \times id_X)^* (\Psi)$ (et $\Psi=(inv \times id_X)^* (\Phi)$).

\subsection{Action de $G$ sur les faisceaux de parties principales}
Dans cette partie, on suppose que $X$ est un $S$-schéma lisse, muni d'une action à gauche de
$G$ et on note $\Ga_{X,(m)}^n=\Ga_{(m)}^n(\Omega^1_X)$ ~\ref{notGa} et $\Sg^{(m)}_{X}=\Sg^{(m)}(\TT_X)$ de ~\ref{defSm}. 
Nous allons montrer d'abord que les faisceaux de parties principales
$\PP_{X,(m)}^{n}$ sont $G$-équivariants, puis par dualité, que les faisceaux
$\DD^{(m)}_{X,n}$ sont $G$-équivariants comme en~\ref{O-mod-equiv}, et enfin que les faisceaux complétés $\what{\DD}^{(m)}_{\XX}$
sont $G$-équivariants.
\label{gequiv_diff}

Enonçons à présent le résultat de cette sous-section. Pour l'algèbre symétrique $\Sg^{(m)}_{X}$ de niveau $m$ et les propriétés de base de cette algèbre, nous nous référons à \cite{Hu1}.
\begin{prop}\label{prop-equiv} \item[(i)]  Les faisceaux $\PP_{X,(m)}^{n}$ et $\Ga_{X,(m)}^n$
\label{gequiv_diffprop}
sont des faisceaux de $\OO_X$-modules $G$-équivariants.
\item[(ii)]Les faisceaux $\DD^{(m)}_{X,n}$ et $\Sg^{(m)}_{X}$
sont des faisceaux de $\OO_X$-modules $G$-équivariants.
\item[(iii)]Les faisceaux $\what{\DD}^{(m)}_{\XX}$ sont des faisceaux de
$\OO_{\XX}$-modules $G$-équivariants.
\end{prop}
\begin{proof} L'assertion (ii) provient de (i) par passage à la dualité. Ensuite, par (ii) les faisceaux
$\DD^{(m)}_{X_i}$ sont $G$-équivariants de façon compatible pour différentes valeurs de $i$, de sorte
que les faisceaux $\what{\DD}^{(m)}_{\XX}$ sont $G$-équivariants par passage à la limite projective sur $i$.
Donc (ii) entraîne (iii). Dans la suite nous montrons donc (i).
Le cas des algèbres
graduées $\Ga_{X,(m)}^n$ se traitera comme celui des faisceaux $\PP_{X,(m)}^n$
(en utilisant $q'_2$ au lieu de $q_2$) ou encore la structure de $G$-equivariance
de $\Omega^1_X$. Dans la suite, on se restreint donc à la $G$-linéarité des
faisceaux $\PP_{X,(m)}^n$.

Soit $d\sigma$ le morphisme obtenu par fonctorialité: $\sigma^*\PP_{X,(m)}^n\rig
\PP_{G\times X,(m)}^n$. On définit alors $\Phi=q_2\circ d\sigma$:
$\sigma^*\PP_{X,(m)}^n\rig p_2^*\PP_{X,(m)}^n$. Le
morphisme $\Phi$ est un $m$-PD-morphisme.
L'algèbre $\PP_{X,(m)}^0$ est isomorphe à $\OO_X$. Pour voir que $\Phi$
est un isomorphisme, il suffit de le vérifier
en les points fermés de $G\times X$ d'après~\ref{subsection-lem_techniques}.

Soient $g$ un point de $G$, de corps résiduel $k(g)$, $ev_g$ l'évaluation:
$\OO_G\rig k(g)$ (qui à $l\in\OO_G$ associe $l(g)$), $i_g$ l'immersion fermée:
$g\times X \hrig G\times X$, $\sigma_g=\sigma \circ i_g$,
$p_{2,g}=p_2\circ i_g$. Le morphisme $i_g^*\Phi$ est obtenu comme morphisme composé
$$\sigma_g^* \PP_{X,(m)}^n \sta{i_g^*d\sigma}{\rig}i_g^*\PP_{G\times X,(m)}^n
\sta{i_g^*q_2}{\rig}p_{2,g}^*\PP_{X,(m)}^n\simeq k(g)\ot_V\PP_{X,(m)}^n.$$

Il suffit alors de constater que le morphisme $i_g^*\Phi$
 co\"\i ncide avec le morphisme
$d\sigma_g$, qui est un isomorphisme d'inverse
$d{\sigma}_{g^{-1}}$. Vérifions pour cela que le diagramme ci-dessous
est commutatif
$$\xymatrix{\sigma^*_g\PP^n_{X,(m)} \ar@{->}[r]^{i_g^*d{\sigma}}\ar@{=}[d]^{}&
i_g^*\PP^n_{G\times X,(m)} \ar@{->}[r]^{i_g^*q_2} &
i_g^*q_2^*\PP^n_{X,(m)} \ar@{=}[d]\\
      \sigma^*_g\PP^n_{X,(m)} \ar@{->}[rr]^{d\sigma_g}
&  & \PP^n_{ X,(m)}\ot_V k(g) .}$$
Soient $u\in \OO_X$, $\tau=1\ot u - u\ot 1\in \PP^n_{ X,(m)}$.
Ecrivons
$\sigma^{\sharp}(u)=\sum_l k_l\ot f_l$, avec $k_l\in\OO_G$ et $f_l\in\sigma^{-1}\OO_X$.
Alors
\begin{align*} d\sigma_g(1\ot\tau)&=\sum_l k_l(g)(1\ot f_l -f_l \ot 1),\\
d\sigma(1\ot \tau) &= \sum_l (1\ot f_l)(1\ot k_l-k_l\ot 1)+ \sum_l (k_l\ot 1)
(1 \ot f_l -f_l \ot 1) \,\in \PP_{G\times X,(m)}^{n}, \\
 i_g^*d\sigma(1\ot \tau) &= \sum_l f_l (1\ot k_l-k_l\ot 1) + \sum_l k_l(g) (1 \ot f_l -f_l \ot 1), \\
i_g^*q_2  \circ i_g^*d\sigma(1\ot \tau) &= \sum_l k_l(g) (1 \ot f_l -f_l \ot 1) \\
                                  &= d\sigma_g(1\ot\tau).
\end{align*}
Les morphismes $d\sigma_g$ et $i_g^*q_2  \circ i_g^*d\sigma$ sont des $m$-PD-morphismes
$\OO_X\ot_V k(g)$-linéaires qui co\"\i ncident sur les éléments $1\ot \tau$.
Comme les éléments $1\ot\tau$ engendrent $\sigma^*_g\PP^n_{X,(m)}$ comme
$m$-PD-algèbre, on voit ainsi que $d\sigma_g$ et $i_g^*q_2  \circ i_g^*d\sigma$ sont égaux.

La condition de cocycle $t_2^*(\Phi)=t_3^*(\Phi)\circ t_1^*(\Phi)$ se vérifie
aussi sur les fibres $g_1\times g_2 \times X$ où $g_1,g_2$ sont deux points de $G$,
définis sur une certaine algèbre $A$, et revient
à montrer que $\Phi_{g_1g_2}=\Phi_{g_2}\circ \sigma_{g_2}^*\Phi_{g_1}$, ce qui résulte de
la fonctorialité rappelée en ~\ref{parties_principales} puisque $\Phi_{g_1}=d\sigma_{g_1}$.
\end{proof}
\section{Algèbres de distributions arithmétiques à un niveau fini}
\label{sect-distribution-algebras}
\subsection{Définition, propriétés à un niveau fini}\label{subsection-def_prop}
Reprenons les notations de l'introduction : $V$ est une $\Ze_{(p)}$-algèbre noetherienne et $G$ est un schéma en groupes affine et lisse
sur $S=Spec \,V$ de dimension relative $N$. Soit $e$ : $ S \hrig G$ la section unité.
Notons $\II$ le faisceau d'idéaux noyau du morphisme de $V$-algèbres
$\varep_G: \OO_G\rig e{*}\OO_S$, $V[G]=\Ga(G,\OO_G)$, $I=\Ga(G,\II)=Ker ( V[G]\thrig V).$
 Soient $P^n_{(m)}(G)$ la $m$-PD-enveloppe à puissances divisées partielles de niveau $m$ 
et d'ordre $n$, de l'idéal $I$ de $V[G]$.
 Comme $G$ est lisse sur
$S$, l'idéal
$\II$ est localement régulier sur $S$, c'est-à-dire que $\II/\II^2$ est localement libre sur $S$.

Soit $S'=Spec \,V'$ un ouvert affine de $S$ sur lequel $\II/\II^2$ est libre de base 
$\overline{t}_1,\ldots,\overline{t}_N$, avec $t_i\in V'\ot_V I$ pour tout $i$. On dira dans ce cas que 
les $t_i$ forment une suite régulière de paramètres de $\II(V')=V'\ot_V I$.
On considère le 
groupe $G_S'=G\times_S S'$ obtenu par changement de base du groupe $G$.
Par 1.5.3 de \cite{Be1}, 
les $V'$-algèbres $P^n_{(m)}(G_{S'})$ sont des $V'$-modules libres de base les $\ut^{\{\uk\}}$
pour $|\uk|\leq n$ en reprenant les notations de ~\ref{formulesmPD}.
La formation des $m$-PD-algèbres 
commute aux changements de base plats (1.4.6 de \cite{Be1}), ce qui implique que 
$P^n_{(m)}(G_{S'})\simeq V'\ot_V P^n_{(m)}(G)$ et nous
donne finalement la 
\begin{prop}\label{desc_loc_P} Pour $S'=Spec \,V'$ un ouvert de $S$, et $t_1,\ldots,t_N$ 
 une suite régulière de paramètres de $V'\ot_V I$, on a un isomorphisme de $V'$-modules
                    $$ V'\ot_V P^n_{(m)}(G) \simeq \bigoplus_{|\uk|\leq n} V' \ut^{\{\uk\}}.$$
\end{prop}
Une conséquence de cet énoncé est que l'on peut faisceautiser la construction des algèbres 
$P^n_{(m)}(G)$ en des faisceaux sur $S$, cohérents et localement libres.

 Dans le cas d'un AVDC (ou d'un quotient), les algèbres $P^n_{(m)}(G)$ sont des $V$-modules libres
de rang fini. Si $\kappa$ est le point fermé de $V$, $e(\kappa)$ est un point fermé de $G$, encore noté 
$\kappa$, $\OO_{G,\kappa}$ désigne alors l'anneau local de $G$ en ce point, et $\II_{\kappa}$ le localisé 
de $I$ en $\kappa$.
On dispose alors de la
\begin{prop}
 Si $V$ est un AVDC, ou un quotient d'un AVDC, de point fermé $\kappa$,
alors l'algèbre
$P^n_{(m)}(G)$ est la $m$-PD-enveloppe de $\II_{\kappa}\subset\OO_{G,\kappa}$. En particulier
ces algèbres sont des $V$-modules libres et en utilisant des paramètres $t_1,\ldots,t_N$ de $\II_{\kappa}$,
on retrouve la description ci-dessus de ~\ref{desc_loc_P}.
\end{prop}
\begin{proof} Comme $V$ est un AVCD ou un quotient d'un AVCD, 
alors $V$ est local et principal, $I/I^2$ est un $V$-module libre de rang fini et on peut appliquer ~\ref{desc_loc_P}.
 \end{proof}
Il résulte aussi de ~\ref{desc_loc_P} que 
$$V \simeq P^0_{(m)}(G).$$
Notons 
$Lie(G)=Hom_V(I/I^2,V).$ Sur un ouvert de $S'=Spec \,V'$ sur lequel $\II$ admet une suite régulière de
générateurs $t_1,\ldots,t_N$, le $V'$-module $V'\ot_V Lie(G)$ est libre de base $\xi_1,\ldots,\xi_N$,
que l'on définit comme la base duale
des $t_1,\ldots,t_N.$

En particulier dans le cas où $V$ est un AVDC (ou un quotient), l'algèbre $Lie(G)$ est un $V$-module libre
de base $\xi_1,\ldots,\xi_N$.

En procédant comme pour
~\ref{desc-P1diff}, on voit qu'on a
 des isomorphismes
\begin{gather} P^1_{(m)}(G) \sta{\sim}{\lrig} V \bigoplus I/I^2 \label{desc-P1}.
\end{gather}
En effet, cet isomorphisme se montre par un calcul local sur un ouvert sur lequel $\II$ est muni d'une
suite régulière de générateurs. 

Nous introduirons le niveau $m$ en indice de l'exposant dans cette notation, lorsque nous
voudrons préciser le niveau.

Enfin, on note $\rho_m$ l'application canonique
\begin{gather}\label{def_rho}
\rho_m\,\colon\,V[G]\rig P^n_{(m)}(G),
\end{gather}
qui munit $P_{(m)}^{n}(G)$ d'une structure de $V[G]$-module.

\vskip5pt
Pour $m'\geq m$, d'après la propriété universelle des algèbres à puissances divisées,
on trouve des homomorphismes d'algèbres filtrées 
$\psi_{m,m'}$
\begin{equation}\label{psi}\psi_{m,m'}\,\colon \, P_{(m')}(G)\rig P_{(m)}(G),\end{equation} qui donnent par passage au quotient
des homomorphismes d'algèbre
$$\psi_{m,m'}^n \,\colon \,P^n_{(m')}(G)\rig P^n_{(m)}(G).$$

Si $I$ admet sur $S$ une suite régulière de générateurs $t_1,\ldots,t_N$ (par exemple si $V$ est un AVDC ou un quotient), on a:
$$\psi_{m,m'}^n(\ut^{\{\uk\}_{(m')}})=
\frac{q_{\uk}^{(m)}!}{q_{\uk}^{(m')}!}\ut^{\{\uk\}_{(m)}}.$$

\vskip5pt
On définit maintenant les algèbres de distributions de niveau $m$: la structure
d'algèbre est expliquée après la définition et provient, comme dans le cas classique,
de la structure de groupe. Commençons par définir les faisceaux de distributions sur $S$.
\begin{defi}\label{defi-distrib} Les distributions de niveau $m$ et
d'ordre $n$ de $G$ sont:
$$D_n^{(m)}(G):=Hom_V(P^n_{(m)}(G),V)$$
et le $V$-module des distributions de niveau $m$ est:
$$D^{(m)}(G):=\varinjlim_n D_n^{(m)}(G).$$
\end{defi}
Remarquons que la relation~\ref{desc-P1} nous donne les isomorphismes
\begin{gather}\label{desc-D1} A_m\,\colon\, D_1^{(m)}(G) \sta{\sim}{\rig} V \bigoplus Lie(G),
\overline{A}_m \,\colon\, \gr_1 D^{(m)}(G) \sta{\sim}{\rig} Lie(G),
\end{gather}
obtenue par dualité à partir de $\overline{A}_m$. 

On dispose d'une action de $D^{(m)}(G)$ sur $V[G]$. Si $P\in D^{(m)}(G)$,
l'action de $P$ est définie par
\begin{gather}\label{action_VG} \xymatrix @R=0mm{ V[G] \ar@{->}[r]^{\rho_m}& P^n_{(m)}(G) \ar@{->}[r]^{P}& V \\
f \ar@{{|}->}[rr] &  & P(\rho_m(f)).}
\end{gather}

Pour un entier $m'\geq m$, les morphismes d'algèbres $\psi_{m,m'}^n$ de \ref{psi} donnent par passage à
la dualité des applications linéaires $$\Phi_{m,m'}^n \,\colon \, D^{(m)}_n(G)\rig D^{(m')}_n(G) \textrm{ et } 
\Phi_{m,m'}\,\colon \, D^{(m)}(G)\rig D^{(m')}(G).$$ On a $\Phi_{m',m''}\circ \Phi_{m,m'}=\Phi_{m,m''}$ pour $m''\geq m'\geq m$.
Pour comparer les algèbres $D^{(m)}(G)$ avec des distributions classiques on définit
\begin{gather}
\;Dist_n(G):=D_n^{(\infty)}(G):=Hom_V(V[G]/I^{n+1},V).
\end{gather}
Alors, $$Dist(G):=\varinjlim_n Dist_n(G)$$ est
l'algèbre de distributions classiques d'un schéma en groupes sur
$S$, définie en II.\S 4.6.1 de \cite{Demazure_gr_alg} et $Dist_n(G)$ est l'espace de
distributions d'ordre $n$. 

On voit, en procédant comme en \ref{desc-D1} qu'on a
l'isomorphisme

$$Dist_1(G)\simeq Lie(G)\bigoplus V .$$  L'application canonique
$V[G]/\II^{n+1} \rig P^n_{(m)}(G)$ induit une application linéaire $$\Phi_{m,\infty}: D^{(m)}(G)\rig Dist(G)$$
compatible avec les applications $\Phi_{m,m'}$.
\vskip5pt

Nous donnons maintenant un énoncé de changement de base. Soient $V'$ une $V$-algèbre plate,
 $S'=Spec\, V'$, $G_{S'}=G\times_S S'$. Alors on a la
\begin{prop} \label{prop-chbase} On dispose d'isomorphismes canoniques

$D_n^{(m)}(G_{S'}) \simeq D_n^{(m)}(G)\ot_V V', \textrm{ resp. } D^{(m)}(G_{S'}) \simeq
D^{(m)}(G)\ot_V V'.$
\end{prop}

\begin{proof} Sous nos hypothèses, d'après 1.5.3 de \cite{Be1}, la formation des algèbres
$P^n_{(m)}(G)$ commute aux changements de base plats et les faisceaux $\OO_S\ot_V P^n_{(m)}(G)$ sont des 
$\OO_S$-modules cohérents localement libres de rang fini sur $S$, donc $\OO_S\ot_V D_n^{(m)}(G)$ est isomorphe  
à $\HH om_{\OO_S}(\OO_S\ot_V P^n_{(m)}(G),\OO_S)$, ce qui permet de montrer l'énoncé.
\end{proof}

\vskip5pt

Supposons que $I$ admette sur $S$ une suite régulière de générateurs $t_1,\ldots,t_N$.
Notons $\uxi^{\la\uk\ra}$ les éléments de la base duale de la famille
$\ut^{\{\uk\}}=t_1^{\{k_1\}}\cdots t_N^{\{k_N\}}, |\uk|\leq N$, $\uxi^{[\uk]}:=\uxi^{\la\uk\ra_{\inft}}$ (de sorte que les éléments $\uxi^{[\uk]}$
forment la base duale de la famille $t_1^{k_1}\cdots t_N^{k_N})$. Soit $m\in\Ne\cup\{\infty\}$.
On obtient par construction la proposition suivante. 
\begin{prop}\label{prop-basiselementsII_f}
 \be \item[(i)] Comme $V$-module,  $D_n^{(m)}(G)$ est libre de base les éléments
$\uxi^{\la\uk\ra}$ avec $|\uk|\leq n$.
\item[(ii)] On a les relations
\begin{gather} \label{relations_xi}\uk!\uxi^{[\uk]}=\uxi^{\la\uk\ra_{0}}\\
 \uxi^{\la\uk\ra_{m}}=\frac{q_{\uk}^{(m)}!}{q_{\uk}^{(m+1)}!} \uxi^{\la\uk\ra_{m+1}}\\
\text{si }\uk=(k_1,\ldots,k_N),\text{ et,}\forall i\leq N,\;k_i\leq p^m,\;
\uxi^{\la\uk\ra_{m}}=\uxi^{[\uk]}.
\end{gather}
\ee
\end{prop}
On notera tout simplement $\xi_l=\xi^{[1]}_l=\xi^{\la 1\ra_{m}}_l$.

En particulier, si $V$ est un AVDC ou d'un quotient d'un AVDC, on a la
\begin{prop}\label{prop-basiselementsII}
Sous l'hypothèse ci-dessus, le module $D_n^{(m)}(G)$
  est libre de base les éléments
$\uxi^{\la\uk\ra}$ avec $|\uk|\leq n$. Ces éléments satisfont la relation (ii) de la proposition précédente.
\end{prop}
On déduit de la dernière formule de (ii) que les applications $\Phi_{m,\infty}$ induisent
un isomorphisme linéaire
$$ \varinjlim_m D^{(m)}(G)\stackrel{\simeq}{\longrightarrow} D ist(G).$$

Il existe un morphisme canonique $D^{(0)}(G) \rig Dist(G)$. 

\vskip5pt

Supposons que $V$ est un AVCD, et soit $t_1,\ldots,t_N$ une suite régulière de générateurs de $I$,
alors ce morphisme envoie $\uxi^{\la\uk\ra_{m}}$ sur $q_{\uk}^{(m)}!\uxi^{[\uk]}$.
Comme
$Lie(G)=H{om}_V(I/I^2,V)$ les éléments $\xi_l$ forment une
$V$-base de $Lie(G)$.  Il suit de la première formule de (ii) de \ref{prop-basiselementsII} que $D^{(0)}(G)\otimes K=
Dist(G)\otimes K$ et donc $D^{(m)}(G)\otimes K=Dist(G)\otimes K$ pour tout $m$. Ce qui montre que $D^{(m)}(G)$ est une forme entière
 de $Dist(G)\otimes K= U(Lie(G)\otimes K)$, où $U(Lie(G)\otimes K)$ est l'algèbre enveloppante universelle 
de $Lie(G)\otimes K$. En général, $I$ est localement régulier sur $S$ et toutes les contructions
commutent aux extensions plates de la base d'après~\ref{prop-chbase}, ce qui montre que  $D^{(m)}(G)$ est une forme entière
 de $Dist(G)\otimes K= U(Lie(G)\otimes K)$ si $V$ est un AVCD.

\vskip5pt

Supposons de nouveau que $I$ admette sur $S$ une suite régulière de
générateurs $t_1,\ldots,t_N$. La base des $\uxi^{\la\uk'\ra}$ est duale de
celle des $\ut^{\{\uk'\}}$, et par définition de l'action de $D^{(m)}(G)$ sur $V[G]$, on a la
\begin{prop} Soit $f\in V[G]$, on a la formule de Taylor
\begin{gather}  \rho_m(f)=\sum_{|\uk|\leq n} \label{f-taylor1}
\uxi^{\la\uk\ra}(f)\ot \ut^{\{\uk\}}.
\end{gather}
\end{prop}

\vskip5pt

Nous donnons maintenant un énoncé important sur la structure d'algèbre de $D^{(m)}(G)$ (sans supposer que $I$ admet une suite 
régulière de générateurs sur $S$).
\begin{prop} \label{prop-commutativite}

  Le module $D^{(m)}(G)$ est une $V$-algèbre filtrée par les sous-modules $D_n^{(m)}(G)$.
L'algèbre graduée associée à cette filtration est commutative.
\end{prop}
\begin{proof}
 
L'application $\mu^{\sharp}$ induit une application $V[G]\rig
V[G]\ot_V V[G]$, qui envoie l'idéal $I$ dans $I\ot V[G]+ V[G]\ot I$
d'après I 7.7 de \cite{Jantzen}. D'autre part, en procédant comme en
2.1.3 de \cite{Be1}, on voit qu'il existe une unique $m$-PD-structure sur
$P^n_{(m)}(G)\ot P^{n'}_{(m)}(G)$ telle que $I\ot P^{n'}_{(m)}(G)+
P^n_{(m)}\ot I$ soit un $m$-PD-idéal. On considère alors l'application
$\delta\circ\mu^{\sharp}$:
$$V[G]\sta{\mu^{\sharp}}{\rig} V[G]\ot_VV[G] \sta{\delta}{\rig }
P^n_{(m)}(G)\ot_V P^{n'}_{(m)}(G),$$
où $\delta$ est l'application canonique $\rho_m\ot\rho_{m}$:
$V[G]\ot_V V[G] \rig P^n_{(m)}(G)\ot_V P^{n'}_{(m)}(G)$.

Comme $\mu^{\sharp}$ applique tout élement de $I$ dans un $m$-PD-idéal, cette application se factorise d'une
unique façon par un $m$-PD-morphisme $\delta^{n,n'}$
\begin{gather}\label{def-delta}
\delta^{n,n'} \,\colon \,
P_{(m)}(G)\sta{\delta^{n,n'}}{\rig}P^n_{(m)}(G)\ot P^{n'}_{(m)}(G).
\end{gather}
On notera éventuellement $\delta^{n,n'}_{(m)}$ lorsqu'on voudra préciser le niveau
$m$ de ce morphisme.

Soient $(u,v)\in D^{(m)}_n(G)\times D^{(m)}_{n'}(G)$, on définit
$u\cdot v$ comme la composée
$$ u\cdot v\colon \, P_{(m)}(G)\sta{\delta^{n,n'}}{\rig}
P^n_{(m)}(G)\ot P^{n'}_{(m)}(G)\sta{u\ot v}{\rig}V.$$

Il nous reste à montrer que $u\cdot v$ s'annule sur $I^{\{n+n'+1\}}$ et définit donc une
application : $P^{n+n'}_{(m)}(G)\rig V$. La vérification est locale sur $S $ et
nous pouvons supposer que $I$ admet sur $S$ une suite régulière de générateurs $t_1,\ldots,t_N$.
 Nous utilisons les notations de ~\ref{prop-basiselementsII_f}.
Il nous faut montrer que
$\delta^{n,n'}$ s'annule sur
les éléments $\ut^{\{\uk\}}=t_1^{\{k_1\}}\cdots t_r^{\{k_r\}}$ avec $|\uk|\geq n+n'+1$, en utilisant
les notations ~\ref{descript_loc_pn}, et passe au quotient en
une application: $P^{n+n'}_{(m)}(G)\rig P^{n}_{(m)}(G)\ot P^{n'}_{(m)}(G)$. Toujours
d'après I 7.7 de \cite{Jantzen}, on peut écrire
$$\mu^{\sharp}(t_i)=1\ot t_i+t_i \ot 1 +\sum_{s=1}^{h_i} a_{i,s} \ot b_{i,s},$$
avec $a_{i,s}$ et $b_{i,s}$ des éléments de $I$. Dans le cas du groupe additif
${\bf G}_a$, on a $\mu^{\sharp}(t_i)=1\ot t_i+t_i \ot 1$.

En appliquant les formules~\ref{subsubsection-formules}, on trouve
$$ \delta^{n,n'}(t_i^{\{k_i\}})=\sum_{\alpha_{i}=0}^{k_i}
\crofrac{k_i}{\alpha_{i}} t_i^{\{\alpha_{i}\}}\ot t_i^{\{k_i-\alpha_{i}\}} +
  \gamma_{k_i} \quad {\rm avec} \quad \gamma_{k_i}\in \sum_{s+t\geq k_i+1}
I^{\{s\}}\ot I^{\{t\}}.$$
Dans cette somme, les termes correspondant à des $\alpha_i\geq n+1$
ou tel que $k_i-\alpha_i\geq n'+1$ ou à des éléments $\gamma_{k_i}$ de
$I^{\{s\}}\ot I^{\{t\}}$ avec $s\geq n+1 $ ou $t\geq n'+1$ sont en
fait nuls.
De plus, la première somme est un tenseur symétrique.
En effectuant le produit, cela donne
$$\delta^{n,n'}(\ut^{\{\uk\}})= \sum_{\ual\leq \uk}\crofrac{\uk}{\ual}
\ut^{\{\ual\}}\ot \ut^{\{\uk-\ual\}} +\gamma_{\uk}\quad {\rm avec} \quad \gamma_{\uk}\in
\sum_{s+t\geq |\uk|+1}
I^{\{s\}}\ot I^{\{t\}}.$$
Dans cette somme, les termes correspondant à des $|\ual|\geq n+1$
ou tels que $|\uk-\ual|\geq n'+1$ ou à des éléments $\gamma_{\uk}$ de
$I^{\{s\}}\ot I^{\{t\}}$ avec ($s\geq n+1 $ ou $t\geq n'+1$) sont
nuls.

Montrons maintenant que $uv$ est d'ordre $\leq n+n'$. Il suffit pour cela de montrer
que $uv(\ut^{\{\uk\}})=(u \ot v)\circ \delta^{n,n'}(\ut^{\{\uk\}})=0$ pour tout $\uk$ tel que $|\uk|=n+n'+1$.
Reprenons la formule précédente. L'élément $\delta^{n,n'}(\ut^{\{\uk\}})$ s'écrit comme
somme d'éléments des idéaux $I^{\{s\}}\ot I^{\{t\}}$ avec
$s+t\geq n+n'+1$, i.e. $s\geq n+1$ ou $t\geq n'+1$. Le résultat est donc bien nul puisque
$u$ s'annule sur $I^{\{n+1\}}$ et $v$ sur $I^{\{n'+1\}}$.

Montrons que $uv-vu$ est d'ordre $\leq n+n'-1$. Il faut montrer que le résultat
du calcul suivant est nul pour un $\uk$ fixé tel que $|\uk|=n+n',$
\begin{multline}(uv-vu)(\ut^{\{\uk\}})=(u \ot v)\circ \delta^{n,n'}(\ut^{\{\uk\}})-
(v \ot u)\circ \delta^{n',n}(\ut^{\{\uk\}})\\
=(u \ot v)\left( \sum_{\ual\leq \uk}\crofrac{\uk}{\ual}
\ut^{\{\ual\}}\ot \ut^{\{\uk-\ual\}} +\gamma_{\uk}\right) -
(v \ot u) \left( \sum_{\ual\leq \uk}\crofrac{\uk}{\ual}
\ut^{\{\ual\}}\ot \ut^{\{\uk-\ual\}} +\gamma'_{\uk} \right).
\end{multline}
On remarque que tous les éléments $(u\ot v)(\gamma_{\uk})$, et $(u\ot v)(\gamma_{\uk'})$
sont nuls pour tous les
multi-indices $k$ et $k'$ intervenant dans les sommes précédentes car chaque élément
$\gamma_{\uk}$ ou $\gamma_{\uk'}$ est dans un idéal du type
$I^{\{s\}}\ot I^{\{t\}}$, avec $s\geq n+1$ ou $t\geq n'+1$. Il vient
donc finalement
$$  (uv-vu)(\ut^{\{\uk\}})=\sum_{\ual\leq \uk}\crofrac{\uk}{\ual}u\left(\ut^{\{\uk\}}\right)
v\left(\ut^{\{\uk\}-\{\ual\}}\right)-v\left(\ut^{\{\ual\}}\right)u\left(\ut^{\{\uk\}-\{\ual\}}\right)=0.$$

L'associativité du produit provient, comme dans le cas classique, de l'associativité de
la loi de groupe et de la propriété universelle des algèbres à puissances divisées. \end{proof}
\begin{prop}L'application $\Phi_{m,m'}: D^{(m)}(G)\rig D^{(m')}(G)$ est un homomorphisme d'algèbres filtrées pour
$m,m'\in\Ne\cup\{\infty\}$. De plus, on a l'égalité $\varinjlim_m D^{(m)}(G) = Dist(G)$ comme algèbres filtrées.
\end{prop}
\begin{proof}
Par la propriété universelle des $m$-PD-enveloppes d'un idéal, il existe une unique flèche
$P_{(m')}(G)\rig P_{(m)}(G)$ pour $m'\geq m$ qui rend commutatif le diagramme
suivant
$$\xymatrix{ P_{(m')}(G)\ar@{->}[d]\ar@{->}[r]^(.4){\delta^{n,n'}_{(m')}} & P^n_{(m')}(G)\ot
P^{n'}_{(m')}(G)\ar@{->}[d]\\
P_{(m)}(G)\ar@{->}[r]^(.4){\delta^{n,n'}_{(m)}} & P^n_{(m)}(G)\ot P^{n'}_{(m)}(G).}$$
On en tire formellement la proposition.
\end{proof}
\vspace{+2mm}
Dans l'énoncé suivant nous donnons quelques propriétés de fonctorialité de
$D^{(m)}(\cdot)$.
\begin{prop}\label{prop-functoriality} Soient $H,G$ deux schémas en groupes sur $S$.

\be

\item[(i)] Soit $f: H\rig G$ un morphisme de $S$-schémas qui est compatible avec les
sections unité $\varepsilon_H,\varepsilon_G$, i.e. $f\circ \varep_H=\varep_G$. Il induit un homomorphisme $D^{(m)}(f)$
 de $V$-modules filtrés $D^{(m)}(H)\rig D^{(m)}(G)$. Si $f$ est un morphisme de schémas en
groupes, alors $D^{(m)}(f)$ (resp. $D^{(m)}(f)$) est multiplicatif.
    \item[(ii)] Il existe un isomorphisme de $V$-algèbres filtrées
$D^{(m)}(H)\otimes_V D^{(m)}(G)\simeq D^{(m)}(H\times G)$ (resp.
$D^{(m)}(H)\otimes_V D^{(m)}(G)\simeq D^{(m)}(H\times G)$) où le terme de gauche est muni
de la filtration produit tensoriel.
\ee

\end{prop}
\begin{proof} 
Soient $f^\sharp: V[G]\rig V[H]$ le morphisme d'algèbres associé à $f$, $I_H$ et $I_G$ les idéaux de 
$V[H]$ et $V[G]$ définissant les sections unités de $G$ et de $H$. L'hypothèse faite sur $H$ implique 
que $f^{\sharp -1}(I_H)=I_G$, de sorte que $f$ induit un
morphisme $P^n_{(m)}(G)\rig P^n_{(m)}(H)$. En dualisant on trouve une application
$V$-lineaire $D_n^{(m)}(H)\rig D_n^{(m)}(G)$. On obtient $D^{(m)}(f)$ en passant à la
limite sur $n$. Cela montre (i).
Maintenant, d'après (i), les morphismes $H\times e_G$ and $e_H\times
G$ induisent des homomorphismes $D^{(m)}(G)\rig D^{(m)}(H\times G)$ et $D^{(m)}(H)\rig
D^{(m)}(H\times G)$. Ces derniers homomorphismes induisent un homomorphisme
$h: D^{(m)}(H)\otimes_V D^{(m)}(G)\rig D^{(m)}(H\times G)$. La question de savoir si ce morphisme est
un isomorphisme (filtré) est locale. Quitte à restreindre $S$, on peut supposer que $I_G$ et $I_H$ 
sont engendrés par une suite régulière de paramètres.
Choisissons une suite
régulière de générateurs $t_1,...,t_N$ de l'ideal $I_{H\times G}$ constituée de
suites regulières pour les idéaux $I_{H}$ et $I_{G}$. Alors la
proposition \ref{prop-basiselementsII} implique que $h$ est un isomorphisme filtré.
\end{proof}

Donnons maintenant deux exemples. D'abord dans le cas du groupe
additif $\mathbf{G}_a^N$, muni des coordonnées
$t_1,\ldots, t_N$, dont la loi de groupe est définie par $\mu^{\sharp}(t_i)=1\ot t_i +t_i
\ot 1,$ pour tout $1\leq i \leq N$. On note $\uder^{\langle \uk \rangle}$ les opérateurs
différentiels relatifs au choix des $t_i$.
On a alors
\begin{prop}\label{DmGa}Dans $D^{(m)}(\mathbf{G}_a^N)$, on a la relation
$$\uxi^{\langle \uk' \rangle}\cdot \uxi^{\langle \uk'' \rangle}=
 \crofrac{\uk'+\uk''}{\uk'}\uxi^{\langle \uk'+\uk'' \rangle}.$$
\end{prop}
\begin{proof} Il nous faut en effet calculer
\begin{multline} (\uxi^{\langle \uk' \rangle}\cdot \uxi^{\langle \uk'' \rangle})(\ut^{\{\ur\}})=
(\uxi^{\langle \uk' \rangle}\ot \uxi^{\langle \uk'' \rangle})
\left( \sum_{\ual\leq \ur}\crofrac{\ur}{\ual}
\ut^{\{\ual\}}\ot \ut^{\{\ur-\ual\}}\right)
=\delta_{\ur,\uk'+\uk''}\crofrac{\uk'+\uk''}{\uk'},
\end{multline}
ce qui donne l'énoncé recherché.
\end{proof}

Dans le cas $N=1$ on en déduit les formules suivantes pour un entier $k\geq 1$
\begin{align*} \text{si}\:k\neq 0\:[p^m],\quad & k\xi^{\langle k \rangle}=\xi\xi^{\langle
k-1 \rangle}\\
\text{si}\:k= 0\:[p^m],\quad & p^m\xi^{\langle k \rangle}=\xi\xi^{\langle
k-1 \rangle},
\end{align*}
dont on déduit la formule suivante pour tout $k\geq 0$
$$\frac{k!}{q_k!}\xi^{\langle k \rangle}=\xi^k.$$
On vérifie au passage, que si $f\in V[T_1,\ldots,T_N]$,
\begin{gather}\label{formulDGA}
\xi^{\langle k \rangle}(f)=\uder^{\langle k \rangle}(f)(0).
\end{gather}
En particulier, dans le case $m=0$, on a $\xi^{\langle k \rangle_{(0)} }=\xi^k$ pour tout $k$.

\vskip5pt

Maintenant, dans $\mathbf{G}_m$, qu'on identifie à $Spec\,V[T,T^{-1}]$, le faisceau
d'idéaux $\II$ est engendré par $\tau=T-1$. La loi de groupe
est donnée par $\mu^{\sharp}(\tau)=\tau\ot 1+1\ot \tau+\tau\ot\tau.$
\begin{prop}\label{DmGm} Dans $D^{(m)}(\mathbf{G}_m)$, on a les relations
\begin{gather*}\xi^{\langle k' \rangle}\cdot \xi^{\langle k'' \rangle}=
\sum_{ max\{k',k''\}\leq l\leq k'+k''}\frac{q_{k'}!\,q_{k''}!\,l!\,}{q_l!\,(k'+k''-l)!\,(l-k')!\,(l-k'')!\,}
\xi^{\langle l\rangle}.\\
\end{gather*}
\end{prop}
\begin{proof}
Calculons, en appliquant ~\ref{subsubsection-formules},
\begin{eqnarray*}\mu^{\sharp}(\tau^{(l)})&=&\sum_{r=0}^l \acfrac{l}{r}(\tau\ot 1+1\ot
\tau)^{\{r\}}\tau^{\{l-r\}}\ot \tau^{l-r}\\
  & = & \sum_{r=0}^l\sum_{s=0}^r q_{l-r}!\acfrac{l}{r}\acfrac{r}{s}\acfrac{s+l-r}{s}\acfrac{l-s}{r-s}.\end{eqnarray*}
Cela nous permet de calculer $\xi^{\langle k' \rangle}\ot \xi^{\langle k''
\rangle}\delta^{k',k''}(\mu^{\sharp}(\tau^{(l)}))$ et nous donne l'énoncé.
\end{proof}
On en déduit les formules suivantes pour un entier $k\geq 1$
\begin{align*} \text{si}\:k\neq 0\:[p^m],\quad & k\xi^{\langle k \rangle}=(\xi-k+1)\xi^{\langle
k-1 \rangle}\\
\text{si}\:k= 0\:[p^m],\quad & p^m\xi^{\langle k \rangle}=(\xi-k+1)\xi^{\langle
k-1 \rangle},
\end{align*}
dont on déduit la formule suivante pour tout $k\geq 0$
\begin{gather}\label{formGm}
\frac{k!}{q_k!}\xi^{\langle k \rangle}=(\xi-k+1)(\xi-k+2)\cdots (\xi-1)\xi.
\end{gather}
En particulier, dans le cas $m=0$,
on a $\xi^{\langle k \rangle_{(0)} }=(\xi-k+1)(\xi-k+2)\cdots (\xi-1)\xi$ pour tout $k$ et alors
$\xi^{{\langle k \rangle}_{(0)}}\neq \xi^k$ pour $k\geq 2$.
On constate aussi dans ce cas pour $f\in V[T,T^{-1}]$
\begin{gather}\label{formulDGm}
\xi^{\langle k \rangle}(f)=\der_T^{\langle k \rangle}(f)(1).
\end{gather}

\vspace{+2mm}
Revenons maintenant au cas d'un schéma en groupes général $G$, et rappelons que $Lie(G)^*:=I/I^2$ et 
que l'algèbre symétrique de niveau $m$, $\Sg^{(m)}(Lie(G))$, est construite en~\ref{defSm}.
\begin{prop} \begin{enumerate}\label{grDm}
\item[(i)] Il existe un isomorphisme canonique de $V$-algèbres graduées
               $$c_m\,\colon\,\gr\pg D^{(m)}(G)\simeq \Sg^{(m)}(Lie(G)).$$
\item[(ii)] Les algèbres $\gr\pg D^{(m)}(G)$ et $D^{(m)}(G)$ sont noetheriennes.
\item[(iii)] Si $I$ est engendré par une suite régulière de paramètres sur $S$,
     l'algèbre $D^{(m)}(G)$ est engendrée par les éléments $\xi^{\langle p^i \rangle}$ pour $i\leq m$
 décrits en~\ref{prop-basiselementsII_f}.
\ee
\end{prop}
\begin{proof} Montrons la première proposition. La filtration décroissante de l'algèbre
$P^{(m)}(G)$ par les idéaux $I^{\{n\}}$ permet de considérer
l'algèbre graduée
$$\gr\pg P^{(m)}(G)=\bigoplus_{n\geq 0}I^{\{n\}}/I^{\{n+1\}},$$
dont l'idéal $I/I^2$ est muni d'une $m$-PD-structure. Le module $Lie(G)^*=I/I^2=\gr_1 P^{(\infty)}(G)$ s'envoie
vers $\gr_1P^{(m)}(G)$ pour tout entier $m$ (\ref{desc-P1}). On dispose donc d'une flèche
$Lie(G)^*\rig \gr\pg P^{(m)}(G)$ et, par la propriété universelle des algèbres
symétriques, d'un morphisme d'algèbres commutatives $\Sg(Lie(G)^*)\rig \gr\pg P^{(m)}(G)$,
qui envoie $Lie(G)^*$ dans un $m$-PD-idéal, de sorte que cette application
se factorise d'une unique façon en un homomorphisme de $m$-PD-algèbres:
$$c_m^* \,\colon\, \Ga_{(m)}(Lie(G)^*)\rig \gr\pg P^{(m)}(G).$$
Voir que cette application est un isomorphisme est une question locale sur $S$ : on peut donc supposer 
que $I$ est engendré par une suite régulière de paramètres $t_1,\ldots,t_N$.
En reprenant
la description donnée en~\ref{descript_loc_pn}, on voit que ces deux algèbres sont des
  $V$-modules libres gradués de base les éléments $\ut^{\{\ur\}}$. Or le morphisme considéré 
envoie l'élément noté $\ut^{\{\ur\}}$ de $\Ga_{(m)}(Lie(G)^*)$
(Thm. 3.9 in \cite{Berthelot-Ogus}) sur l'élément noté de la même façon de $\gr\pg P^{(m)}(G)$, donc c'est un 
isomorphisme gradué.

En dualisant, cette application donne un isomorphisme $V$-linéaire
$$  c_m\,\colon\,Hom_V(\gr\pg  P_{(m)}(G),V )\rig \Sg^{(m)}(Lie(G)) .$$ 
On a des suites exactes courtes
$$ 0 \rig I^{\{n\}}/I^{\{n+1\}} \rig P_{(m)}^{n+1}(G) \rig P_{(m)}^{n}(G) \rig 0.$$
En dualisant on trouve des injections
$$ gr_{n+1}D^{(m)}(G) \hrig Hom_V(I^{\{n\}}/I^{\{n+1\}},V).$$
Voir que cette flèche est un isomorphisme est une question locale sur $S$ : on peut donc supposer
que $I$ a une suite régulière de paramètres, et dans ce cas l'assertion est claire. Comme cette flèche 
est un isomorphisme, on trouve l'assertion (i) de l'énoncé.

Vérifions maintenant que cette application est un homomorphisme d'algèbres. De nouveau, cela se vérifie localement 
sur $S$ et on peut supposer que $I$ a une suite régulière de paramètres, $t_1,\ldots,t_N$. On reprend les notations 
de ~\ref{prop-basiselementsII_f}.  Dans l'algèbre symétrique de niveau $m$, on a l'égalité
$$\uxi^{\langle\uk'\rangle}\cdot\uxi^{\langle\uk''\rangle}
=\crofrac{\uk'+\uk''}{\uk'}\uxi^{\langle \uk'+\uk'' \rangle}.$$
Effectuons le calcul dans l'algèbre $\gr\pg D^{(m)}(G)$. Le produit
$\uxi^{\langle\uk'\rangle}\cdot\uxi^{\langle\uk''\rangle}$ est d'ordre $\leq
|\uk'|+|\uk''|$. Il suffit donc de calculer, pour $|\ur|=|\uk'|+|\uk''|$
\begin{multline} ( \uxi^{\langle \uk' \rangle}\cdot \uxi^{\langle \uk'' \rangle})(\ut^{\{\ur\}})=
(\uxi^{\langle \uk' \rangle}\ot \uxi^{\langle \uk'' \rangle})
\left( \sum_{\ual\leq \ur}\crofrac{\ur}{\ual}
\ut^{\{\ual\}}\ot \ut^{\{\ur-\ual\}}\right)+ \gamma_{\ur}
=\delta_{\ur,\uk'+\uk''}\crofrac{\uk'+\uk''}{\uk'},
\end{multline}
puisque l'élément $\gamma_{\uk'+\uk''}$ est dans
$\sum_{s+s'\geq|\uk'+\uk''|+1}I^{\{s\}}\ot I^{\{s'\}}$, de sorte
que $(\uxi^{\langle \uk' \rangle}\ot \uxi^{\langle \uk'' \rangle})(\gamma_{\ur})=0$.
On voit ainsi que la formule du produit est exactement la même pour les deux algèbres
considérées et donc qu'on a l'isomorphisme d'algèbres annoncé.

L'assertion (iii) provient de l'isomorphisme précédent et de l'observation (1.3.6 de
\cite{Hucomp}) que l'algèbre symétrique de niveau $m$, $\Sg^{(m)}(Lie(G))$, est
engendrée par les éléments $\xi_l^{\langle p^i \rangle}$ avec $i\leq m$ et $\xi_l$ les
éléments d'une base de $Lie(G)$.
Finalement, $\gr\pg D^{(m)}(G)\simeq \Sg^{(m)}(Lie(G))$ est un anneau noetherien, car
c'est une $V$-algèbre de type fini. Par un argument standard, on en déduit que
$D^{(m)}(G)$ est noetherien (e.g. \cite{MCR}, Thm. 1.6.9). Cela montre (ii) et (iii). 
\end{proof}

\vskip5pt

%
\subsection{PD-stratifications de niveau $m$}
\label{strat_m}
Nous retournons au cas où $G$ est un schéma en groupes quelconque affine et lisse sur $S$.
Grâce au formalisme présenté ici, nous pouvons décrire la donnée d'une structure de
 $D^{(m)}(G)$-module de façon analogue au cas classique en utilisant la notion
de stratification. Il s'agit essentiellement de
reproduire le 2.3 de \cite{Be1}. %

\begin{defi} Soit $E$ un $V$-module. On appelle $G$-$m$-PD-stratification sur $E$ la
donnée d'une famille compatible d'isomorphismes, $P_{(m)}^{n}(G)$-linéaires,
$$\varep_n\,\colon\quad P_{(m)}^{n}(G)\ot_{V} E \sta{\sim}{\rig} E\ot_{V} P_{(m)}^{n}(G),$$
tels que \be\item[(i)]$\varep_0=Id_E$,
\item[(ii)]pour tous $n,n'$, le diagramme suivant
$$\xymatrix { P_{(m)}^{n}(G)\ot_{V} P_{(m)}^{n'}(G)\ot_{V} E
\ar@{->}[rr]^{\delta^{n,n'*}(\varep_{n+n'})}
\ar@{->}[rd]_{q^{n,n'*}_1(\varep_{n+n'})}^{\sim}
& & E\ot_{V} P_{(m)}^{n}(G)\ot_{V} P_{(m)}^{n'}(G)\\
& P_{(m)}^{n}(G)\ot_{V} E \ot_{V} P_{(m)}^{n'}(G)
\ar@{->}[ru]_{q^{n,n'*}_0(\varep_{n+n'})}^{\sim}, & }$$
dont les flèches sont obtenues par extension des scalaires
à partir des applications $\delta^{n,n'}$ (\ref{def-delta}) et des applications
$$q^{n,n'}_0\,\colon\, P^{n+n'}_{(m)}(G) \thrig P^n_{(m)}(G) \rig P^n_{(m)}(G) \ot_{V} P^{n'}_{(m)}(G),$$ et
$$q^{n,n'}_1\,\colon\, P^{n+n'}_{(m)}(G) \thrig P^{n'}_{(m)}(G) \rig P^n_{(m)}(G) \ot_{V} P^{n'}_{(m)}(G),$$
 est commutatif.
\ee
\end{defi}
On dispose de la caractérisation suivante de la structure de $D^{(m)}(G)$-module.
\begin{prop} \label{equiv_stratification}
Soit $E$ un $V$-module. Il y a équivalence entre
\be\item [(i)]se donner une structure de $D^{(m)}(G)$-module sur $E$ prolongeant sa structure de $V$-module,
\item[(ii)]se donner une famille compatible d'homomorphismes $V$-linéaires
$$\theta_n\,\colon\quad E \rig E\ot_{V} P_{(m)}^{n}(G),$$ telle que $\theta_0=Id_{E}$,
et pour tous $n,n'$ le diagramme suivant soit commutatif
$$\xymatrix{E\ot P_{(m)}^{n+n'}(G)\ar@{->}[r]^(.45){Id\ot\delta^{n,n'}}&
E\ot P_{(m)}^{n}(G)\ot P_{(m)}^{n'}(G)\\
E \ar@{->}[r]^{\theta_{n'}}\ar@{->}[u]^{\theta_{n+n'}}&
 E\ot P_{(m)}^{n'}(G),\ar@{->}[u]^{\theta_{n}\ot Id}
}$$
\item[(iii)]se donner une $G$-$m$-PD-stratification sur $E$.
\ee

De plus, un homomorphisme $V$-linéaire $\lambda$: $E\rig F$ entre deux
$D^{(m)}(G)$-modules à gauche est $D^{(m)}(G)$-linéaire si et seulement s'il commute aux
homomorphismes $\theta_n$ (resp. aux isomorphismes $\varep_n$). Dans ce cas, on dit
qu'il est horizontal.
\item[(iv)]
Nous avons un énoncé analogue pour les $D^{(m)}(G)$-modules. Dans ce cas, se donner une $G$-$m$-PD-stratification
sur un $D^{(m)}(G)$-module $M$ revient à se donner une famille compatible d'homomorphismes $V$-linéaires
$\theta_n\,\colon\quad M \rig M\ot_{V} P_{(m)}^{n}(G),$ tels que les diagrammes analogues à ceux ci-dessus soient
commutatifs.

\end{prop}
La démonstration de ces résultats est rigoureusement identique à celle de Grothendieck dans le cas des opérateurs différentiels usuels (cf. proposition 2.3.2 de \cite{Be1}) et nous ne
la donnons pas. Signalons seulement comment sont définis les
$\theta_n$. Si $E$ est un $D^{(m)}(G)$-module, l'homomorphisme
$\theta_n$ est obtenu comme le composé
$$\theta_n\,\colon\quad E\rig Hom_{V}(D^{(m)}(G),E)\sta{\sim}{\rig} E\ot P_{(m)}^{n}(G).$$
Remarquons que comme d'habitude, le deuxième isomorphisme provient du fait que les 
$V$-modules considérés sont localement libres sur $S$.

La commutativité du diagramme imposée dans la proposition précédente correspond
alors à l'associativité de l'action de $D^{(m)}(G)$ sur $E$.

Supposons que $I$ soit engendré sur $S$ par une suite régulière de paramètres,
et reprenons les notations de \ref{prop-basiselementsII},
on voit que si $x\in E$,
$$\theta_n(x)=\sum_{|\uk|\leq n}\uxi^{\langle \uk \rangle}x\ot \ut^{\{k\}},$$
et que les isomorphismes $\varep_n$ obtenus par extension des scalaires ont un inverse
défini par
$$\varep_n^{-1}(x\ot 1)=\sum_{|\uk|\leq n}(-1)^{|\uk|}\ut^{\{\uk\}}\ot \uxi^{\langle \uk
\rangle}.$$

Cette description en termes de $G$-$m$-PD-stratification pour les $D^{(m)}(G)$-modules
nous donne le corollaire suivant comme en 1.2.3.3 de \cite{Be1}.
\begin{cor} Supposons que $I$ soit engendré sur $S$ par une suite régulière de paramètres, et reprenons les notations 
de \ref{prop-basiselementsII}.
\label{struct_tens_f}Soient $E$, $F$ deux $D^{(m)}(G)$-modules à gauche. Alors
\be\item[(i)]Il existe sur $E\ot_{V} F$ une unique structure de $D^{(m)}(G)$-module à gauche
telle que, pour tout $\uk$ et $x\in E$, $y\in F$, on ait
$$\uxi^{\langle \uk \rangle}\cdot(x\ot y)=\sum_{\ul\leq\uk}\crofrac{\uk}{\ul}
\uxi^{\langle \ul \rangle}x\ot \uxi^{\langle \uk-\ul \rangle}y.$$
\item[(ii)]Il existe sur $Hom_{V}(E,F)$ une unique structure de $D^{(m)}(G)$-module à gauche
telle que, pour tout $\uk$, tout morphisme $\varphi$: $E\rig F$, tout $x\in E$, on ait
$$(\uxi^{\langle \uk \rangle}\varphi)(x)=\sum_{\ul\leq \uk}(-1)^{|\ul|}
\crofrac{\uk}{\ul} \uxi^{\langle \uk-\ul \rangle}(\varphi(\uxi^{\langle \ul \rangle}x)).$$
\ee
\end{cor}
\begin{cor} Sous les hypothèses~\ref{section-rappels},
on dispose d'un foncteur de
la catégorie des $G$-modules vers la catégorie des $D^{(m)}(G)$-modules.
\end{cor}
\begin{proof}Soit $M$ un $G$-module dont la structure est définie par une
application $\Delta_M$ $M\rig M\ot V[G]$ vérifiant~\ref{comodule}.
Il suffit de montrer que
l'on peut munir $M$ d'une $G$-$m$-PD-stratification. Pour cela on vérifie le critère
de la proposition~\ref{equiv_stratification} (pour les $D^{(m)}(G)$-modules), en posant
$$\theta_n=(Id_M\ot\rho_m)\circ \Delta_M\,\colon\,M\rig M\ot P^n_{(m)}(G).$$
Pour $n=0$, l'algèbre $ P^n_{(m)}(G)$ est égale à $V$ et on vérifie comme dans le
cas classique que $\theta_0=Id_M$ en utilisant les relations~\ref{comodule}.
Nous vérifions à présent que le diagramme donné en (ii) de la proposition~\ref{equiv_stratification}
commute, en utilisant la commutation du diagramme ci-dessous
(par définition de l'application
$\delta^{n,n'}$~\ref{def-delta})
$$\xymatrix{V[G] \ar@{->}[r]^{\mu^{\sharp}}\ar@{->}[d]_{\rho_{m}} & V[G]\otimes
V[G]\ar@{->}[d]^{\rho_{m}\ot\rho_{m}} \\
P^{n+n'}_{(m)}(G) \ar@{->}[r]^<<<<<{\delta^{n,n'}}& P^n_{(m)}(G)\ot P^{n'}_{(m)}(G),}$$
et les égalités qui suivent, lesquelles font intervenir les relations vérifiées par l'application $\Delta_M$
\begin{align*}
Id\ot\delta^{n,n'}\circ (Id_M\ot\rho_{m})\circ \Delta_M & =  (Id_M\ot(\rho_m\otimes
\rho_{m})\circ \mu^{\sharp})\circ \Delta_M \\
& =  (Id_M\ot(\rho_m\ot\rho_{m}))\circ (\Delta_M\ot id_{V[G]})\circ  \Delta_M \\
& =  \left((Id_M\ot\rho_m)\circ
\Delta_M\ot Id_{P^{n'}_{(m)}(G)}\right)\circ(Id_M\ot\rho_{m})\circ \Delta_M.
\end{align*}
 Toutes les constructions sont fonctorielles.
\end{proof}

\vskip4pt
Remarque: de façon analogue aux stratifications, nous pouvons donner une description des $D^{(m)}(G)$-modules à droite
en définissant la notion de costratification comme dans le chapitre 1 de ~\cite{Be-smf}. Nous n'entrerons pas dans ces détails ici.

\subsection{Liens avec les faisceaux différentiels sur $G$.}
\label{liens-faisdiffG}
Soient $G$ un schéma en groupes lisse sur $S$, $\Delta$ l'immersion diagonale $G\hrig G\times G$, définie par l'idéal
$\II_{\Delta}$, $e_G = e \times id_G$, c'est-à-dire
$$ e_G \,\colon\,\xymatrix @R=0mm { G \ar@{^{(}->}[r] & G \times G \\
                                                      g  \ar@{|->}[r]& (e,g) .}$$
Alors on a pour $t\in V[G]$, $(\varep_G\ot Id) (1\ot t -t \ot
1)=t-\varep_G(t)\in \II$, si bien que $(e_G)^*(\II_{\Delta})=\II$,
et le diagramme suivant est
cartésien
\begin{gather}\label{diagide}
\xymatrix { S \ar@{^{(}->}[r]^e \ar@{^{(}->}[d]_{e}\ar@{}[dr]|{\square} & G
\ar@{^{(}->}[d]^{\Delta}\\
G \ar@{^{(}->}[r]_{e_G}& G\times G.}
\end{gather}
Considérons la $m$-PD-enveloppe de $V[G]$ pour l'idéal $I$ définissant $e$, $P^n_{(m)}(G)$
introduite au début de ~\ref{sect-distribution-algebras}. Le faisceau $\II_{\Delta}/\II_{\Delta}^2$ est 
localement libre sur $G$. Par la propriété de changement de base dans le cas régulier de 1.5.3 de \cite{Be1} par le
morphisme $e$, on dispose de la 
%
\begin{prop}
\label{reduction_mod_e}
Il existe des isomorphismes canoniques de $m$-PD-algèbres
$$ V[G]/I\ot_{V[G]}\Ga(G,\PP_{G,(m)}^{n})\simeq P_{(m)}^{n}(G),$$
pour cet isomorphisme $\Ga(G,\PP_{G,(m)}^{n})$ est un $V[G]$-module 
via $f \mapsto f\ot 1$.
\end{prop}
On en déduit des surjections 
\begin{gather}\label{def_sm_gl}
s_m \,\colon\, \Ga(G,\PP_{G,(m)}^{n})\thrig P_{(m)}^{n}(G).
\end{gather}
Soit $r_m$ ~\ref{def_rm} le morphisme canonique $V[G]\otimes V[G]\rig \Ga(G,\PP^n_{G,(m)})$, $d_1$ :
$V[G] \rig V[G]\otimes V[G]$ donné par $d_1(f)=1\ot f$ 
 et reprenons $\rho_m$ de~\ref{def_rho}.
L'énoncé se complète de l'observation suivante, qui résulte directement du changement de base 
car $(\varep_G\ot id)(1\ot f)=f$ pour $f\in V[G]$.
\begin{prop}\label{reduction_mode_comp} Le diagramme suivant de $V[G]$-modules est commutatif
$$\xymatrix{ V[G] \ar@{->}[r]^{d_1}\ar@{->}[d]^{\rho_m} & \Ga(G,\PP^n_{G,(m)})\ar@{->>}[dl]^{s_m}\\
   P_{(m)}^{n}(G).}$$
\end{prop}

Fixons un entier $m$. En dualisant les isomorphismes de la proposition~\ref{reduction_mod_e}, on
obtient la
\begin{prop}\label{reduction_Dmod_e}   Il existe des isomorphismes de $V$-modules
 $$ \beta_m\,\colon\,D^{(m)}_{n}(G)\simeq V[G]/I\ot_{V[G]}\Ga(G,\DD^{(m)}_{G,n}),\text{ resp. }
D^{(m)}(G)\simeq V[G]/I\ot_{V[G]}\Ga(G,\DD^{(m)}_G).$$
\end{prop}

\vskip5pt

Remarque. La surjection : $\Ga(G,\DD^{(m)}_{G,n})\rig D^{(m)}_{n}(G)$ qui résulte de ces applications n'est pas un morphisme 
d'anneaux. Par exemple,
dans le cas où $m=0$, $p\neq 2$,
$G=\mathbf{G}_m=Spec\, V[T,T^{-1}]$, $T-1$ est une coordonnée locale au voisinage de
$1$, et aussi un générateur de $\II_{\kappa}$. Ainsi, $\tau=1\ot T - T \ot 1$, et
$e^*(\tau)=T-1$, ce qui donne $e^*(\der_T)=\xi^{\langle 1 \rangle}=\xi$. Dans ce cas,
$\der_T^{\langle 2 \rangle}=\der_T^2$.
Or, on a aussi
$e^*(\der_T^{\langle 2 \rangle})=\xi^{\langle 2 \rangle}$. Et on calcule
\begin{align*}
 e^*(\der_T^2) &= \xi^{\langle 2 \rangle} \neq \xi^2,
\end{align*}
d'après la formule ~\ref{formGm}.

\vskip5pt

Soit $P\in \Ga(G,\Dm_G)$, on note $P(e)$ l'image de $P$ par la surjection canonique : 
\begin{gather}\label{def_fe_Pe}
\Ga(G,\Dm_G)\thrig
          V[G]/I\ot_{V[G]} \Ga(G,\Dm_G)  .
\end{gather}
          
 De même, soit $f \in V[G]$, on note $f(e)=\varep_G(f)$ l'image de $f$ par la surjection 
$\varep_G$ : $V[G]\thrig V$ définissant la section identité de $G$, l'action de $P_{(m)}^n(G)$ se fait via 
$\rho_m$. Alors on a 
\begin{prop}\label{calcul_expl} Soient $f$ une section locale de $V[G]$
et $P\in\Ga(G,\Dm_G)$. Alors
$$(P(e))(f)=(P(f))(e).$$
\end{prop}
\begin{proof}
Notons $d_0$ : $V[G]\rig V[G]\ot_V V[G]$ l'application définie par $d_0(f)=f \ot 1$.
L'énoncé résulte du diagramme suivant. Par définition, dans ce diagramme, le carré de droite est commutatif, 
et le carré du milieu est commutatif car provient du changement de base par $e_G$.

$$\xymatrix  { V[G] \ar@{->}[rd]^{Id}\ar@{->}[r]^{d_0} &V[G]\ot_V V[G]\ar@{->}[r]^{r_m}\ar@{->>}[d]^{\varep_G \otimes Id} & 
\Ga(G,\PP_{X,(m)}^{n}) \ar@{->}[r]^{P}\ar@{->>}[d] & V[G]\ar@{->>}[d]^{\varep_G} \\
                     & V[G]\ar@{->}[r]^{\rho_m} & P^n_{(m)}(G)\ar@{->}[r]^{P(e)} & V .  }$$
\end{proof}

%
%
\subsection{Opérateurs différentiels sur des $G$-schémas}
\label{op_diff_inv}
Nous nous plaçons toujours sous les hypothèses~\ref{section-rappels}. Le but est de montrer le théorème ~\ref{opp_diff_inv}. On commence par construire,
dans le cas où un schéma en groupes $G$ agit à gauche sur un $S$-schéma lisse $X$, un anti-homomorphisme
du faisceau d'algèbres des distributions de $G$, $D^{(m)}(G)$ vers le faisceau d'algèbres $\Ga(X,\Dm_X)$.
 Dans le cas d'une action à droite, la construction analogue
fournit un {\it homomorphisme} entre ces deux algèbres, cf. II.\S4. No.4.5 de \cite{Demazure_gr_alg}.
On en
déduit un homomorphisme
d'algèbres de l'algèbre des distributions de $G$, $D^{(m)}(G)$, vers l'algèbre des opérateurs
différentiels globales de niveau $m$ sur $X$. 

 On rappelle 
qu'on note l'action $\sigma$ : $G\times X \rig X$ et $\sigma^{\sharp}$ le morphisme de faisceaux 
$$ \sigma^{\sharp}\, \colon \, \OO_X \rig V[G]\ot_V \OO_X .$$
%
\begin{prop}\label{prop_Qhom} Soient $X$ un $S$-schéma lisse sur lequel $G$ agit à gauche, et $m$ un entier.
 Il existe un anti-homomorphisme d'algèbres filtrées $Q_m$
de l'algèbre $D^{(m)}(G)$ vers l'algèbre des sections globales sur $X$ du faisceau
$\DD^{(m)}_X$. Si $X$ est muni d'une action à droite, on trouve un homomorphisme d'algèbres filtrées
de l'algèbre $D^{(m)}(G)$ vers l'algèbre des sections globales sur $X$ du faisceau
$\DD^{(m)}_X$.
\end{prop}
\begin{proof} Pour la démonstration, on traite le cas de l'action de $G$ à gauche sur $X$, le cas à droite 
étant symétrique.
Soit $\sigma: G\times X\rightarrow X$ l'action de $G$ sur $X$.
On décompose
$$id_X \,\colon\, \xymatrix @R=0mm {X \ar@{^{(}->}^{e_X}[r] & G \times X \ar@{->}[r]^{\sigma} & X \\
        x \ar@{|->}[r]& (e,x). &  }$$
Rappelons que d'après ~\ref{projP}, on dispose d'un projecteur $q_1$, qui est un $m$-PD-morphisme,  $q_1$ :
$\PP_{G\times X,(m)}^{n}\thrig p_1^*\PP_{G,(m)}^{n}$.
Considérons le composé
\begin{equation}\label{equ-q1dsigma} q_1\circ d{\sigma}\,\colon\, \sigma^*\PP^n_{X,(m)} \rig \PP^n_{G\times X,(m)} \thrig
p_1^*\PP^n_{G,(m)}\end{equation} où $d\sigma$ est le $m$-PD morphisme qui apparaît dans la preuve de Prop. \ref{prop-equiv}. Comme $id_X=\sigma\circ e_X$, on trouve en appliquant $e_X^*$ au morphisme précédent un
morphisme de faisceaux d'algèbres $\PP^n_{X,(m)}\rig e_X^*p_1^*\PP^n_{G,(m)}$. 
On voit que $p_1\circ e_X=e \circ st_X$ où $st_X$ est le morphisme structural de $X$ vers $S$, de sorte que
\begin{align*}e_X^*p_1^*\PP^n_{G,(m)} &\simeq \OO_X\ot_V e^*\PP^n_{G,(m)} \\
                   &\simeq  \OO_X\ot_V P^n_{(m)}(G) \text{ d'après~\ref{reduction_mod_e}}.
\end{align*}
Avec cette identification, 
 on voit que la flèche $p_1^*\PP_{G,(m)}^{n}\rig e_{X*}\OO_X\ot_V P^n_{(m)}(G) $ est la flèche
$p_{1}^*s_m $ où $s_m$ est donnée en~\ref{def_sm_gl}.
Ceci nous donne finalement une application 
\begin{gather}\label{def_sigmamn} \sigma_m^{(n)}\,\colon \, \PP^n_{X,(m)} \rig \OO_X\ot_V P_{(m)}^{n}(G).
\end{gather}
La remarque suivante sera utile dans la suite.
\begin{prop}\label{OxP-loclib} Le faisceau $\OO_X\ot_V P_{(m)}^{n}(G)$ est un faisceau de $\OO_X$-modules localement libres.
\end{prop}
En effet, le faisceau $\OO_S\ot_V P^n_{G,(m)}$ est localement libre sur $S$. Supposons que sur 
$S$, l'idéal $I$ soit engendré par une suite régulière d'éléments $t_1,\ldots,t_N$, alors
  $P_{(m)}^{n}(G)$ est libre de base les éléments  $\ut^{\{\uk\}}$ pour
$|\uk|\leq n$ et le faisceau
$\OO_X\ot_V P_{(m)}^{n}(G)$ est un $\OO_X$-module libre de base
les éléments
\begin{gather}\label{base-OxP}
1 \ot \ut^{\{\uk\}}, \textrm{ tels que } |\uk|\leq n.
\end{gather}
Introduisons $\rho_m$ (resp. $r_m$) est l'application canonique de ~\ref{def_rho} (resp. ~\ref{def_rm}).

\vskip5pt

Dans la suite, on munit $\OO_X\ot_V P_{(m)}^{n}(G)$ de la $m$-PD-structure provenant de celle de $P_{(m)}^{n}(G)$.
Les propriétés de $\sigma_m^{(n)}$ sont décrites par la
\begin{prop} \be \item[(i)] L'application $\sigma_m^{(n)}$ est un $m$-PD-morphisme.
\item[(ii)] Le morphisme $\sigma_m^{(n)}$ est $\OO_X$-linéaire pour la structure de $\OO_X$-module à gauche de
$\PP^n_{X,(m)}$.
\item[(iii)] Soit $b\in\OO_X$, tel que $ \sigma^{\sharp}(b)=\sum_i c_i \ot d_i$, avec $c_i\in V[G]$, et
$d_i\in \OO_X$, alors
$$\sigma_m^{(n)}(1 \ot b)=\sum_i \rho_m(c_i) \ot d_i \in P_{(m)}^{n}(G)\ot \OO_X,$$ où
$\rho_m$ est l'application canonique de ~\ref{def_rho} (resp. ~\ref{def_rm}).
\item[(iv)] Les applications $\sigma_m^{(n)}$ sont compatibles entre
elles pour $n$ et $m$ variables.
\ee
\end{prop}
\begin{proof} Le (i) provient du fait que l'application $q_1\circ d\sigma$, defini dans la preuve de la proposition précédente, est un $m$-PD morphisme et du fait que la formation des
$m$-PD-enveloppes commute aux extensions de base dans le cas où on considère des idéaux réguliers. La
$\OO_X$-linéarité de (ii) est automatique car $q_1$ et $d{\sigma}$ sont $\OO_{G\times X}$-linéaires.
Montrons (iii). Dans la suite, grâce à la section canonique $1$ de $\PP^{n}_{X,(m)}$, on identifie
$V[G]\ot_V p_2^{-1}\OO_X$, à un sous-faisceau de 
$$\sigma^*\PP^{n}_{X,(m)}\simeq (V[G]\ot \OO_X)\ot_{\sigma^{-1}\OO_X} \sigma^{-1}\PP^{n}_{X,(m)}.$$
On a
$$ d\sigma(1\ot 1 \ot (1\ot \sigma^{-1}(b))=\sum_i (c_i\ot 1)\cdot (1\ot d_i).$$
Calculons maintenant le projecteur $q_1$ sur les éléments du type $1\ot d_i$.
 Soient $x_1,\ldots,x_M$ un système de
coordonnées locales sur $X$, $x'_1,\ldots,x'_N$ des coordonnées locales sur $G$,
 $\tau_i=1\ot x_i -x_i \ot 1$, $\tau'_j=1\ot x'_j -x'_j \ot 1$. Les opérateurs différentiels relatifs à ce
choix de coordonnées sont alors notés $\uder^{\la \uk' \ra}\uder^{\la \uk \ra}$ tandis que les éléments
$\utau^{\{\uk'\}}\utau^{\{\uk\}}$ forment une base locale de $\PP_{G\times X,(m)}^{n}$ pour
$|\uk|+[\uk'|\leq n$.
Reprenons les notations de l'énoncé, la formule de Taylor donne
$$1 \ot d_i= \sum_{|\uk|\leq n}\uder^{\la \uk \ra}(d_i)\utau^{\{\uk\}}\in \PP_{G\times X,(m)}^{n},$$
de sorte que $q_1(1\ot 1 \ot 1\ot d_i)=1\ot d_i \in p_1^*\PP_{G,(m)}^{n}$
et donc par les propriétés de linéarité des applications considérées,
$q_1(1\ot 1 \ot c_i\ot d_i)=(1 \ot c_i)(1 \ot d_i)\in p_1^*\PP_{G,(m)}^{n}$. 
Il résulte de ~\ref{reduction_mode_comp}, que $s_m(1 \ot c_i)=\rho_m(c_i)$, ce qui démontre la formule de (iii).
L'assertion (iv) est claire.  \end{proof}
Dans la situation (iii) de la proposition, on notera dans la suite
plus simplement $$\sigma_m^{(n)}(1\ot b)=\sum_i c_i \ot d_i.$$

\vskip5pt
Indiquons maintenant comment calculer l'application graduée en degré $1$ de $\sigma_m^{(n)}$ pour sa filtration
$m$-PD-adique. Soient $\gr_1\, q_1$ et $\gr_1\, d{\sigma}$ les gradués de degré $1$ 
des $m$-PD morphismes $q_1$ et $d{\sigma}$.  Donc, on dispose de
$$ \gr_1\, q_1\circ \gr_1\, d{\sigma}= \gr_1\, (q_1\circ d{\sigma})\,\colon\, \sigma^*\Omega^1_X \rig \Omega^1_{G\times X} \rig
p_1^*\Omega^1_{G}.$$
Alors l'application $\gr_1 \sigma_m^{(n)}$ est obtenue comme la composée:
\begin{gather}
\gr_1 \sigma_m^{(n)}\,\colon\,  \Omega^1_X\sta{e_X^*\gr_1\,d\sigma}{\longrightarrow} e_X^*\Omega^1_{G\times X}\sta{e_X^* \gr_1\,q_1}{\longrightarrow}e_X^*p^*_1\Omega^1_G \thrig
I/I^2\ot_{V}\OO_X,
\end{gather}
où la dernière flèche est la surjection naturelle.
\vskip5pt
Soit maintenant $u$ un élément de $ D^{(m)}_n(G)$,
on lui associe l'opérateur différentiel $Q_{m,n}(u)\in \Ga(X,\DD^{(m)}_{X,n})$ suivant (défini localement)
$$Q_{m,n}(u)\,\colon\,\PP_{X,(m)}^{n}\sta{\sigma_m^{(n)}}{\rig} P_{(m)}^n(G)\ot \OO_X\sta{u\ot Id}{\rig}\OO_X.$$
Ces applications $Q_{m,n}$
passent à la limite inductive sur $m$ pour $m$ variable en un morphisme filtré $V$-linéaire
\begin{gather}\label{def_Qm}
Q_m \,\colon\, D^{(m)}(G)\rig\Ga(X,\DD^{(m)}_{X}).
\end{gather}
 Constatons enfin que $Q_m(1)=1$. En effet, $P^{0}_{(m)}(G)\simeq V$ et
$\sigma^{(0)}_m$ : $\OO_X \rig \OO_X $ est $\OO_X$-linéaire et unitaire, donc vaut
$id_{\OO_X}$. Ainsi $Q_m(1)=1$. 

Pour compléter la preuve de la proposition \ref{prop_Qhom}, il nous reste à vérifier le 
\begin{lem} L'application 
$Q_m \;\colon\; D^{(m)}(G)\rig\Ga(X,\DD^{(m)}_{X})$ est un
anti-homomorphisme d'algèbres, i.e. $Q_m(uv)=Q_m(v)Q_m(u)$.
\end{lem}
Soient $u\in D^{(m)}_n(G)$, $v\in D^{(m)}_{n'}(G)$, les applications
$Q_m(uv)$ et $Q_m(v)Q_m(u)$ vus comme éléments de
$\HH om_{\OO_X}(\PP^{n+n'}_{X,(m)},\OO_X)$ sont $\OO_X$-linéaires (pour l'action à gauche
sur $\PP^{n}_{X,(m)}$). Pour les comparer, nous
calculons d'abord $Q_m(uv)(1\ot b)$ pour $b\in \OO_X$. Décomposons
$\sigma^{\sharp}(b)=\sum_i c_i\ot d_i$, avec $c_i\in V[G]$, $d_i\in \OO_X$.

Rappelons que $\delta^{n,n'}$ est définie en~\ref{def-delta}. 
Pour calculer $Q_m(uv)(1\ot b)$, nous considérons l'application $R_m$, obtenue comme le
composé suivant
$$\xymatrix @R=0mm {\PP_{X,(m)}^{n+n'}  \ar@{->}[r]^<<<<<{\sigma_{m}^{(n+n')}} &
 P_{(m)}^{n+n'}(G)\ot_{V} \OO_X
\ar@{->}[r]^(.4){\delta^{n,n'}\ot Id}
&  P_{(m)}^{n}(G) \ot_{V} P_{(m)}^{n'}(G)\ot_{V} \OO_X \\
1\ot b \; \ar@{{|}->}[r] & \sum_i c_i\ot d_i \; \ar@{{|}->}[r] &\sum_{i} \mu^{\sharp}(c_i)\ot
d_i,}$$
qui est un $m$-PD-morphisme comme composé de $m$-PD-morphismes.

Reprenons la définition du produit des opérateurs différentiels de ~\ref{produitD}, qui fait intervenir
$\delta^{n',n} \;\colon \;\PP^{n+n'}_{X,(m)}\rig
\PP^{n'}_{X,(m)}\ot_{\OO_X}\PP^{n}_{X,(m)},$ tel que
$\delta^{n',n}(a\ot b)=a\ot 1\ot 1\ot b$.

Si $\PP^{n'}_{X,(m)}$ est muni de la structure de $\OO_X$ donnée par la multiplication à droite, nous identifions
dans le diagramme qui suit $\PP_{X,(m)}^{n'} \ot_{\OO_X}P_{(m)}^{n}(G)\ot_{V} \OO_X$ à
$P_{(m)}^{n}(G)\ot_{V}\PP_{X,(m)}^{n'}$ en envoyant $(a\ot b)\ot (c\ot d)$ sur $c\ot(a\ot db)$.
Pour calculer $Q_m(v)Q_m(u)$
notons $S_m$ le composé
$$\xymatrix @R=0mm {\PP_{X,(m)}^{n+n'} \ar@{->}[r]^<<<<<{\delta^{n',n}} &
 \PP_{X,(m)}^{n'} \ot_{\OO_X} \PP_{X,(m)}^{n}\ar@{->}[r]^{Id\ovot\sigma_m^{(n)}} &
 P_{(m)}^{n}(G) \ot_V\PP_{X,(m)}^{n'}
 \ar@{->}[r]^(.45){Id \ot \sigma_m^{(n')}} &  P_{(m)}^{n}(G) \ot_V P_{(m)}^{n'}(G)\ot_V\OO_X\\
1\ot b \;\ar@{|->}[r] & 1\ot 1 \ot 1\ot b \;\ar@{|->}[r] & \sum_i c_i \ot (1\ot d_i) \;\ar@{|->}[r] &
\sum_{i} c_i\ot\sigma^{\sharp}(d_i),}$$
qui est un $m$-PD-morphisme.

La relation de comodule~\ref{comodule} appliquée à $\OO_X$ dit exactement que
$R_m(1\ot b)=S_m(1\ot b)$ pour tout $b$ de $\OO_X$, et par $\OO_X$-linéarité,
on voit que $R_m(1\ot b - b \ot 1)=S_m(1\ot b - b \ot 1)$. Comme $R_m$ et $S_m$ sont des $m$-PD-morphismes, ils coïncident donc
pour tout élément $x=1\ot b- b\ot 1$, sur les éléments $x^{\{q\}}$, pour tout entier $q$.
Ces éléments engendrent $\PP_{X,(m)}^{n}$ comme $\OO_X$-algèbre (1.4.4 de \cite{Be1})
de sorte que les $m$-PD-morphismes $R_m$ et $S_m$ sont égaux.

Soient $u,v$ deux éléments de $D^{(m)}(G)$, et $\Phi_{u,v}=Id\ovot u \ovot v$ l'homomorphisme
d'évaluation $ \OO_X\ot_V P_{(m)}^n(G) \ot_V P_{(m)}^{n'}(G)\rig \OO_X.$ L'opérateur différentiel $Q_m(uv)$ est égal à $\Phi_{u,v}\circ R_m$ et $Q_m(v)Q_m(u)$
est égal à $\Phi_{u,v}\circ S_m$, d'où l'égalité $Q_m(uv)=Q_m(v)Q_m(u)$.
\end{proof}
Donnons maintenant une variante à droite de ce lemme, dans le cas où $G$ agit à droite sur un $S$-schéma lisse $X$. Dans 
ce cas la définition de $Q_m$ est tout à fait analogue et on obtient le
\begin{lem}\label{Qm_droit} Si $G$ agit à droite sur un $S$-schéma lisse $X$, l'application 
$Q_m \;\colon\; D^{(m)}(G)\rig\Ga(X,\DD^{(m)}_{X})$ est un homomorphisme d'algèbres, i.e. $Q_m(uv)=Q_m(u)Q_m(v)$.
\end{lem}

\vskip5pt Dans la suite, on se donne une action à gauche de $G$ sur un $S$-schéma lisse $X$. La version à droite
s'obtient facilement à partir de la version à gauche. Les raisonnements faisant intervenir les gradués des algèbres
$D^{(m)}(G)$ et $\Ga(X,\DD^{(m)}_{X})$ sont complètement identiques.

Reprenons maintenant les notations de ~\ref{prop-basiselementsII} et de ~\ref{parties_principales}.
\begin{cor}\label{compat-Ox} \be \item[(i)] $\Ga(X,\OO_X)$ est un $D^{(m)}(G)$-module, 
dont la structure est compatible avec sa structure de $\Ga(X,\Dm_X)$-module.
\item[(ii)] On suppose que l'idéal $I$ est engendré par une suite régulière 
d'éléments $t_1, \ldots,t_N$, alors $D^{(m)}(G)$ est libre de base les éléments  $\uxi^{\la\uk\ra}$
et on peut considérer
$\uxi_X^{\la\uk\ra}=Q_m(\uxi^{\la\uk\ra})$.
On dispose du formulaire suivant pour $\uk$ tel que $|\uk|\leq n$,
et $\ui$ tel que $|\ui|\leq n$,
\begin{gather} \forall f\in\OO_X,\, \sigma_m^{(n)}(1\ot f)=\sum_{|\ul|\leq n} \label{f-taylor}
 \ut^{\{\ul\}}\ot \uxi_X^{\la\ul\ra}(f)\textrm{ (Formule de Taylor)},\\
      \forall \tau\in \PP_{X,(m)}^{n},\;  \sigma_m^{(n)}(\tau)=
\sum_{|\ul|\leq n}\ut^{\{\ul\}} \ot \uxi_X^{\la\ul\ra}(\tau) ,\\
 \uxi_X^{\la\uk\ra}f=\sum_{\uk'+\uk''=\uk}\acfrac{\uk}{\uk'}\uxi_X^{\la\uk'\ra}(f)\uxi_X^{\la\uk''\ra}
\in \Dm_X \label{def_prod_am}.
\end{gather}
\ee
\end{cor}
\begin{proof} Commençons par (i). 
La structure de $D^{(m)}(G)$-module de $\OO_X$ est donnée par
le morphisme composé, pour $u\in D^{(m)}(G)$,
$$ \xymatrix @R=0mm { \OO_X \ar@{->}[r] & \PP_{X,(m)}^{n} \ar@{->}[r]^{Q_m(u)} & \OO_X \\
                  f\ar@{{|}->}[r] & 1\ot f &  },
$$
et est donc par définition compatible avec l'anti-homomorphisme de faisceaux d'algèbres $D^{(m)}(G)\rig \Ga(X,\Dm_X)$.

Montrons (ii). Comme $\uxi^{\la\uk\ra}$ est la base duale de $\ut^{\{\uk\}}$, on a
$$\sigma_m^{(n)}(1\ot f)=\sum_{|\ul|\leq n} \ut^{\{\ul\}} \ot Q_m(\uxi^{\la\ul\ra})(f) ,$$
ce qui donne ~\ref{f-taylor}. C'est exactement la même chose pour la formule suivante. Pour la dernière formule, 
il suffit de comparer ces deux opérateurs différentiels sur les sections locales $\tau$ de $\PP_{X,(m)}^{n}$.
Soit $\tau$ une section locale de $\PP_{X,(m)}^{n}$,
$$(\uxi_X^{\la\uk\ra}f)(\tau)=(\uxi^{\la\uk\ra}\ot Id)\left(\sigma_m^{(n)}\left( (1\ot
f)\tau\right)\right).$$ Or on a
\begin{align} \left(\sigma_m^{(n)}\left( (1\ot f)\tau \right)\right) &=
   \sigma_m^{(n)}(1\ot f)\sigma_m^{(n)}(\tau)\\
&=  \left(\sum_{|\uk'|\leq n}\ut^{\{\uk'\}}\ot \uxi_X^{\la\uk'\ra}(f)\right)
     \left(  \sum_{|\uk''|\leq n}\ut^{\{\uk''\}}\ot \uxi_X^{\la\uk''\ra}(\tau)\right)
\text{ d'après ~\ref{f-taylor}}
\end{align}
et donc
\begin{align} ( \uxi^{\la\uk\ra}\ot Id )\left(\sigma_m^{(n)}\left( (1\ot f)\tau \right) \right)
&= \sum_{\uk'+\uk''=\uk} \acfrac{\uk}{\uk'}\uxi_X^{\la\uk'\ra}(f)\uxi_X^{\la\uk''\ra}(\tau),
\end{align}
		ce qui montre la dernière formule.
\end{proof}
Soit $D^{(m)}(G)^{op}$ l'algèbre opposée à $D^{(m)}(G)$ (i.e munie du produit $PQ=QP$ dans $D^{(m)}(G)$).
Considérons $\AA^{(m)}_{X}= D^{(m)}(G)^{op}\ot_V\OO_X$. L'application $Q_m$ qui est un
homomorphisme d'algèbres $D^{(m)}(G)^{op}\rig \Dm_X$, s'étend
par $\OO_X$-linéarité en une application $Q_{m,X}$ : $\AA^{(m)}_{X} \rig \Dm_X$.
Dans la suite, nous utilisons la notion de $\OO_X$-anneau en sens de Beilinson \cite{BeilinsonICM}.
Nous utilisons toujours les notations de ~\ref{prop-basiselementsII}.
Nous tirons de la proposition~\ref{prop_Qhom} et du corollaire précédent le
\begin{cor} \label{prop-Am}\be
\item[(i)]Le faisceau $\AA^{(m)}_{X}$ est un faisceau de $\OO_X$-modules localement libres.
\item[(ii)]Il existe une unique structure de $\OO_X$-anneau filtré sur $\AA^{(m)}_{X}$ compatible
avec la structure d'algèbre de $D^{(m)}(G)^{op}$ et telle que si l'idéal $I$ est engendré par une suite 
régulière de paramètres $t_1,\ldots,t_N$, on ait, pour $f\in\OO_X$,
$$ (1\ot f)(\uxi^{\la\uk\ra}\ot 1)=\sum_{\uk'+\uk''=\uk}\acfrac{\uk}{\uk'}\uxi^{\la\uk''\ra}\ot \uxi_X^{\la\uk'\ra}(f).$$ De plus, $Q_{m,X}$ : $\AA^{(m)}_{X}\rig
\Dm_X$ est un homomorphisme de $\OO_X$-anneaux filtrés.
\item[(iii)] Il existe un isomorphisme canonique de $\OO_X$-algèbres graduées
               $$c_{m,X}^*\,\colon\,\gr\pg\AA^{(m)}_{X}\simeq \Sg^{(m)}(Lie(G))\ot_V\OO_X.$$
      En outre, les faisceaux $\AA^{(m)}_{X}$ resp. $\gr\pg\AA^{(m)}_{X}$ sont des faisceaux de $\OO_X$-anneaux resp. de $\OO_X$-algèbres cohérents,
 à sections noethériennes sur les ouverts affines.
\ee
\end{cor}
\begin{proof} Le (i) vient du fait que $$\AA^{(m)}_{X}=D^{(m)}(G)\ot_V\OO_X,$$ comme faisceau de groupes abéliens
et du fait que $\OO_S\ot_V D^{(m)}(G)$ est localement libre sur $S$.

Passons au (ii). On observe que $Q_{m,X}(\uxi^{\la\uk\ra}\ot 1)=Q_{m}(\uxi^{\la\uk\ra})$ et si $f\in\OO_X$, $Q_{m,X}(1\ot f)=f$.
On définit une multiplication twistée sur le produit tensoriel $\AA^{(m)}_{X}$ en
utilisant l'application $Q_m$ et l'action naturelle de $\DD^{(m)}_X$ sur $\OO_X$. 
Cette multiplication induit sur
$\AA^{(m)}_X$ une structure de $\OO_X$-anneau. En particulier, on peut former les produits
$(1\ot f)(\uxi^{\la\uk\ra}\ot 1)$ pour $f\in \OO_X$. Et puisque, par définition,
$\uxi_X^{\la\uk\ra}=Q_{m}(\uxi^{\la\uk\ra})$, la formule de l'énoncé implique que
\begin{align}
Q_{m,X}\left((\uxi^{\la\uk\ra}\ot 1)( 1\ot f)\right)&=\sum_{\uk'+\uk''=\uk}
\acfrac{\uk}{\uk'}\uxi_X^{\la\uk'\ra}(f)\uxi_X^{\la\uk''\ra} \\
       &=\uxi_X^{\la\uk\ra}f \text{ d'après \ref{def_prod_am}}\\
        &= Q_{m,X}(\uxi^{\la\uk\ra}\ot 1)Q_{m,X}(1\ot f)
\end{align}
Le $\OO_X$-anneau $\AA^{(m)}_{X}$ est filtrée par les sous-$\OO_X$-modules
$\AA^{(m)}_{X,n}=D^{(m)}_n(G)\ot_V\OO_X$, et, par définition du produit, la structure
de $\OO_X$-anneau est compatible à cette filtration. D'autre part, l'application $Q_m$ étant filtrée,
c'est aussi le cas de $Q_{m,X}$. Comme $\OO_X$ est plat sur $V$, on voit que $\gr\pg\AA^{(m)}_{X}\simeq  \gr\pg
D^{(m)}(G)\ot \OO_X$
et l'isomorphisme cherché suit alors de~\ref{grDm}. Il en résulte que $\gr\pg\AA^{(m)}_{X}$ est une $\OO_X$-algèbre de type
fini, donc noetherienne sur les ouverts affines car $X$ est noetherien, de sorte que
 $\AA^{(m)}_{X}$ est aussi noetherien sur les ouverts affines par un argument standard.
De plus, pour tout couple d'ouverts affines $U,V$ de $X$ tels que $V\subset U$,
$\OO_X(V)$ est plat sur $\OO_X(U)$, si bien que $\AA^{(m)}_{X}(V)$ est plat à droite et à gauche
sur $\AA^{(m)}_{X}(U)$. Les deux conditions du critère 3.1.1 de \cite{Be1} sont donc remplies,
ce qui montre que $\AA^{(m)}_{X}$ est un faisceau cohérent sur $X$. Le même raisonnement s'applique au
faisceau $\gr\pg\AA^{(m)}_{X}$.
\end{proof}
\subsubsection{Etude du gradué de $Q_{m,X}$}
\label{grQm}
Comme le morphisme $Q_{m,X}$ est filtré, il induit par passage aux gradués un morphisme
$$\gr\pg Q_{m,X}\,\colon\, \gr\pg\AA^{(m)}_{X}\rig \gr\pg \Dm_X.$$
En particulier, on dispose de $\gr_1 Q_{m,X}$ : $ \gr_1 D_m(G)\ot_V\OO_X\rig \gr_1\Dm_X$. On en déduit
via les identifications $\ovA_m$ de~\ref{desc-D1} et $\ovB_m$ de ~\ref{desc-D1diff}
que $\gr_1 Q_{m,X}$ induit une unique application $\QQ_{m,1}$ faisant commuter le diagramme suivant
$$\xymatrix{  \gr_1 D^{(m)}(G)\ot_V\OO_X\ar@{->}[r]^>>>>{\gr_1Q_{m,X}}\ar@{->}[d]^{\ovA_m\ot Id}_{\wr} & \gr_1\Dm_X\ar@{->}[d]^{\ovB_m}_{\wr} \\
                Lie(G ) \ot_V\OO_X\ar@{->}[r]^<<<<<<<<{\QQ_{m,1}} & \TT_X.}$$
Nous avons alors le %
\begin{souslem} Pour tout $m$, les applications $\QQ_{m,1}$ co\"\i ncident en une même application
$\QQ_{1}$ : $Lie(G )\ot_V\OO_X \rig \TT_X $.
\end{souslem}
\begin{proof} Notons $\II_X$ l'idéal diagonal de $X$, $\ovA_m^{*-1}$ la flèche obtenue par dualité
à partir de $\ovA_m$ de~\ref{desc-D1}
et $\SS_{m,1}$ : $\II_X/\II_X^{2}\rig  I/I^2\ot_V\OO_X$,
obtenu comme le composé, après les identifications ~\ref{desc-D1diff},
$$\SS_{m,1}\,\colon \,\xymatrix @R=0mm {\II_X/\II_X^2\ar@{->}[r]^{\ovB_m^*} & gr_1\PP_{X,(m)}^{n}\ar@{->}[r]^{gr_1\sigma_m^{(n)}}
 &  gr_1P^n_{(m)}(G)\ot_V\OO_X\ar@{->}[r]^{ \ovA_m^{*-1}\ot Id}&  I/I^2\ot_V\OO_X
.}$$
Alors
l'application $\QQ_{m,1}$ est donnée par le diagramme suivant
$$ \xymatrix @R=0mm {  Lie(G) \ot_V\OO_X\ar@{->}[r] & \TT_X \\
                    \ovu \ar@{{|}->}[r] & \ovu \circ \SS_{m,1}.}$$
Il suffit donc de montrer que ces applications $\SS_{m,1}$ sont égales à une même
application $\SS_{1}$, ce qui provient du fait que les morphismes $\sigma_m$ sont compatibles
pour $m$ variable, de sorte que les $\gr_1 \sigma_m$ sont tous égaux.
\end{proof}
Reprenons l'isomorphisme $d_m$ de ~\ref{grdiffDm} : $\gr\pg\Dm_X \sta{\sim}{\rig} \Sg^{(m)}(\TT_X)$, construit de façon analogue
à $c_m$ en~\ref{grDm} et qui provient par dualité de l'isomorphisme canonique de $m$-PD-algèbres
$d_m^* $ : $ \Gamma_{X,(m)}(\II/\II^2)\sta{\sim}{\rig} \gr\pg \PP_{X,(m)}$ de~\ref{isomgrP}. On a alors la
\begin{sousprop} \label{prop_grAm}\be \item[(i)]
Le diagramme suivant est commutatif.
                 $$\xymatrix{ \gr\pg \AA^{(m)}_{X}\ar@{->}[d]^{c_{m,X}}_{\wr} \ar@{->}[r]^{\gr\pg Q_{m,X}} & \gr \pg \Dm_X \ar@{->}[d]^{d_m}_{\wr}\\
                       \Sg^{(m)}(Lie(G)) \ot_V\OO_X       \ar@{->}[r]& \Sg^{(m)}(\TT_X),
                    }$$
                    où la seconde flèche horizontale est l'application $\Sg^{(m)}(\QQ_{m,1})$.
\item[(ii)] Si $X$ est un espace homogène sous l'action de $G$, alors $\gr\pg Q_{m,X}$ est surjectif.
\ee
\end{sousprop}
\begin{proof} On a les égalités
$$ \gr\pg\Dm_X = \HH om_{\OO_X}\left(\bigoplus_n \II_X^{\{n\}}/\II_X^{\{n+1\}},\OO_X\right)\text{ et }
   \gr\pg D^{(m)}(G) = Hom_{V}\left(\bigoplus_n I^{\{n\}}/I^{\{n+1\}},V\right).$$
Les applications $\sigma_m^{(n)}$ induisent une application
 $$\gr\pg\sigma_m\,\colon\,   \bigoplus_n \II_X^{\{n\}}/\II_X^{\{n+1\}} \rig \bigoplus_n I^{\{n\}}/I^{\{n+1\}}\ot_V\OO_X,$$
de sorte que l'application $\gr\pg Q_{m,X}$ est donnée par dualité 
$$ \xymatrix @R=0mm { \gr\pg \AA^{(m)}_{X}\ar@{->}[r] & \gr\pg\Dm_X  \\
                  \ovu \ar@{{|}->}[r] & \ovu \circ \gr\pg \sigma_m.}
$$
De même l'application $\gr_1Q_{m,X}$ est donnée par 
$$ \xymatrix @R=0mm { \gr_1 \AA^{(m)}_{X}\ar@{->}[r]&  \gr_1\Dm_X  \\
                  \ovu \ar@{{|}->}[r] & \ovu \circ \gr_1 \sigma_m.}
$$

La commutativité du diagramme de l'énoncé revient donc à la commutativité du diagramme suivant
$$\xymatrix{ \bigoplus_n \II_X^{\{n\}}/\II_X^{\{n+1\}} \ar@{->}[r]^>>>>>>{\gr\pg \sigma_m} &
\bigoplus_n I^{\{n\}}/I^{\{n+1\}}\ot_V\OO_X \\
     \Gamma_{X,(m)}(\II_X/\II_X^2) \ar@{->}[u]^{id \ot c_m^*}_{\wr}\ar@{->}[r]^>>>>>>{\Gamma_{(m)}(\SS_{1})} &
 \Gamma_{(m)}(I/I^2)\ot_V\OO_X\ar@{->}[u]^{d_m^*}_{\wr}.}$$
Nous devons comparer $d_m^* \circ \Gamma_{(m)}(\SS_{1})$ et $\gr\pg \sigma_m \circ (id \ot c_m^*)$,
qui sont des $m$-PD-morphismes de $m$-PD-algèbres graduées. En degré $1$, $d_m^*=\ovB_m^*$
(resp. $c_m^*=\ovA_m^*$), si bien que ces morphismes sont égaux en degré $1$ par définition de
$\SS_{1}$. Comme la $m$-PD-algèbre $\Gamma_{X,(m)}(\II/\II^2)$ est engendrée par les éléments de degré $1$,
on voit ainsi que les morphismes $d_m^* \circ \Gamma_{(m)}(\SS_{1})$ et $\gr\pg \sigma_m \circ (id \ot c_m^*)$
sont égaux.

Observons pour (ii) que $\QQ_1$ est l'application habituelle $Lie(G)\ot_V\OO_X \rig  \TT_X$, qui est surjective
si $X$ est un $G$-espace homogène, par un argument classique refait en 1.6 de ~\cite{huy-beil_ber}. Ceci montre par (i) que l'application
$\gr\pg Q_{m,X}$ est surjective.
\end{proof}

\subsubsection{Opérateurs différentiels invariants sur $G$}
\label{defop_diff_inv}
On suppose ici que $G$ est un schéma en groupes affine et lisse sur $S$. On rappelle les définitions pour
le cas d'une action à droite $X \times G \rig X$. Ces définitions s'adaptent de façon
évidente pour le cas d'une action à gauche.

Comme $G$ est plat, le morphisme de projection $p_1$: $X\times G\rig X$ est affine et
plat. Soit $\EE$ un faisceau de $\OO_X$-modules $G$-équivariant par un isomorphisme
$\Phi$ :
$\sigma^* \EE \sta{\sim}{\rig} p_1^*\EE.$
 La formule de K\"unneth donne un isomorphisme pour tout $k\in\Ne$
$$H^k(X\times G, p_1^*\EE)\simeq H^k(X,\EE)\ot_V V[G] .$$
 L'application composée suivante $u_G^k(\Phi)$, seulement notée $u_G$ pour $k=0$:
$$H^k(X,\EE)\rig
H^k(X\times G,\sigma^*\EE)\sta{H^k(\Phi)}{\rig}H^k(X\times G,p_1^*\EE)$$
donne donc finalement un morphisme $\Delta^k$, seulement notée $\Delta$ pour $k=0$:
$$\Delta^k\;\colon\;H^k(X,\EE)\rig  H^k(X,\EE)\ot_V V[G],$$
dont le lecteur pourra vérifier qu'il définit une structure de $G$-module (à gauche) sur
$H^k(X,\EE)$. Les relations de co-module viennent des relations de cocycles.

\begin{sousdef}
Les éléments $G$-invariants de $\Ga(X,\EE)$
sont les éléments $P$ de $\Ga(X,\EE)$ tels que $\Delta(P)=P\ot 1$.
Le sous-espace de $\Ga(X,\EE)$ formé par des éléments $G$-invariants est noté $\Ga(X,\EE)^G$.
\end{sousdef}
En particulier, on dispose des espaces $\Ga(X,\PP_{X,(m)}^n)^G$, $\Ga(X,{\DD}^{(m)}_{X})^G$
et $\Ga(X,\DD^{(m)}_{X,n})^G$ etc.

\vskip5pt
Si $X$ est égal à $G$, l'action des translations à droite de
$G$ sur lui-même ($g\in G$ opére par $h\mapsto hg$) va donner une action à gauche sur $\Ga(G,\DD^{(m)}_G)$ et $Q_m$ est alors
à valeurs dans les opérateurs différentiels invariants pour cette action,
ce qui étend la proposition classique (II, par. 4, no 6 de \cite{Demazure_gr_alg}) suivante. Pour énoncer cette
proposition, on considère les applications $Q_m$ de ~\ref{prop_Qhom}, ainsi que l'application d'évaluation
$ev_e$ : $\Ga(G,\DD^{(m)}_{G})^G \rig D^{(m)}(G)$ qui est définie par
$P \mapsto P(e)$, cf. ~\ref{def_fe_Pe}.

\begin{sousth} \label{opp_diff_inv}
 Les applications canoniques $Q_{m}$ et $ev_e$ sont des isomorphismes d'algèbres filtrées,
inverses l'un de l'autre entre $\Ga(G,\DD^{(m)}_{G})^G$ et  $D^{(m)}(G)$. Ces applications induisent
des bijections canoniques entre $\Ga(G,\DD^{(m)}_{G,n})^G$ et $D^{(m)}_n(G)$.
\end{sousth}
\begin{proof} Montrons d'abord l'énoncé à un ordre $n$ fixé, le cas de
$D^{(m)}(G)$ s'en déduit par passage à la limite. Observons d'abord que si $u\in D^{(m)}_n(G)$,
alors $Q_m(u)$ est invariant.

Soit $P\in\Ga(G,\DD^{(m)}_{G,n})$, et
$\Phi: \mu^*\PP_{G,(m)}^n\rig p_1^* \PP_{G,(m)}^n$ l'isomorphisme donnant la $G$-équivariance des faisceaux $\PP_{G,(m)}^n$
pour l'action des translations à droite de $G$ sur $G$.
La condition d'invariance de $P$ s'écrit
$$\forall \tau\in\mu^{-1}\PP_{G,(m)}^n,\quad P(\tau)\ot_{{\mu}^{-1}\OO_X}1\ot 1=
(P\ot_{{p_1}^{-1}\OO_X}1\ot 1)(\Phi(  \tau \ot 1\ot 1 ))\;\in\OO_{G\times G},$$
ou encore pour tout $\tau \in\mu^{-1}\PP_{G,(m)}^n$
$$\mu^{\sharp}(P(\tau))=(P\ot_{{p_1}^{-1}\OO_X}1\ot 1)(\Phi(  \tau \ot 1\ot 1
))\;\in\OO_{G\times G}.$$

Considérons le $m$-PD-morphisme composé $\overline{U}_m$
$$\mu^{-1}\PP_{G,(m)}^n\sta{\mu^{-1}\mu^{(n)}_m}{\rig}P^{n}_{(m)}(G)\ot_V\mu^{-1}V[G]
\sta{Id\ot \mu^{\sharp}}{\rig}P^{n}_{(m)}(G)\ot_V\OO_{G\times G},$$
que nous étendons par $\mu^{-1}\OO_{G}$ linéarité en un $m$-PD-morphisme
$$U_m : \mu^*\PP_{G,(m)}^n\rig P^{n}_{(m)}(G)\ot_V\OO_{G\times G}.$$ Ici, l'application $\mu_m^{(n)}$ est définie en ~\ref{def_sigmamn} pour $X:=G$ et $\sigma:=\mu$. Considérons à présent le $m$-PD-morphisme composé
$\overline{V}_m$
$$\xymatrix{ \mu^{-1}\PP_{G,(m)}^n\ar@{->}[r]^{\Phi}& p_1^*\PP_{G,(m)}^n &
\ar@{->}[r]^>>>>>>>>{p_1^{-1}\mu_m^{(n)}\ot_{p_1^{-1}V[G]} Id_{G\times G}}&
P^{n}_{(m)}(G)\ot_V \OO_{G\times G} ,}$$
que nous étendons par $\mu^{-1}\OO_{G}$ linéarité en un $m$-PD-morphisme
$$V_m : \mu^*\PP_{G,(m)}^n\rig P^{n}_{(m)}(G)\ot_V \OO_{G\times G}.$$

Nous observons que pour tout $\tau\in \mu^{-1}\PP_{G,(m)}^n$, et $P=Q_m(u)$
$$Q_m(u)(\tau)\ot_{{\mu}^{-1}\OO_X}1\ot 1= (u \ot Id \ot Id)\circ U_m,$$
$$(Q_m(u)\ot_{{p_1}^{-1}\OO_X}\ot 1\ot 1)(\Phi(\tau \ot 1 \ot 1))=(u \ot Id\ot Id )\circ V_m.$$
Nous sommes donc ramenés à montrer que $U_m$ et $V_m$ co\"\i ncident, et,
comme ce sont des $m$-PD-morphismes $\OO_{G\times G}$-linéaires, à montrer que $U_m(\tau
\ot 1\ot 1)=
V_m(\tau \ot 1 \ot 1)$ pour tout $\tau=1\ot t - t\ot 1$, où $t\in V[G]$. Dans ce
cas le calcul provient en fait d'un calcul à valeurs dans $V[G]\ot_V\OO_{G\times G}$
et on compose avec $\rho_m\ot Id_{G\times G}$ pour l'avoir à valeurs dans
$P^{n}_{(m)}(G)\ot_V \OO_{G\times G}$. Posons
$\mu^{\sharp}(t)=\sum_i a_i \ot b_i$, si bien que
$$
\Phi(\tau\ot 1 \ot 1)=\sum_i(1\ot a_i - a_i\ot 1)(1\ot b_i)\in p_1^*\PP_{G,(m)}^n,
$$
\begin{align*}U_m(\tau \ot 1\ot 1)&=(\rho_m\ot Id_{G\times G})\circ(Id \ot \mu^{\sharp})\left(\sum_i a_i \ot b_i
- 1\ot t\right) \\
&= (\rho_m\ot Id_{G\times G})\left(\sum_i  b_i \ot \mu^{\sharp}(a_i)
-1\ot \mu^{\sharp}(t)\right),
\end{align*}
\begin{align*}V_m( \tau \ot 1 \ot 1)&=\sum_i\mu_m(1 \ot a_i - a_i\ot
1)\ot_{p_1^{-1}V[G]}(1 \ot b_i)\\
                                  &=(\rho_m\ot Id_{G\times G})\left(\sum_i (\mu^{\sharp}(a_i)-1\ot a_i)\ot_{p_1^{-1}V[G]}
(1\ot b_i)\right)\\
  &=(\rho_m\ot Id_{G\times G})\left(\sum_i \mu^{\sharp}(a_i)\ot b_i -\sum_i 1\ot a_i \ot b_i\right) \\
 &=(\rho_m\ot Id_{G\times G})\left(\sum_i \mu^{\sharp}(a_i)\ot b_i -1\ot
\mu^{\sharp}(t)\right),
\end{align*}
d'où l'égalité cherchée.

Ainsi l'application $Q_m$ est à valeurs dans l'algèbre $\Ga(G,\DD^{(m)}_{G,n})^G$.
Construisons maintenant une application réciproque $\Ga(G,\DD^{(m)}_{G,n})^G \rig D^{(m)}_n(G)$
pour terminer la démonstration du théorème.

Soit $P$ un opérateur différentiel de $\Ga(G,\DD^{(m)}_{G,n})$,
qui définit une application $V[G]$-linéaire à gauche $\PP_{G,(m)}^n\rig V[G]$.
 Comme en ~\ref{reduction_Dmod_e}, on note $P(e)$ l'élément de $ D^{(m)}_{n}(G)$,
induit par $\varep_G^* P$.

Observons tout de suite que si $u\in D^{(m)}_n(G)$, alors
$$Q_m(u)(e)=u.$$ L'opérateur $Q_m(u)$ est en effet donné par le diagramme
$$\PP_{G,(m)}^{n}\sta{\mu_m}{\rig} P_{(m)}^n(G)\ot V[G]\sta{u\ot Id}{\rig}V[G],$$
et donc $Q_m(u)(e)$ par le diagramme
$$\varep_G^*\PP_{G,(m)}^{n}\simeq P_{(m)}^{n}(G) \sta{u}{\rig} V,$$
ce qui montre notre assertion.

\vskip5pt

Soit maintenant $P$ un élément de $\Ga(G,\DD^{(m)}_{G,n})^G$, nous allons vérifier que
$$Q_m(P(e))=P.$$

Considérons $e\times Id_G$ : $  G \rig G\times G$, qui envoie
$g$ sur $(e,g)$ introduit en ~\ref{liens-faisdiffG}. Alors le faisceau $(e\times Id_G)^*(p_1^*\PP_{G,(m)}^n)$
est égal à $P_{(m)}^{n}(G)\ot_VV[G]$ et
$(e\times Id_G)^*(p_1^*\PP_{G,(m)}^n)$ à $\PP_{G,(m)}^n$.
Reprenons le $m$-PD-morphisme $\Phi$ : $\mu^*\PP_{G,(m)}^n\rig p_1^*\PP_{G,(m)}^n$ donnant la $G$-équivariance des faisceaux $\PP_{G,(m)}^n$
et le $m$-PD-morphisme $\mu_m^{(n)}$ de ~\ref{def_sigmamn} (pour $X:=G$ et $\sigma:=\mu$). Nous disposons alors du lemme suivant.
\begin{souslem} Les applications $(e\times Id_G)^*\Phi$ et ${\mu}_m^{(n)}$ co\"\i ncident.
\end{souslem}
\begin{proof} Comme on doit comparer deux $m$-PD-morphismes $\OO_{G}$-linéaires,
de $\PP_{G,(m)}^n\rig P_{(m)}^{n}(G)\ot_V V[G]$,
il suffit de comparer ces applications sur les éléments $\tau$ avec
$\tau=1\ot t -t \ot 1$ et $t$ une section locale de $V[G]$. Posons
$\mu^{\sharp}(t)=\sum_i a_i \ot b_i$.
Calculons
\begin{align*}
 (\varep_G^*\ot Id_G)^*\circ\Phi(\tau\ot 1 \ot 1) &= (\varep_G^*\ot Id_G)^*\left(\sum_i (1\ot a_i -a_i \ot
1)\ot_{p_1^{-1}V[G]}(1\ot b_i)\right)\\
 &= (\rho_m \ot Id_G)\left(\sum_i(a_i -a_i(e))\ot_V b_i\right) \\
 &= (\rho_m \ot Id_G)\left(\mu^{\sharp}(t)-1\ot t\right),
\end{align*}
qui est bien égal à $\mu_m^{(n)}(\tau)$, comme cela a déjà été calculé.
\end{proof}

Considérons maintenant le diagramme suivant, où $P$ est un opérateur
différentiel sur $G$
$$ \xymatrix{ \PP_{G,(m)}^n \ar@{->}[rr]^{P}& & V[G] \\
 \mu^*\PP_{G,(m)}^n \ar@{}[dr]|{\circlearrowright} \ar@{->}[r]^{\Phi}\ar@{->>}[u]\ar@{->>}[d] &
p_1^*\PP_{G,(m)}^n \ar@{}[dr]|{\circlearrowright} \ar@{->}[r]^{P\ot Id_{G\times G}}\ar@{->>}[d]&
\OO_{G\times G}\ar@{->>}[u]\ar@{->>}[d]\\
\PP_{G,(m)}^n\ar@{->}[r]^>>>>>{\mu_m^{(n)}}&
P_{(m)}^{n}(G)\ot_V V[G] \ar@{->}[r]^<<<<<{P(e)\ot Id_G}& V[G]  .} $$

Les flèches de la ligne du haut et de la ligne du bas se déduisent de
la ligne du milieu par application de $\varep_G^*\ot Id_G$.
Le carré en bas à gauche est commutatif grâce au lemme précédent, le carré
en bas à droite est commutatif par définition de $P(e)$. La ligne du bas
calcule donc $Q_m(P(e))$.

Supposons que
$P$ est invariant, alors pour toute section locale $\tau $ de
$\PP_{G,(m)}^n$, on a $$\mu^{\sharp}(P(\tau))=(P\ot 1 \ot 1)\circ \Phi(\tau
\ot 1 \ot 1).$$ Comme
$P(\tau)=(\varep_G^*\ot Id_G)\circ \mu^{\sharp}(P(\tau))$, la condition
d'invariance dit précisément que le rectangle du haut est commutatif.

Et donc, si $P$ est invariant, on a bien $P=Q_m(P(e))$.

Comme l'application d'évaluation
restreinte à $\Ga(G,\DD^{(m)}_{G})^G$ est un inverse de $Q_m$, il est automatique que cette application est
un homomorphisme d'algèbres. Ceci montre finalement le théorème.
\end{proof}
Terminons par deux exemples. D'abord le cas du groupe additif de dimension $1$,
$\mathbf{G}_a=Spec\, V[T]$, $\der^{\la k \ra }$ les opérateurs différentiels relatifs au choix de la coordonnée
$T$.
\begin{sousprop} Reprenons la base de $D^{(m)}(\mathbf{G}_a)$ de ~\ref{DmGa}. Alors
    $$Q_m(\xi^{\langle k \rangle})=\der^{\la k \ra }.$$
\end{sousprop}
\begin{proof} Notons $u=T\ot 1$, $v=1\ot T \in V[T]\ot V[T]$, $\tau=1 \ot T -T \ot 1 \in \PP_{G,(m)}^{n}$,
$\tau_u=1\ot 1 \ot u - u\ot 1\ot 1 $,
$\tau_v=1\ot 1 \ot v - v\ot 1\ot 1 \in \PP_{G\times G,(m)}^{n}$. On note $t=T$ définissant l'idéal
de l'élément neutre de $\mathbf{G}_a$. Ainsi les éléments $t^{\{k\}}$ forment une base de
$P_{\mathbf{G}_a,(m)}^{n}$. Rappelons l'application $q_1\circ d{\mu}$ de (\ref{equ-q1dsigma}) pour $\sigma:=\mu$ égal à l'action à droite de $\mathbf{G}_a$ sur lui-même.

Comme la loi de groupes est donnée
 par $\mu^{\sharp}(T)=1 \ot T + T \ot 1$, on a
\begin{align*} d\mu (\tau) &= \tau_u + \tau_v \\
               d\mu (\tau^{\{k\}}) &= (\tau_u + \tau_v)^{\{k\}}\\
                                       &= \sum_{k'+k''=k} \crofrac{k}{k'}\tau_u^{\{k'\}}\tau_v{\{k''\}}
             \text{ par ~\ref{formulesmPD}},
\end{align*}
et donc
$q_1\circ d\mu(\tau^{\{k\}})=\tau_u^{\{k\}}$, ce qui, en appliquant la description explicite de
 $ \tilde{\beta}_m$ de ~\ref{reduction_mod_e}, donne
$$\mu_m^{(n)}(\tau^{\{k\}})=t^{k}\ot 1 \in P_{\mathbf{G}_a,(m)}^{n}\ot V[T].$$ Par définition,
 les $\xi^{\{l\}}$ sont la base duale des $T^{k}$, ce qui implique que les $Q_m(\xi^{\{l\}})$ sont la
base duale des $\tau^{\{k\}}$, d'où l'énoncé.
\end{proof}

Prenons maintenant $G=\mathbf{G}_m$, qu'on identifie à $Spec\, V[T,T^{-1}]$. Soit $t=T-1$,
 $I=t V[T,T^{-1}]$ est l'idéal définissant l'élément neutre. Soient
$\der^{\la k \ra }$ les opérateurs différentiels relatifs au choix de la coordonnée
$T$.
\begin{sousprop} Reprenons la base de $D^{(m)}(\mathbf{G}_m)$ donnée en~\ref{DmGm}. Alors
    $$Q_m(\xi^{\langle k \rangle})=T^k\der^{\la k \ra }.$$
\end{sousprop}
\begin{proof} Notons $u=T\ot 1$, $v=1\ot T \in V[T,T^{-1}]\ot V[T,T^{-1}]$,
$\tau=1 \ot T -T \ot 1 \in \PP_{G,(m)}^{n}$,
$\tau_u=1\ot 1 \ot u - u\ot 1\ot 1 $,
$\tau_v=1\ot 1 \ot v - v\ot 1\ot 1 \in \PP_{G\times G,(m)}^{n}$. Avec ces notations, l'immersion
fermée $e_G$ correspond à la surjection
$$ \varep_G \ot id_G \,\colon\, \xymatrix @R=0mm { V[u,u^{-1},v,v^{-1}] \ar@{->>}[r] & V[T,T^{-1}]\\
                                      u  \ar@{{|}->}[r] & 1 \\
                                       v \ar@{{|}->}[r] & T.}$$
Comme la loi de groupes est donnée
 par $\mu^{\sharp}(T)=T \ot T $, on a
\begin{align*} d\mu (\tau) &= (v +\tau_v )\tau_u + u\tau_v \\
              q_1\circ d\mu (\tau^{\{k\}}) &= v^k \tau_u^{\{k\}}.
\end{align*}
Et finalement on trouve
$$    \mu_m^{(n)}(\tau^{\{k\}})=t^{\{k\}}\ot T^k\in P_{\mathbf{G}_m,(m)}^{n}\ot V[T,T^{-1}],$$ ce qui permet de montrer
la formule de l'énoncé.
\end{proof}

Passons à une variante twistée de ce qui précède.
\subsection{Opérateurs différentiels twistés sur des $G$-schémas}
\label{op_difft_inv}
Reprenons ici la situation de la sous-section précédente~\ref{op_diff_inv} :
$G$ agit à gauche sur un schéma $X$, lisse sur $S$, via un morphisme $\sigma: G\times X\rightarrow X$. On note aussi
$p_2$ : $G\times X\rig X$ la deuxième projection. Supposons de plus que $X$
est muni d'un faisceau inversible $\LL$, $G$-équivariant. Suivant ~\ref{operateur-difft}, nous
considérons les faisceaux d'opérateurs différentiels twistés par $\LL$
$$^t{}\DD^{(m)}_{X}\simeq
\LL\ot_{\OO_{X,g}}\DD^{(m)}_{X}\ot_{\OO_{X,d}}\LL^{-1}.$$
Dans cette sous-section, nous
allons démontrer que tous les énoncés précédents restent vrais dans le cas de
ces faisceaux. Pour ce faire, définissons en suivant les notations de ~\ref{faisceauxpp_grad},
le fibré vectoriel associé à $\LL$
$$Y=Spec~\Sg(\LL),$$ et $q$ le morphisme canonique : $Y\rig X$.
 Comme $\LL$ est équivariant, il existe d'après
~\ref{O-mod-equiv} un isomorphisme $\Phi_{\LL}$ : $\sigma^*\LL\simeq p_{2}^*\LL$
vérifiant certaines conditions de cocycle.
Cet isomorphisme s'étend en un
isomorphisme gradué $\Phi$ : $\sigma^*\Sg(\LL)\simeq p_{2}^*\Sg(\LL)$,
qui permet de définir une action à gauche $\sigma_Y$ : $G\times Y \rig Y$. On notera $p'_2$ la
deuxième projection, $p'_2$ : $G \times Y \rig Y$.
Par définition,
$\Phi_{|{\sigma^{-1}\OO_X}}$ est le morphisme $\sigma^{\sharp}$ : $\sigma^{-1}\OO_X\rig \OO_{G\times X}$
 donnant l'action de $G$ sur $X$. Le faisceau $\Dm_Y$ est $G$-équivariant d'après la proposition ~\ref{gequiv_diffprop}.
De même, le faisceau twisté $^t{}\DD^{(m)}_{X}$ est  $G$-équivariant comme produit tensoriel de
faisceaux $G$-équivariants d'après les résultats de~\ref{FaisceauxGequiv}.

Introduisons sur $X$ le faisceau de $\OO_X$-modules $\Dm_Y(\LL)$ de ~\ref{def_DYL} avec son morphisme de restriction $r_{\LL}$.
\begin{prop} \be \item[(i)] Le faisceau $\Dm_Y(\LL) $ est un sous-faisceau $G$-équivariant de $q_*\Dm_Y$,
ce faisceau étant vu comme faisceau de $\OO_X$-modules $G$-équivariant.
\item[(ii)] Le morphisme de restriction $r_{\LL}$ est un morphisme de $G$-faisceaux équivariants, c'est-à-dire
qu'il commute à l'action de $G$.
\ee
\end{prop}
\begin{proof} Soient $q'=id_G\times q$ : $G\times Y \rig G\times X$, $\Psi_X$ :
$\sigma^* \PP_{X,(m)}^{n}\simeq p_2^* \PP_{X,(m)}^{n}$, décrivant l'action de $G$ sur $\PP_{X,(m)}^{n}$.
 Comme $\sigma_Y$, (resp.
$p'_2$) est lisse, on dispose d'isomorphismes canoniques $q'_*\sigma_Y^*\Dm_Y \simeq \sigma^* q_*\Dm_Y$,
(resp. $q'_*{p'}_2^*\Dm_Y \simeq {p}_2^* q_*\Dm_Y$). Comme $\Dm_Y$ est $G$-équivariant, il existe un
isomorphisme $\sigma_Y^*\Dm_Y \simeq {p'}_2^*\Dm_Y$, vérifiant les conditions de cocycle~\ref{cond_cocycles}. Par fonctorialité,
on en déduit un isomorphisme $\Phi_Y$ : $\sigma^* q_*\Dm_Y\simeq p_2^* q_*\Dm_Y$ définissant la structure de $G$-module équivariant
de $q_*\Dm_Y$. Pour montrer (i), il suffit de montrer que $\Phi_Y$ induit un isomorphisme $\sigma^*\Dm_Y(\LL)\simeq p_2^*\Dm_Y(\LL)$.
La vérification est formelle. Montrons par exemple que si $P\in \Dm_Y(\LL)$,
$$\Phi^{-1}_Y(1\ot P)(\sigma^*(\PP_{X,(m)}^{n}\ot \LL))\subset \sigma^*\LL,$$ l'autre vérification étant analogue.
Le morphisme
$\Phi^{-1}_Y$ s'obtient après application de $\HH om_{q_*\OO_Y}(\cdot,q_*\OO_Y)$ à $\Psi_Y$ :
$\sigma^*q_* \PP_{Y,(m)}^{n}\simeq p_2^* q_*\PP_{Y,(m)}^{n}$, décrivant l'action de $G$ sur $q_*\PP_{Y,(m)}^{n}$. Par
définition de $\Psi_Y$, $\PP_{X,(m)}^{n}\ot \LL$ est un sous-$G$-module équivariant de $q_*\PP_{Y,(m)}^{n}$.
Finalement, si $T\in \PP_{X,(m)}^{n}\ot \LL$,
$$ \Phi^{-1}_Y(1\ot P)(1\ot T)=\Phi^{-1}\left((1 \ot P)(\Psi_Y(1\ot T))\right).$$
Or, $\Psi_Y(1\ot T)\in p_2^*(\PP_{X,(m)}^{n}\ot \LL)$, donc $(1 \ot P)(\Psi_Y(1\ot T))\in p_2^*\LL$,
et le terme de droite de l'égalité précédente appartient à $\Phi^{-1}(p_2^*\LL)\subset \sigma^*\LL$. Finalement,
on voit que $\Phi^{-1}_Y(1 \ot P)$ vérifie la condition annoncée. Le fait que $r_{\LL}$ est un morphisme de $G$-modules
équivariants provient du fait que $\psi_Y$, restreint à $\PP_{X,(m)}^{n}\ot \LL$, est induit par $\Psi_X$ et $\Phi$.
\end{proof}
Nous en déduisons le
\begin{cor}\label{Qtmc} Soient $X$ un $S$-schéma lisse sur lequel $G$ agit à gauche, et $m$ un entier. Il
existe un anti-homomorphisme d'algèbres filtrées $^t{}Q_m$
de l'algèbre $D^{(m)}(G)$ vers l'algèbre des sections globales sur $X$ du faisceau
$^t{}\DD^{(m)}_X$.
\end{cor}
\begin{proof} Comme le fibré vectoriel $Y$ est muni d'une action de $G$, on dispose
d'après~\ref{prop_Qhom} d'un anti-homomorphisme d'anneaux : $Q_{m,Y}$ :
$ D^{(m)}(G) \rig \Ga(Y,\Dm_Y)$. Montrons que $Q_{m,Y}$ est à valeurs dans $\Dm_Y(\LL)$.
Reprenons pour cela la construction de loc. cit. ainsi que le morphisme de $m$-PD-algèbres
~\ref{def_sigmamn} $ q_*\sigma_{m,Y}^{(n)}\,\colon\,q_*\PP_{Y,(m)}^{n}\rig  P_{(m)}^n(G)\ot_V q_*\OO_Y$,
resp. $\sigma_{m,X}^{(n)}$ sur le schéma $X$. On rappelle que si $u\in D^{(m)}_n(G)$, $Q_{m,n}(u)$ est le composé
$$Q_{m,n}(u)\,\colon\,\PP_{Y,(m)}^{n}\sta{\sigma_m^{(n)}}{\rig} P_{(m)}^n(G)\ot
\OO_Y\sta{u\ot Id}{\rig}\OO_Y.$$ Pour montrer que $Q_{m,n}(u)\in \Dm_Y(\LL)$, il suffit donc de montrer les
deux inclusions
\begin{gather} \label{inclusion_sig}
\sigma_{m,Y}^{(n)}(\PP_{X,(m)}^{n})\subset P_{(m)}^n(G)\ot_V\OO_X,
\end{gather}
$$\sigma_{m,Y}^{(n)}(\PP_{X,(m)}^{n}\ot \LL)\subset P_{(m)}^n(G)\ot_V\LL.$$ La première
inclusion résulte du fait que la construction est fonctorielle et donc
$${{q_*\sigma_{m,Y}^{(n)}}}_{|\PP_{X,(m)}^{n}}=\sigma_{m,X}^{(n)}.$$ Pour la deuxième inclusion,
comme $\sigma_{m,Y}^{(n)}$ est un morphisme de $m$-PD-algèbres et par la formule ci-dessus ~\ref{inclusion_sig},
 il suffit de vérifier que
pour tout $ f\in \LL$, $\sigma_{m,Y}^{(n)}(1 \ot f)\in p_2^*(\LL)$. Or, c'est le cas, puisque, par construction,
$\sigma_{m,Y}^{(n)}(1 \ot f)=\Phi_{\LL}(1\ot f)\in p_2^*(\LL).$ Notons à présent
\begin{gather} {}^t\sigma_{m,X}^{(n)}={q_*\sigma_{m,Y}^{(n)}}_{|\PP_{X,(m)}^{n}(1\ot \LL)},
\end{gather}
qui est $\OO_X$-linéaire, puisque $\sigma_{m,Y}^{(n)}$ est $\OO_Y$-linéaire.
On définit maintenant
${}^t Q_m=r_{\LL}\circ Q_m,$ qui définit un anti-homomorphisme d'algèbres filtrées
$${}^t Q_m \;\colon \; D^{(m)}(G)\rig \Ga(X,{}^t\DD^{(m)}_X).$$
\end{proof}

\vskip4pt
Remarque: par définition, ${}^t\sigma_{m,X}^{(n)}$ est aussi défini comme suit. Soient
$T\in \PP_{X,(m)}^{n}$, $f\in \LL$, alors
$${}^t\sigma_{m,X}^{(n)}(T \ot f)=\sigma_{m,X}^{(n)}(T)\Phi_{\LL}(1 \ot f),$$
où on prend le produit dans la $\OO_Y$-algèbre $P_{(m)}^n(G)\ot_V\OO_Y$, produit qui est en fait à valeurs
dans $P_{(m)}^n(G)\ot_V\LL$.
Reprenons le faisceau $\AA^{(m)}_{X}=\OO_X\ot_V D^{(m)}(G)^{op}$ de \ref{prop-Am}.
Comme ${}^tQ_m$ est filtré, il induit une application
$\gr\pg {}^tQ_m$ : $\gr\pg \AA^{(m)}_{X}\rig \gr\pg {}^t\DD^{(m)}_X$. Nous avons vu en~\ref{isomgrDt},
que $\gr\pg {}^t\DD^{(m)}_X$ est canoniquement isomorphe à $\gr\pg \DD^{(m)}_X$. Montrons
maintenant la
\begin{prop}\label{grQtw} Via l'isomorphisme canonique de ~\ref{isomgrDt}, on a l'égalité
$\gr\pg {}^tQ_m=\gr\pg Q_m$. En particulier, si $X$ est un $G$-espace homogène,
le morphisme gradué $\gr\pg {}^tQ_m$ est surjectif.
\end{prop}
\begin{proof} Nous reprenons les notations de~\ref{grQm}. Soient ${}^tQ_{m,1}$ l'application
induite sur les gradués de degré $1$ par ${}^tQ_m$, soit
$${}^t\QQ_{m,1}\;\colon \; Lie(G)\ot_V \OO_X\sta{ \ovot {}^tQ_{m,1}\ovot Id}{\lorig}\TT_X.$$ Alors, en refaisant la démonstration de ~\ref{prop_grAm}, on voit qu'on a un diagramme commutatif
$$\xymatrix{ \gr\pg \AA^{(m)}_{X}\ar@{->}[d]^{c_{m,X}}_{\wr} \ar@{->}[r]^{\gr\pg {}^tQ_{m,X}} & \gr \pg {}^t\Dm_X
\ar@{->}[d]^{{}^td_m}_{\wr}\\
                       \Sg^{(m)}(Lie(G)) \ot_V \OO_X   \ar@{->}[r]& \Sg^{(m)}(\TT_X),
                    }$$
                    où la seconde flèche horizontale est l'application $\Sg^{(m)}({}^t\QQ_{m,1})$.
Pour conclure, il suffit donc vérifier que ${}^t\QQ_{m,1}=\QQ_{m,1}$. Soient
$$h_m\;\colon\; \Omega^1_X\ot\LL\sta{\gr_1\sigma_{m,Y}^{(n)}}{\lorig}\II_{\kappa}/\II_{\kappa}^2\ot\LL,$$
et $g_m=id_{\LL^{-1}}\ovot h_m$ : $\Omega^1_X \rig \II_{\kappa}/\II_{\kappa}^2$.
Alors la flèche ${}^t\QQ_{m,1}$ s'obtient après application de
$\HH om_{\OO_X}(\cdot,\OO_X)$ à la flèche $g_m$.
 Il suffit donc finalement de montrer que
$g_m=\gr_1\sigma_{m,X}^{(n)}.$ Or, $\sigma_{m,Y}^{(n)}$ est $\OO_Y$ linéaire. Soient
$\omega\in\Omega^1_X$, $f\in\LL$,
\begin{align*}
\gr_1\sigma_{m,Y}^{(n)}(\omega\ot (1\ot f))&=\gr_1\sigma_{m,Y}^{(n)}((f\ot 1)\omega) \\
         &= \gr_1\sigma_{m,X}^{(n)}(\omega)\ot f,
\end{align*}
ce qui donne bien que $g_m=\gr_1\sigma_{m,X}^{(n)}.$

Comme $\gr\pg Q_m$ est surjectif si $X$ est un $G$-espace homogène par (ii) de ~\ref{prop_grAm}, on trouve que
$\gr\pg {}^tQ_m$ est surjectif dans ce cas.
\end{proof}

\vskip4pt
Remarque: on peut retrouver ce résultat à partir de (iv) de~\ref{rL}, en utilisant la description locale de
$r_{\LL}$.

\section{Algèbres de distributions arithmétiques (faiblement) complétées}\label{alg_dist_comp}
Dans toute cette section, $V$ est un AVDC. Rappelons que $S=Spec\, V$ et $\SS=Spf\,V$. Dans ce cas
$p$-adique ou dans le cas d'un quotient de $V$, nous avons déjà remarqué que les algèbres de
distributions $D_n^{(m)}(G)$ sont des $V$-modules libres~\ref{prop-basiselementsII}, dont une base
s'exprime facilement à partie d'une base de $Lie(G)$. 
\subsection{Définitions}
\label{distrib_cris}
 Soit $\GG$ un $\SS$-schéma formel. On dit que $\GG$ est
 un $\SS$-schéma formel en groupes, si tous les $S$-schémas $G_i$
sont des schémas en groupes, et si tous les morphismes $G_{i+1}\hrig G_i$ sont des
morphismes de schémas en groupes. Dans la suite, on suppose que $\GG$ est un $\SS$-schéma formel
en groupes affine et lisse, i.e., tous les $G_i$ sont des schémas en groupes affines et lisses sur $S$ (de la même dimension relative).
Ainsi, l'algèbre des sections globales $\OO(\GG):=\Ga(\GG,\OO_{\GG})$ est une
algèbre topologiquement de type fini sur $V$ et donc $\pi$-adiquement complète.
On a $\GG= Spf\, \OO(\GG)$ comme $\SS$-schéma formel, et les $\OO(\GG)/\pi^{i+1}\OO(\GG)$ sont les algèbres des groupes
$G_i$, i.e. $\OO(\GG)/\pi^{i+1}\OO(\GG)=V[G_i]$.

Un cas particulier est le complété formel $\GG$ d'un schéma en groupes affine et lisse $G\rig S$
le long de sa fibre spéciale.

\vskip5pt

Donnons maintenant les définitions des algèbres de distributions complètes et faiblement
complètes pour $\GG$, un $\SS$-schéma formel en groupes affine et lisse. Le morphisme $G_{i+1}\hrig G_i$ induit un homomorphisme
$D^{(m)}(G_{i+1})\rightarrow D^{(m)}(G_i)$, cf. \ref{prop-functoriality}.
\begin{defi}On pose
$\what{D}^{(m)}(\GG):=\varprojlim_{i}D^{(m)}(G_i)$.
\end{defi}
Soient $t_1,\ldots, t_N$ une suite régulière de générateurs de l'idéal complet  
$$ I =Ker \left(\OO(\GG) \sta{\varep_G}{\thrig}V\right) . $$
 Nous pouvons alors reprendre les
 notations~\ref{subsection-def_prop}, et on a
$$\what{D}^{(m)}(\GG)=\left\{ \sum_{\uk}a_{\uk}\uxi^{\langle \uk \rangle}\;|\; a_{\uk}\in V,\,\vp(a_{\uk})\rig +\infty \; \text{si} \;|\uk|\rig +\infty \right\}. $$
Si $m'\geq m$, on dispose de morphismes d'algèbres canoniques $\what{D}^{(m)}(\GG)\rig
\what{D}^{(m')}(\GG)$, ce qui permet de donner la définition des algèbres de distributions
suivantes.
\begin{defi}On pose $D^{\dagger}(\GG):=\varinjlim_{m}\what{D}^{(m)}(\GG)$ et
$D^{\dagger}(\GG)_{\Qr}:=D^{\dagger}(\GG)\otimes\Qr. $
\end{defi}
L'algèbre $D^{\dagger}(\GG)$ est séparée et est faiblement complète comme limite
inductive d'algèbres complètes (voir par exemple la remarque (ii) de \cite{Hucomp}).
C'est naturellement une sous-algèbre de la complétion $\pi$-adique
$$\what{Dist}(\GG):=\varprojlim_{i}Dist(\GG)/\pi^{i+1} Dist(\GG)\simeq \varprojlim_{i}Dist(G_i)$$
de l'algèbre des distributions classiques $Dist(\GG)$ sur $\GG$.
Cette complétion est décrite explicitement par
$$\what{Dist}(\GG)=\left\{ \sum_{\uk}a_{\uk}\uxi^{[\uk]}\;|\;
\vp(a_{\uk})\rig +\infty \; \text{si} \;|\uk|\rig +\infty \right\}$$
en utilisant les notations de \ref{prop-basiselementsII}. La proposition suivante est dans
l'esprit de la proposition 2.4.4 de \cite{Be1} et montre que les éléments de la
sous-algèbre $D^{\dagger}(\GG)\subseteq \what{Dist}(\GG)$ peuvent être caractérisés par
des conditions de croissance des coefficients $a_{\uk}$ qui sont du type de Monsky-Washnitzer \cite{MonskyWashI}.
\begin{prop}\label{prop-dagger}
Pour $$P= \sum_{\uk}a_{\uk}\uxi^{[\uk]}$$
dans $\what{Dist}(\GG)$, soit $P_i\in Dist(G_i)$ sa réduction modulo $\pi^{i+1}$ . Les
conditions suivantes sont équivalentes:
\be

\item[(i)] L'élément $P$ appartient à $D^{\dagger}(\GG)$.

\item[(ii)] Il existe des constantes réelles $\alpha,\beta$ avec $\alpha >0$ ayant la
propriété :
${\rm ord} (P_i)\leq \alpha i +\beta$ pour tout $i$. Ici, ${\rm ord}(P_i)$ designe l'ordre fini de la distribution $P_i$.

\item[(iii)] Il existe des constantes réelles $c,\eta$ avec $\eta>0$ ayant la
propriété : $\vp(a_{\uk})\geq \eta |\uk| + c$ pour tout $\uk$.

\ee
\end{prop}
\begin{proof}
D'après la proposition \ref{prop-basiselementsII} on a
$\uxi^{\langle \uk \rangle_{(m)}}=q^{(m)}_{\uk}!\uxi^{[\uk]}$.
On peut donc utiliser les arguments donnés dans la démonstration de la proposition 2.4.4 of \cite{Be1}.
\end{proof}

\vskip4pt
Remarque: Il est immédiat d'après la fonctorialité~\ref{prop-functoriality} établie à un niveau fini que
la formation des algèbres $D^{\dagger}(\GG)$ et $\what{D}^{(m)}(\GG)$ est fonctorielle par rapport au groupe $\GG$.

\subsection{Liens avec les faisceaux différentiels sur un $\GG$-schéma}
\label{op-diff-inv-comp}
Nous revenons ici à la situation de ~\ref{distrib_cris}. On note encore $e$ l'immersion fermée de schémas formels
définissant l'élément neutre : $\SS \hrig \GG$. Nous donnons d'abord les énoncés reliant les algèbres de
distributions et les opérateurs différentiels invariants. Il s'agit de passer à la limite pour tous les
résultats obtenus à un niveau fini.

Fixons un entier $m$. En passant à la limite projective sur $i$ à partir de ~\ref{reduction_Dmod_e}, on trouve l'énoncé
suivant
\begin{prop}\label{reduction_Dcmod_e} \be
\item[(i)] Il existe un isomorphisme de $V$-modules $\beta_m$ : $\what{D}^{(m)}(\GG)\simeq 
\OO(\GG)/I\ot_{\OO(\GG)} \Ga(\GG,\Dcm_{\GG}).$
\item[(ii)] Il existe un isomorphisme de $V$-modules $\beta^{\dagger}$ : $D^{\dagger}(\GG)_{\Qr}\simeq \OO(\GG)/I
\ot_{\OO(\GG)}\Ga(\GG,\Ddag_{\GG,\Qr}).$
\ee
\end{prop}
Notons que le (ii) s'obtient par passage à la limite inductive sur $m$ à partir de (i). Ceci permet d'introduire
la morphisme d'évaluation $ev_e$.

On définit $P(e)$ la distribution obtenue par $P(e)=\beta_m^{-1}(e^*P)$ (resp. $P(e)={\beta^{\dagger}}^{-1}(e^*P)$). On déduit
alors de ~\ref{calcul_expl} par passage au complété $p$-adique, puis par passage à la limite inductive sur $m$
la
\begin{prop}\label{calcul_explc} Soit $P\in\Ga(\GG,\Dcm_{\GG})$ (resp. $P\in \Ga(\GG,\Ddag_{\GG,\Qr})$). Alors
$$\forall f\in \OO(\GG), \, (P(e))(f)=(P(f))(e).$$
\end{prop}
\vskip5pt
Supposons maintenant que $X$ est un $S$-schéma lisse muni d'une action à droite de $G$ et $\XX$ est le schéma formel obtenu par complétion
le long de $\pi\OO_X$. Donc, $\XX$ est un $\SS$-schéma lisse muni d'une action à droite de $\GG$.
En passant à la limite à partir des homomorphismes d'algèbres $Q_m$ de
~\ref{Qm_droit}, on voit aussi qu'il existe des homomorphismes d'algèbres toujours notés
$Q_m$
\begin{gather}\label{def_Qmc}
Q_m \,\colon\,\what{D}^{(m)}({\cal G})\rig\Ga(\XX,\Dcm_{\XX}).
\end{gather}
En passant à la limite sur $m$ et en tensorisant par $\Qr$, on trouve un homomorphisme d'algèbres
\begin{gather}\label{def_Q}
Q \,\colon\,D^{\dagger}({\cal G})_{\Qr}\rig\Ga(\XX,\Ddag_{\XX,{\Qr}}).
\end{gather}
Les faisceaux
$\Dcm_{\XX}$ sont alors $\GG$-équivariants (au sens où les faisceaux
$\Dm_{X_i}$ sont $G_i$-équivariants pour tout $i$ et de façon compatible). Par passage à la limite inductive sur $m$,
les faisceaux $\Ddag_{\XX,\Qr}$ sont eux aussi $\GG$-équivariants.

D'après~\ref{defop_diff_inv}, on dispose donc d'applications
$$\Ga(X_i,\DD^{(m)}_{X_i})\rig \Ga(X_i,\DD^{(m)}_{X_i})\ot_{V_i} V_i[G_i],$$
qui, en passant à la limite projective sur $i$ donnent un morphisme
$\what{u}_G$ : $\varprojlim_{i}(\Ga(X_i, \DD^{(m)}_{X_i}))\rig
\varprojlim_{i}(\Ga(X_i,p_{1,i}^* \DD^{(m)}_{X_i}))$, et, compte
tenu de l'identification habituelle
une flèche $$\what{u}_G\;\colon\;\Ga(\XX,\what{\DD}^{(m)}_{\XX})\rig
\Ga(\XX,\HH^0p_1^!\what{\DD}^{(m)}_{\XX}).$$
On notera $\iota$ l'injection canonique $\Ga(\XX,\what{\DD}^{(m)}_{\XX})\hrig
\Ga(\XX,\HH^0p_1^!\what{\DD}^{(m)}_{\XX})$. En passant à la limite sur
$m$ et en tensorisant par $\Qr$, on obtient une application
$$u_G^{\dagger}\;\colon\;\Ga(\XX,\DD^{\dagger}_{\XX,\Qr})\hrig
\Ga(\XX,\HH^0p_1^!\DD^{\dagger}_{\XX,\Qr}).$$
On continue de noter $\iota$ l'injection canonique $\Ga(\XX,\what{\DD}^{(m)}_{\XX})\hrig
\Ga(\XX,\HH^0p_1^!\what{\DD}^{(m)}_{\XX})$ (resp. $\Ga(\XX,{\DD}^{\dagger}_{\XX,\Qr})\hrig
\Ga(\XX,\HH^0p_1^!{\DD}^{\dagger}_{\XX,\Qr})$). Ceci nous permet de définir les éléments $\GG$-invariants
de la façon suivante.
\begin{sousdef}
\be
\item[(i)] Les éléments $G$-invariants de $\Ga(\XX,\what{\DD}^{(m)}_X)$
sont les éléments $P$ de $\Ga(\XX,\what{\DD}^{(m)}_{\XX,\Qr})$ tels que $\what{u}_G(P)=\iota(P)$.
\item[(ii)] Les éléments $G$-invariants de $\Ga(\XX,{\DD}^{\dagger}_{\XX,\Qr})$ sont les
éléments $P$ de $\Ga(\XX,{\DD}^{\dagger}_{\XX,\Qr})$ tels que $u_G^{\dagger}=\iota(P)$.
\ee
\end{sousdef}
Plaçons-nous maintenant dans le cas où $\XX=\GG$ et $\GG$ opère sur lui-même par translation à droite. Comme au niveau
fini, $Q_m$ et $Q$ sont en fait des homomorphismes d'algèbres à valeurs dans les
modules des opérateurs différentiels invariants sur $\GG$.
Par passage à la limite projective sur $i$, puis limite inductive sur $m$, on dispose, à partir de
~\ref{opp_diff_inv} du
\begin{thm} \label{opp_diffc_inv}
Les applications canoniques $Q_{m}$ et $ev_e$ sont des isomorphismes d'algèbres filtrées,
inverses l'un de l'autre entre les algèbres $\what{D}^{(m)}(\cal G)$ et $\Ga(\GG,\Dcm_{\GG})^{\GG}$
(resp. les applications canoniques
$Q$ et $ev_e$ sont des isomorphismes d'algèbres filtrées,
inverses l'un de l'autre entre les algèbres $D^{\dagger}(\cal G)_{\Qr}$ et $\Ga(\GG,\Ddag_{\GG,\Qr})^{\GG}$).
\end{thm}
\vskip5pt
Retournons au cas général et supposons donné sur $\XX$ (muni d'une action à gauche)
un faisceau inversible de $\OO_{\XX}$-modules $\LL$. On note $\LL_i$
le faisceau induit sur $X_i$ par $\LL$. Supposons de plus que $\LL$ est $\GG$-équivariant au sens où
pour tout $i\geq 0$, les faiceaux $\LL_i$ sont $G_i$-équivariant, de façon compatible pour tout $i$.
On dispose alors des faisceaux d'opérateurs différentiels $\GG$-équivariants
$^t{}\Dcm_{\XX}=\LL\ot_{\OO_{X,g}}\Dcm_{\XX}\ot_{\OO_{X,d}}\LL^{-1}$, resp.
$^t{}\Ddag_{\XX,\Qr}=\LL\ot_{\OO_{X,g}}\Ddag_{\XX,\Qr}\ot_{\OO_{X,d}}\LL^{-1}$.
En passant à la limite projective sur $i$ à partir du corollaire ~\ref{Qtmc}, on dispose de la
\begin{prop} \label{Qctm}Il existe un anti-homomorphisme d'algèbres filtrées $^t{}Q_m$ : $\what{D}^{(m)}(G)\rig
\Ga(\XX,^t{}\Dcm_{\XX})$ (resp. $^t{}Q$ : $D^{\dagger}(G)_{\Qr}\rig \Ga(\XX,^t{}\Ddag_{\XX,\Qr})).$
\end{prop}
Si $\XX$ est muni d'une action à droite, on obtient un homomorphisme d'algèbres filtrées.

\subsection{Distributions analytiques rigides}
\label{dist_an_rig}
Nous supposons ici que $V$ est un AVDC dont le corps des fractions $K=Frac(V)$ est une extension finie de $\Qr_p$. 
Nous utiliserons \cite{SchneiderNFA} pour des notions
élémentaires d'analyse fonctionelle non-archimédienne. Soit $\GG$ le complété formel d'un schéma en groupes affine et lisse $G\rig S$
le long de sa fibre spéciale $G_\kappa$. Soit $\OO(\GG)$ l'algèbre affine de $\GG$. Soit $\what{G}$ la complétion de $G$ 
relativement au point fermé correspondant à l'élément neutre de
$G_\kappa$. Donc, $\what{G}$ un schéma en groupes affine formel sur $\SS$ (mais pas un $\SS$-schéma formel car, en général,
$\pi \OO_{\what{G}}$ n'est pas un idéal de définition!). On écrit
 $$G^\circ:=\what{G}^{\rm rig}$$ pour sa fibre générique au sens de Berthelot, cf. 7.1 de \cite{deJong}. C'est un groupe
 analytique rigide, en général non-affinoïde, sur $K$ avec $Lie(G^\circ)=Lie(G)\otimes K$. On peut décrire ces objects
 comme suit. Prenons $t_1,...,t_N$ une suite regulière de générateurs de $\II_\kappa$ qui engendrent $V[G]$
 comme $V$-algèbre, i.e. $V[G]=V[t_1,...,t_N]$ et $\OO(\GG)=V\langle t_1,...,t_N\rangle$. Il suit que
$\OO(\what{G})=V[[t_1,...,t_N]]$
est isomorphe à l'anneau des séries formelles sur $V$ en $N$ variables et l'espace $G^\circ$ est isomorphe au disque 
unité ouvert de dimension $N$, cf. Lemma 7.4.3 de \cite{deJong}. En particulier, il existe un recouvrement
admissible croissant formé par des ouverts affinoides $(G^\circ)_r=Sp(A_r), r<1$ qui sont associés aux algèbres de Tate
$$A_r:=\left\{ \sum_{\uk} b_{\uk}\;\ut^{\uk}: b_{\uk}\in K, | b_{\uk}| r^{|\uk|}\rightarrow 0\right\},$$
 si $|\uk|\rightarrow\infty$. Il en résulte que
$\OO(G^\circ)=\varprojlim_{r<1} A_r$
est une algèbre de Fréchet nucléaire sur $K$, cf. Prop. 19.9 de \cite{SchneiderNFA}. Remarquons que 
$t_1,...,t_N\in\OO(G^\circ)$ sont des sections globales sur $G^\circ$. On note
$$D^{an}(G^\circ):=\OO(G^\circ)':={\rm Hom}_K^{\rm ct}(\OO(G^\circ),K)$$
le dual continu de l'espace Fréchet $\OO(G^\circ)$. On équipe $D^{an}(G^\circ)$ avec la topologie forte, i.e.
$D^{an}(G^\circ)=\OO(G^\circ)'_b$ au sens de \cite{SchneiderNFA}, \S6. Cette topologie est égale à la topologie limite inductive via l'isomorphisme canonique et topologique
$$(\varprojlim_{r<1} A_r)'_b\stackrel{\simeq}{\longrightarrow} \varinjlim_{r<1} (A_r)'_b,$$
cf. Prop. 16.5 de \cite{SchneiderNFA}. La topologie forte sur $(A_r)_b'$ est simplement la topologie de Banach, Remark 6.7 de \cite{SchneiderNFA}. Par conséquent, $D^{an}(G^\circ)$ est un espace localement convexe de type compact et le produit de convolution fait de
$D^{an}(G^\circ)$ une $K$-algèbre topologique, cf. 5.2 de \cite{Emerton}. Il y a un homomorphisme canonique $\delta: G^\circ(K)\hookrightarrow D^{an}(G^\circ)^\times$ du groupe des points de $G^\circ$ à valeurs dans $K$ dans le groupe des éléments inversibles où $\delta_g$ est la forme linéaire continue $f\mapsto f(g)$ de $\OO(G^\circ)$.
De plus, à $\xi\in Lie(G^\circ)$ on peut associer la
forme linéaire continue
\begin{equation}\label{equ-actionLie} f\mapsto \xi.f:=\frac{d}{dt} f(e^{t\xi}) |_{t=0}\end{equation}
de $\OO(G^\circ)$ où $e$ est l'application exponentielle de $G^\circ$ définie près de $1\in G^\circ$. Il en résulte un homomorphisme d'anneaux
$U(Lie(G^\circ))\rightarrow D^{an}(G^\circ).$

D'un autre côté, chaque $\what{D}^{(m)}(\GG)$ est muni de la topologie $\pi$-adique. Munissons $\what{D}^{(m)}(\GG)_\Qr$ de la topologie d'espace de Banach
dont $\what{D}^{(m)}(\GG)$ est la boule unité. La limite inductive $D^{\dagger}(\GG)_\Qr$ est alors une $K$-algèbre topologique sur un espace localement convexe.
 \begin{prop}\label{prop-isomorph}
L'application $U(Lie(G^\circ))\rightarrow D^{an}(G^\circ)$ s'étend en un isomorphisme d'anneaux topologiques
$$ D^{\dagger}(\GG)_\Qr\stackrel{\simeq}{\longrightarrow}D^{an}(G^\circ).$$
\end{prop}
\begin{proof}
Nous généralisons les arguments de 2.3.3 de \cite{patel_schmidt_strauch_dist_algebra_GL2} dans le cas $G=GL_2$ et $V=\Zp$.
On reprend les notations de \ref{distrib_cris}. Soit
$\xi_1,...,\xi_N$ la base de $Lie(G)$ duale de $t_1,...,t_N$. Par définition des éléments $\uxi^{[\uk]}$,
$$\uxi^{[\uk]} . \ut^{\uk'} = \left\{
\begin{array}{lcl}
1 & , & \uk = \uk' \\
0 & , &  \uk \neq \uk' \end{array} \right. \;.$$
Par conséquent, on peut identifier l'espace dual $A'_r$, avec le produit convolution, via l'application $U(Lie(G^\circ))\rightarrow D^{an}(G^\circ)$, à l'algèbre de Banach formée des sommes infinies $$ \lambda=\sum_{\uk}a_{\uk}\uxi^{[\uk]}$$
avec $a_{\uk}\in K$ et $ |a_{\uk}|\leq C r^{|\uk|}$
pour tout $\uk$ et une certaine constante $C=C(\lambda)$ dépendant sur $\lambda$. Il résulte alors de (iii) de \ref{prop-dagger} que la limite inductive des $A'_r$ est égale à $D^{\dagger}(\GG)_\Qr$.
\end{proof}

\begin{cor}\label{thm-coherence}
Soit $e$ l'indice de ramification de $K$ sur $\Qr_p$. Si $e<p-1$, alors l'anneau $D^{\dagger}(\GG)_\Qr$ 
est cohérent.
\end{cor}
\begin{proof}
Soient $G^\circ(K)\subset G(K)$ les groupes des points de $G^\circ$ et $G$ \`a valeurs dans $K$. Le groupe de Lie $p$-adique $G^\circ(K)$
est un groupe standard au sens de III.\S 7.3 dans \cite{B-L}. La hypoth\`ese sur $e$ implique que $G^\circ(K)=\exp (\pi Lie(G))$ 
o\`u $\exp$ est l'application exponentielle du groupe de Lie $G(K)$, cf. III.\S 7. Prop. 14 dans \cite{B-L}. Donc, le sous-groupe $G^\circ(K)\subset G(K)$ est un \og good analytic subgroup \fg au sens de 5.2 dans \cite{Emerton} et $D^{an}(G^\circ)$ est cohérent, cf. 5.3.12 de \cite{Emerton}.
\end{proof}

Remarques: 
\begin{enumerate}
		\item Pour chaque $m$, il existe un $m' > m$ tel que que la fl\`eche naturelle
		        $\what{D}^{(m)}(\GG)_\Qr\rightarrow \what{D}^{(m')}(\GG)_\Qr$ soit une application 
				lin\'eaire {\it compacte} entre 
		          espaces de Banach. En fait, par la proposition, cette application se 
				   factorise par une application $(A_r)'_b\rightarrow (A_{r'})'_b$ avec $r'>r$
				   convenable et on peut appliquer Rem. 16.7 de \cite{SchneiderNFA}.
          \item Dans le cas où $\GG$ est le complété formel d'un groupe algébrique $G$ réductif et déployé 
				  sur $S$, on peut montrer que l'anneau $D^{\dagger}(\GG)_\Qr$ est cohérent sans hypothèse sur l'indice 
				  de ramification de $V$.
\end{enumerate}
\vskip8pt

Soit $\GG^{rig}$ la fibre générique de
$\GG$ au sens usuel de Raynaud \cite{BoschLuetkebohmert}. Ainsi, $\GG^{rig}$ est un groupe affinoïde sur $K$ avec
$\GG^{\rm rig}(K)=G(V)$ comme groupe de points à valeurs dans $K$. 
L'inclusion naturelle $V\langle t_1,...,t_N\rangle \hookrightarrow V[[t_1,...,t_N]]$ induit une immersion ouverte 
$G^\circ\subseteq \GG^{\rm rig}$.
Cette immersion établit une bijection de $G^\circ (K)$ avec le sous-groupe des points de $G(V)$ qui se spécialisent en $1_\kappa\in G_\kappa$ (i.e. avec les $S$-morphismes
$f: S\rightarrow G$ qui ayant la propriété que $f_\kappa$ factorise via l'immersion $1_\kappa\hookrightarrow G_\kappa$).  On dit que
$G^\circ(K)$ est le \og premier groupe de congruence \fg \, de $G(V)$.
\vskip4pt

Par exemple, si $G={\rm GL}_n$ est le groupe linéaire général et
${\rm M}_n$ son algèbre de Lie, on a $G^\circ (L)=1+ {\rm M}_n(\mathfrak{m}_L)$
pour toute extension de corps valués complets $K\subseteq L$ où $\mathfrak{m}_L$ désigne l'idéal maximal de l'anneau de valuation de $L$. Donc, $G^\circ(K)=1+\pi {\rm M}_n(V)$ est le premier groupe de congruence de ${\rm GL}_n(V)$ au sens usuel.

\vskip4pt



\subsection{Représentations analytiques rigides}
\label{rep_an_rig}

Nous utilisons les notations de la section précédente et supposons encore que le corps $K$ est une extension finie de $\Qr_p$.
Nous donnons ici une interprétation de la catégorie des ${D}^{an}(G^\circ)$-modules de présentation finie en termes de certaines représentations $\pi$-adiques du groupe $\pi$-adique
$G^\circ(K)$, (groupe des points de $G^\circ$ à valeurs dans $K$). Rappelons comment est définie la topologie
$\pi$-adique sur $G(V)$. L'espace $V^{n^2}$ est muni de la topologie produit, qui induit la topologie produit sur
l'ouvert ${\rm GL}_n(V)\subset V^{n^2}$. La topologie $\pi$-adique sur $G(V)$ est
définie comme la topologie induite
relative à une immersion fermée $G\hookrightarrow {\rm GL}_n$ quelconque. Cette topologie est
plus fine que la topologie de Zariski et elle est localement compacte parce que $V$ est localement compact, e.g. 0.6 de \cite{Landvogt}.
Finalement, $G^\circ(K)\subset G(V)$ ainsi muni de la topologie induit.

\vskip4pt

Remarquons que l'involution $\tau: g\mapsto g^{-1}$ sur $G^\circ$ s'etend par fonctorialité en un anti-automorphisme de $D^{an}(G^\circ)$. On a donc une équivalence entre les
${D}^{an}(G^\circ)$-modules à gauche et à droite et on va considérer seulement des modules à gauche dans la suite. Le groupe $G$ agit à gauche sur lui-même par
conjugaison (i.e. $g$ agit par $h\mapsto g^{-1}hg$) et par fonctorialité sur $G^\circ$ et $D^{an}(G^\circ)$. Nous notons cette action \og adjointe \fg par
$g\mapsto Ad(g)$. Nous appelons {\it $({D}^{an}(G^\circ),G(V))$-module de présentation finie} un ${D}^{an}(G^\circ)$-module $M$ de présentation finie qui est
aussi muni d'une action $K$-linéaire $\rho$ du groupe $G(V)$, qui est compatible au sens que
pour tous $x\in D^{an}(G^\circ), g\in G(V), m\in M$,
$$x \rho(g)m= \rho(g) (Ad(g)x)m.$$

 \vskip5pt
 Du côté des représentations, pour un sous-groupe ouvert $H$ de $G(V)$ et un $K$-espace vectoriel topologique $W$, nous appelons {\it action topologique} de $H$ sur $W$ une action du groupe $H$ sur $W$ par des applications linéaires et continue, cf. (0.11) de \cite{Emerton}.
 En particulier, pour chaque $w\in W$ on a une application $o_w: H\rightarrow W$ définie par $h\mapsto hw.$
  Suivant 2.1.18/19 de \cite{Emerton}, nous introduisons l'espace des fonctions analytiques rigides sur $G^\circ$ en valeurs dans un espace Fréchet $W$ quelconque,
$$C^{\rm an}(G^\circ,W):=\OO(G^\circ)\hat{\otimes}_K W.$$ Ici, on prend comme produit tensoriel topologique le produit tensoriel projectif, cf. 17.B de \cite{SchneiderNFA}. Par évaluation aux points de $G^\circ(K)\subset G^\circ$, on obtient une inclusion de $C^{\rm an}(G^\circ,W)$ dans l'espace des fonctions (continues)
$G^\circ(K)\rightarrow W$, cf. 2.1.20 de \cite{Emerton}. En fait, comme $G^\circ$ est isomorphe à un polydisque ouvert, l'ensemble $G^\circ(K)$ est Zariski-dense
dans $G^\circ$. Nous avons la définition suivante, cf. Théorème 3.4.3 et Définition
3.6.1 de \cite{Emerton}.
\begin{defi} Soit $W$ un espace nucléaire Fréchet avec une action topologique de $G^\circ(K)$ où $G(V)$.
On appelle $W$ une représentation {\it $G^\circ$-analytique} (ou simplement: {\it analytique}) si, pour tout $w\in W$,
la fonction $o_w: G^\circ(K)\rightarrow W$ est dans l'espace $C^{\rm an}(G^\circ,W)$.
\end{defi}
Soit $W$ une représentation analytique de $G^\circ(K)$. Nous notons $W':={\rm Hom}_K^{\rm ct}(W,K)$ le dual continu de $W$. En appliquant \cite{Emerton}, Cor. 5.1.8 à un recouvrement affinoide convenable de $G^\circ$ on voit que $W'$
est un module sur l'anneau des distributions $D^{an}(G^\circ)$ de la manière suivante $$\langle\lambda(w'),w\rangle:= \lambda ( w'\circ o_w)$$ pour $w\in W,
w'\in W', \lambda\in D^{an}(G^\circ)$. Ici, $w'\circ o_w$ est vu comme élément de l'espace $C^{\rm an}(G^\circ,K)=\OO(G^\circ)$.
Supposons maintenant que l'anneau $D^{an}(G^\circ)$ est cohérent. C'est vrai dans la situation de 
\ref{thm-coherence}, et nous espérons que ce resultat reste vrai en général. En suivant la stratégie de Schneider et Teitelbaum de
\cite{SchnTei03} nous disons que la représentation analytique $W$ est {\it admissible}, si $W'$ est de présentation finie comme $D^{an}(G^\circ)$-module. On dit qu'une
représentation analytique de $G(V)$ est {\it admissible}, si elle l'est comme représentation de $G^\circ(K)$. Les morphismes dans ces catégories sont
par définition les applications $K$-linéaires continues et équivariantes.
 \begin{prop}
 Le foncteur $W\mapsto W'$ donne une anti-équivalence de catégories entre des représentations analytiques admissibles de $G^\circ(K)$ et les ${D}^{an}(G^\circ)$-modules de présentation finie. Le foncteur induit une équivalence entre les sous-catégories des $G(V)$-représentations
analytiques admissibles et celle des $({D}^{an}(G^\circ),G(V))$-modules de présentation finie.
 \end{prop}
 \begin{proof}
 Le corps $K$ est complet relativement à une valuation discrète et il est donc sphériquement complet, cf. Lemma 1.6 de \cite{SchneiderNFA}.
 Le passage au dual fort est donc une anti-équivalence involutive entre les espaces nucléaires Fréchet et les espaces de type compact, cf. \cite{SchnTei03}, Cor.
1.4. Si $M$ est un ${D}^{an}(G^\circ)$-module de type fini, engendré par des
générateurs $x_1,...,x_n$, on pose $M_m:= \sum_i \what{D}^{(m)}(\GG)_\Qr x_i\subseteq M$. Comme $\what{D}^{(m)}(\GG)_\Qr$ est une algèbre de Banach noethérienne, $M_m$ a une unique structure de $\what{D}^{(m)}(\GG)_\Qr$-module de
Banach et l'inclusion $M_m\rightarrow M_{m'}$, pour $m'>m$, est continue et compact. Munissons $M=\varinjlim_{m} M_m$  de la
topologie limite inductive. Il est facile de voir que cette topologie ne dépend pas du choix des générateurs $x_i$ pour
$M$, et qu'elle fait de $M$ un
${D}^{an}(G^\circ)$-module séparément continu sur un espace de type compact. De plus, chaque application ${D}^{an}(G^\circ)$-linéaire entre deux modules $M,M'$ de type
fini est automatiquement continue. Soit maintenant $W:=M'_b$ le dual fort de $M$. Par les arguments dans Cor. 3.3 de \cite{SchnTei01} l'espace $W$ est un ${D}^{an}(G^\circ)$-module séparément continue via la structure contragredient $$\langle \lambda.w,x\rangle:=w(\lambda^\tau.x)$$ pour $w\in W,~x\in M$ et $\lambda\in {D}^{an}(G^\circ)$.
Utilisant l'inclusion naturelle $\delta: G^\circ(K)\hookrightarrow D^{an}(G^\circ)^\times$ on a une action topologique de $G^\circ(K)$ sur $W$ donné par
 $$\langle g.w,x\rangle:=w(\delta_{g^{-1}}.x)$$ pour $w\in W,~x\in M$. Par construction l'application $o_w$ s'etend à une application linéaire continue $ D^{an}(G^\circ)\rightarrow W$. L'isomorphisme canonique
 $$ C^{\rm an}(G^\circ, W)=\OO(G^\circ)\hat{\otimes}_K W\simeq {\rm Hom}_K^{\rm ct}(\OO(G^\circ)'_b,W),$$
 \cite{SchneiderNFA}, Cor. 18.8, implique que $o_w\in C^{\rm an}(G^\circ,W)$ pour $w\in W$. Cela montre que $W$ est une représentation analytique de
$G^\circ(K)$. Par réflexivité on a $W'=M$ et la correspondance $M\mapsto W$ induit donc un quasi-inverse pour le foncteur $W\mapsto W'$ considéré dans l'assertion. La proposition en résulte.
\end{proof}

\vskip4pt
Remarque: soit $Z_K$ le centre de l'algèbre enveloppante $U(Lie(G_K))$ de $Lie(G_K)$. Comme $G$ est connexe, l'action adjointe de $G$ sur $Lie(G_K)$ stabilise
les éléments de $Z_K$, cf. \cite{Demazure_gr_alg}, II.\S6.1.5. Si $\theta$ est un caractère de $Z_K$ à valeurs dans $K$, on voit que l'action $g\mapsto Ad(g)$ donne une action de $G$ sur l'anneau quotient
$$D^{an}(G^\circ)_{\theta}:=D^{an}(G^\circ)/(\ker\theta) D^{an}(G^\circ).$$
Finalement, $Lie(G_K)$ agit sur une représentation analytique $W$ par une formule analogue à (\ref{equ-actionLie}), cf. \cite{Emerton}, p. 69, et on voit que $W$ est à caractère infinitésimal $\theta$ si et seulement si le $D^{an}(G^\circ)$-module $W'$ est en fait un module sur $D^{an}(G^\circ)_{\theta}$. Dans cette situation, on a une version évidente de la proposition précedente pour des représentations ayant $\theta$ comme caractère infinitésimal.

On déduit immédiatement de cette proposition le 
\begin{cor} La catégorie des représentations analytiques admissibles de $G^\circ(K)$ est abélienne.
\end{cor}
\vskip4pt

Terminons par quelques examples. Considérons le groupe fini de type de Lie $G(k)$ où $k$ est le corps résiduel de $V$, dans ce cas 
 les $G(k)$-représentations de dimension finie sont analytiques admissibles (en fait, on a $G(k)\simeq G(V)/G^\circ(K)$). Comme premiers examples de dimension infinie sur $K$,
notons qu'il y a un foncteur de la catégorie des $U(Lie(G_K))$-modules de type fini $M$ vers celle des
 $D^{an}(G^\circ)$-modules de présentation finie, qui à $M$ associe $M\mapsto D^{an}(G^\circ)\otimes_{U(Lie(G_K))} M$. 
Ce foncteur est donc à valeurs dans la
catégorie des représentations analytiques admissibles. Pour des représentations algébriques de dimension finie du schéma en groupes $G$, ce foncteur est simplement la restriction aux points rationels $G(V)\subset G$.


%

\noindent Christine Huyghe \\
\noindent IRMA, Universit\'e de Strasbourg \\
\noindent 7, rue René Descartes \\
\noindent 67084 Strasbourg cedex FRANCE \\
\noindent m\'el huyghe@math.unistra.fr\\
http://www-irma.u-strasbg.fr/\textasciitilde huyghe

\vspace{+8mm}
\noindent Tobias Schmidt \\
\noindent IRMAR, Universit\'e de Rennes 1\\
\noindent Campus de Beaulieu\\
\noindent 35042 Rennes cedex FRANCE \\
\noindent m\'el Tobias.Schmidt@univ-rennes1.fr


\end{document}